\documentclass[12pt, leqno]{amsart}
\usepackage{
amsmath,
amssymb,
amsthm,
mathrsfs,
amsfonts,
ytableau,
enumitem,
comment,
upgreek,
soul,
yhmath,
mathtools,
kotex,
stackengine,
caption,
tabularx,
stackrel,
amscd}
\usepackage{tikz,tikz-cd}
\usetikzlibrary{decorations.pathreplacing,decorations.pathmorphing}
\usetikzlibrary{arrows,arrows.meta}
\usetikzlibrary{shapes,shapes.geometric,shapes.symbols}
\usetikzlibrary{matrix,calc}
\usepackage[poly,all]{xy}
\usepackage{tabu}
\usepackage[colorinlistoftodos, textwidth = 2.3cm]{todonotes}
\usepackage{graphicx}
\usepackage{shuffle}
\usepackage{lipsum}
\usepackage{hyperref}
\hypersetup{
    colorlinks=true,
    linkcolor = blue,
    citecolor=magenta,
    urlcolor=cyan,
}
\usepackage[capitalise, nameinlink]{cleveref}
\crefname{equation}{}{}
\crefname{figure}{{\sc Figure}}{{\sc Figure}}
\crefname{section}{Section}{sections}

\ytableausetup{mathmode, boxsize=1em,aligntableaux=center}

\newtheorem{theorem}{Theorem}[section]
\newtheorem{proposition}[theorem]{Proposition}
\newtheorem{lemma}[theorem]{Lemma}
\newtheorem{corollary}[theorem]{Corollary}

\newtheorem*{claim*}{Claim}

\theoremstyle{definition}
\newtheorem{algorithm}[theorem]{Algorithm}
\newtheorem{example}[theorem]{Example}
\newtheorem{definition}[theorem]{Definition}
\newtheorem{remark}[theorem]{Remark}
\newtheorem*{roadmap}{The procedure of our method}

\numberwithin{equation}{section} \numberwithin{figure}{section}
\numberwithin{table}{section}

\def \hp {0.35}
\def \vp {0.45}
\def \ccc {3mm}
\def \lw {0.5mm}
\def \hhh{0.9em}
\def \vvv{0.9em}

\def \N {\mathbb N}
\def \Z {\mathbb Z}                           
\def \C {\mathbb C}
\def \Q {\mathbb Q}

\newcommand{\nc}{\newcommand}
\nc{\ra}{\rightarrow}
\nc{\la}{\leftarrow}
\nc{\SG}{\mathfrak{S}}
\nc{\PCT}{\mathrm{PCT}}
\nc{\SPCT}{\mathrm{SPCT}}
\nc{\SPYCT}{\mathrm{SPYCT}}
\nc{\RT}{\mathrm{RT}}
\nc{\SRT}{\mathrm{SRT}}
\nc{\RCT}{\mathrm{RCT}}
\nc{\SRCT}{\mathrm{SRCT}}
\nc{\SYCT}{\mathrm{SYCT}}
\nc{\SRET}{\mathrm{SRET}}
\nc{\SPYCTsa}[2]{\mathrm{SPYCT}^{#1}(#2)}
\nc{\DIF}[1]{\mathfrak{S}^*_{#1}}
\nc{\mDIF}[1]{\mathcal{V}_{#1}}
\nc{\RDIF}[1]{\mathcal{R}\mathfrak{S}^*_{#1}}
\nc{\mRDIF}[1]{\mathcal{R}\mathcal{V}_{#1}}
\nc{\ESF}[1]{\mathcal{E}_{#1}}
\nc{\mESF}[1]{X_{#1}}
\nc{\mRESF}[1]{\mathcal{R}X_{#1}}
\nc{\RESF}[1]{\mathcal{R}\mathcal{E}_{#1}}
\nc{\YQS}[1]{\widehat{\mathcal{S}}_{#1}}
\nc{\mYQS}[1]{\widehat{\mathbf{S}}_{#1}}
\nc{\QS}[1]{\mathcal{S}_{#1}}
\nc{\mQS}[1]{\mathbf{S}_{#1}}
\nc{\posetP}[1]{\mathcal{P}_{#1}}
\nc{\posetSEE}[1]{P_{\mQS{#1}}}
\nc{\PQS}[2]{\mathcal{S}^{#1}_{#2}}
\nc{\mPQS}[2]{\mathbf{S}^{#1}_{#2}}
\nc{\YRQS}[1]{\mathcal{R}\mathcal{S}_{#1}}
\nc{\mYRQS}[1]{\mathbf{R}\mathbf{S}_{#1}}
\nc{\RQS}[1]{\mathcal{R}_{#1}}
\nc{\mRQS}[1]{\mathbf{R}_{#1}}
\nc{\stan}{\mathrm{stan}}
\nc{\Span}{\mathrm{span}}
\nc{\comp}{\mathrm{comp}}
\nc{\rmst}{\mathrm{st}}
\nc{\Des}[2]{\mathrm{Des}_{#1}(#2)}
\nc{\des}{\mathrm{des}}
\nc{\NDes}{\mathrm{NDes}}
\nc{\ADes}[1]{\mathrm{ADes}(#1)}
\nc{\set}{\mathrm{set}}
\nc{\wt}{\mathrm{wt}}
\nc{\ch}{\mathrm{ch}}
\nc{\id}{\mathrm{id}}
\nc{\Sym}{\mathrm{Sym}}
\nc{\Qsym}{\mathrm{QSym}}
\nc{\Nsym}{\mathrm{NSym}}
\nc{\sh}{\mathrm{sh}}
\nc{\bfRa}[1]{\mathbf{R}_{#1}}
\nc{\bfB}{\mathbf{B}}
\nc{\bfi}[1]{\mathbf{i}_{#1}}
\nc{\bfj}{\mathbf{j}}
\nc{\bfs}{\mathbf{s}}
\nc{\bfx}{\mathbf{x}}
\nc{\bfy}{\mathbf{y}}
\nc{\bfc}{\mathbf{c}}
\nc{\bfm}{\mathbf{m}}
\nc{\bfM}{\mathbf{M}}
\nc{\hbfS}{\widehat{\mathbf{S}}}
\nc{\bfF}{\mathbf{F}}
\nc{\bfG}{\mathbf{G}}
\nc{\calS}{\mathcal{S}}
\nc{\calK}{\mathcal{K}}
\nc{\hcalS}{\widehat{\mathcal{S}}}
\nc{\alphamax}{\alpha_{\rm max}}
\nc{\brho}{\overline{\rho}}
\nc{\bphi}{\overline{\phi}}
\nc{\calV}{\mathcal{V}}
\nc{\calRV}{\mathcal{R}\mathcal{V}}
\nc{\RX}[1]{\mathcal{R}X_{#1}}
\nc{\calR}{\mathcal{R}}
\nc{\calG}{\mathcal{G}}
\nc{\calD}{\mathcal{D}}
\nc{\tal}{\lambda(\alpha)}
\nc{\tbe}{\widetilde{\beta}}
\nc{\opi}{\overline{\pi}}
\nc{\calP}{\mathcal{P}}
\nc{\rmtop}{\mathrm{top}}
\nc{\rad}{\mathrm{rad}}
\nc{\bfP}{\mathbf{P}}
\nc{\SET}{\mathrm{SET}}
\nc{\SIT}{\mathrm{SIT}}
\nc{\rev}{\mathrm{r}}
\nc{\rc}{\mathrm{t}}
\nc{\Th}{\theta}
\nc{\htau}{\widehat{\tau}}
\nc{\mPhi}{\Phi}
\nc{\mphi}{\phi}
\nc{\mPsi}{\Psi}
\nc{\hmPsi}{\widehat{\Psi}}
\nc{\mpsi}{\psi}
\nc{\mGam}{\Gamma}
\nc{\tcd}{\mathtt{cd}}
\nc{\diagramAE}{\mathtt{cd}}
\nc{\tcds}{\mathtt{shcd}_1}
\nc{\trd}{\mathtt{rd}}
\nc{\trcd}{\mathtt{rcd}}
\nc{\rmr}{\mathrm{r}}
\nc{\rmc}{\mathrm{c}}
\nc{\rmt}{\mathrm{t}}
\nc{\bubact}{\,\scalebox{0.6}{$\bullet$}\,}
\nc{\hbubact}{\,\scalebox{0.6}{$\widehat{\bullet}$}\,}
\nc{\col}{\rm col}
\nc{\row}{\rm row}
\nc{\calE}{\mathcal{E}}
\nc{\calA}{\mathcal{A}}
\nc{\calB}{\mathcal{B}}
\nc{\scrB}{\mathscr{B}}
\nc{\scroB}{\overline{\mathscr{B}}}
\nc{\scrP}[2]{\mathscr{P}_{#1}(#2)}
\nc{\scroP}[2]{{\overline{\mathscr{P}}}_{#1}(#2)}
\nc{\calC}{\mathcal{C}}
\nc{\calF}{\mathcal{F}}
\nc{\ocalF}{\overline{\calF}}
\nc{\calL}{\mathcal{L}}
\nc{\calT}{\mathscr{T}}
\nc{\sfB}{\mathsf{B}}
\nc{\sfb}{\mathsf{b}}
\nc{\sft}{\mathsf{t}}
\nc{\sfF}{\mathsf{F}}
\nc{\sfSF}{\mathsf{SF}}
\nc{\sfST}[1]{\mathsf{ST}(#1)}
\nc{\sfD}{\mathsf{D}}
\nc{\sfTab}[1]{\mathsf{Tab}(#1)}
\nc{\sfLL}{\mathsf{lp}}
\nc{\sfip}{\mathsf{ip}}
\nc{\sfT}{\mathsf{T}}
\nc{\sfL}{\mathsf{L}}
\nc{\sfM}{\mathsf{M}}
\nc{\sfpr}{\mathsf{pr}}
\nc{\calEsa}{\mathcal{E}^\sigma(\alpha)}
\nc{\setIndSummandY}{\mathsf{ind}(\mathbf{Y}_\alpha)}
\nc{\revSum}[2]{o_{#1}(#2)}
\nc{\tauC}{\tau_{\scalebox{0.5}{$C$}}}
\nc{\sytabC}{\sytab_{\scalebox{0.5}{$C$}}}
\nc{\bbfP}{\overline{\bfP}}
\nc{\pr}{\mathbf{pr}}
\nc{\Ups}{\Upsilon}
\nc{\pact}{\diamond}
\nc{\tauE}{\tau_{\scalebox{0.5}{$E$}}}
\nc{\tauEL}[1]{\ensuremath{\tau_{\scalebox{0.5}{$#1$}}}}
\nc{\taupE}{\ensuremath{\tau'_{\scalebox{0.5}{$E$}}}}
\nc{\fkTSk}[1]{\mathfrak{T}^{\la}_{#1}}
\nc{\tauF}{\tau_{\scalebox{0.5}{$F$}}}
\nc{\tauG}{\tau_{\scalebox{0.5}{$G$}}}
\nc{\rtE}{T_{\scalebox{0.5}{$E$}}}
\nc{\rtF}{T_{\scalebox{0.5}{$F$}}}
\nc{\rtG}{T_{\scalebox{0.5}{$G$}}}
\nc{\oPaE}{\overline{\Phi}_{\alpha_E}}
\nc{\oPaF}{\overline{\Phi}_{\alpha_F}}
\nc{\oPaG}{\overline{\Phi}_{\alpha_G}}
\nc{\tab}{\tau}
\nc{\sytab}{\widehat{\tau}}
\nc{\hatE}{\widehat{E}}
\nc{\hcalE}{\widehat{\calE}}
\nc{\hatC}{\widehat{C}}
\nc{\bal}{{\boldsymbol{\upalpha}}}
\nc{\bbe}{{\boldsymbol{\upbeta}}}
\nc{\bgam}{{\boldsymbol{\upgamma}}}
\nc{\bdel}{{\boldsymbol{\updelta}}}
\nc{\weakcon}{\odot}
\nc{\calM}{\mathcal{M}}
\nc{\calN}{\mathcal{N}}

\nc{\ldalpha}{\lambda(\alpha)}
\nc{\SRIT}{\mathrm{SRIT}}
\nc{\re}{\mathrm{rev}}
\nc{\otau}{\overline{\tau}}
\nc{\rtop}{{\rm top}}
\nc{\sfc}{\mathsf{c}}
\nc{\sfr}{\mathsf{r}}
\nc{\JC}{\mathtt{JC}}
\nc{\tH}{\mathtt{H}}
\nc{\tS}{\mathtt{S}}
\nc{\tV}{\mathtt{V}}
\nc{\TcalTt}{{T^{\calT}}}
\nc{\fkT}{\mathfrak{T}}
\nc{\TfkTt}{{T^{\fkT}}}
\nc{\calTuT}{\calT^T}
\nc{\projco}{\Phi}
\nc{\indexK}[2]{\calK_{#1}(#2)}
\nc{\soc}{\mathrm{soc}}
\nc{\longelto}[1]{w_0(#1)}
\nc{\longeltt}[2]{\boldsymbol{w_0}(#1;#2)}
\nc{\tJ}[2]{{\tt J}_{#1;#2}}
\nc{\TLba}[2]{{T^\la_{#1;#2}}}
\nc{\autotheta}{\uptheta}
\nc{\calW}{\mathcal{W}}
\nc{\autophi}{\upphi}
\nc{\autochi}{\upchi}
\nc{\autoomega}{\upomega}
\nc{\hautophi}{{\widehat{\autophi}}}
\nc{\hautotheta}{{\widehat{\autotheta}}}
\nc{\hautoomega}{{\widehat{\autoomega}}}
\nc{\bfpi}{\boldsymbol{\uppi}}
\nc{\osfB}{\overline{\sfB}}
\nc{\ourMP}[1]{\mathbf{M}_{#1}}
\nc{\ourPoset}[1]{P_{#1}}
\nc{\DHTMP}[1]{M_{#1}}
\nc{\WBIM}[2]{\mathsf{B}(#1,#2)}
\nc{\setInt}[1]{\mathsf{Int}(#1)}
\nc{\poset}[1]{\mathsf{Poset}(#1)}
\nc{\modR}{\text{\bf mod-}H_n(0)}
\nc{\modA}{\text{\bf mod-}A}
\nc{\modRb}{\text{\bf mod-}H_\bullet(0)}
\nc{\Rmod}{H_n(0)\text{\bf -mod}}
\nc{\Rbmod}{H_\bullet(0)\text{\bf -mod}}
\nc{\rmread}{\mathsf{read}}
\nc{\rbread}{\mathrm{read}}
\nc{\rmIm}{\mathrm{Im}}
\nc{\Hom}{\mathrm{Hom}}
\nc{\SPCTsa}{\SPCT^\upsig(\alpha)}
\nc{\fkp}{\mathfrak{p}}
\nc{\bfR}{\mathbf{R}}
\nc{\SYRT}{\mathrm{SYRT}}
\nc{\ova}{\overline{\alpha}}
\nc{\cocover}{\mathsf{cocover}}
\nc{\tPhi}{\widetilde{\mPhi}}
\nc{\sfem}{\mathsf{em}}
\nc{\upsig}{{\boldsymbol{\upsigma}}}
\nc{\rmw}{\mathrm{w}}
\nc{\tread}[1]{\underline{\mathsf{read}}(#1)}
\nc{\Yread}[1]{\mathtt{read}(#1)}
\nc{\sfLS}[2]{\mathsf{LS}(#1;#2)}
\nc{\sfUS}[2]{\mathsf{US}(#1;#2)}
\nc{\setAlphaE}{(\alpha^-,E^-)}
\nc{\balalp}{\underline{{\boldsymbol{\upalpha}}}}
\nc{\balalpp}[1]{\underline{{\boldsymbol{\upalpha}}}_{(#1)}}
\nc{\SGR}[1]{\ensuremath{\Sigma_R(#1)}}
\nc{\SGL}[1]{\ensuremath{\Sigma_L(#1)}}
\nc{\inc}[1]{\mathsf{inc}(#1)}
\nc{\BST}[1]{\mathrm{BST}(#1)}
\nc{\BSTx}[2]{\mathrm{BST}_{#1}(#2)}
\nc{\Kpw}{K_{(P,\omega)}(x)}
\nc{\Kp}{K_{P}(x)}
\nc{\simFR}[1]{\bfF^R_{#1}}
\nc{\simFL}[1]{\bfF^L_{#1}}
\nc{\bfwD}[1]{\mathbf{w}_{#1}}
\nc{\bfwPD}[1]{\mathbf{v}_{#1}}
\nc{\setA}[2]{A_{#1;#2}}
\nc{\setAA}[2]{\widetilde{A}_{#1;#2}}
\nc{\upineq}{\rotatebox{90}{$<$}}
\nc{\downineq}{\rotatebox{270}{$<$}}
\nc{\diagineq}{\rotatebox{135}{$<$}}
\nc{\bfw}{\mathbf{w}}
\nc{\bfv}{\mathbf{v}}
\nc{\frakS}{\mathfrak{S}}
\nc{\frakmk}{\mathfrak{m}_k}
\nc{\frakmone}{\mathfrak{m}_1}
\nc{\valp}[1]{{#1}_{\scalebox{0.4}{[V]}}}
\nc{\xalp}[1]{{#1}_{\scalebox{0.4}{[X]}}}
\nc{\salp}[1]{{#1}_{\scalebox{0.4}{[S]}}}
\nc{\sourceE}[1]{E_{#1}}
\nc{\sourceEE}[1]{E^{\rightarrow}_{#1}}
\nc{\sinkE}[1]{E^{\la}_{#1}}
\nc{\setS}[1]{\mathcal{E}_{#1}}
\nc{\setSE}[1]{\mathcal{E}^{\rightarrow}_{#1}}
\nc{\setSSE}[1]{\mathcal{E}^{\la}_{#1}}
\nc{\risefree}[1]{{\rm RiseFree}(#1)}
\nc{\fallfree}[1]{{\rm FallFree}(#1)}
\nc{\typeA}{{$\wedge$-regular}}
\nc{\typeAA}{\overline{41}\,\overline{23}}
\nc{\typeAB}{\overline{14}\,\overline{32}}
\nc{\typeB}{{$\vee$-regular}}
\nc{\sourcetauE}[1]{\ensuremath{\tau}^\rightarrow_{#1}}
\nc{\sourcehtauE}[1]{\hat{\tau}^\rightarrow_{#1}}
\nc{\sourceSIT}[1]{\calT_{#1}}
\nc{\sinkSIT}[1]{\calT'_{#1}}
\nc{\sourceSET}[1]{\sfT_{#1}}
\nc{\sinkSET}[1]{\sfT'_{#1}}
\nc{\sinktauE}[1]{\ensuremath{\tau^\la_{#1}}}
\nc{\partialSinkE}{\tau_{E[-1],1}}
\nc{\bfRaG}{\mathbf{R}_{\alpha,G}}
\nc{\sourceTD}[1]{\ensuremath{T^\rightarrow_{#1}}}
\nc{\sinkTD}[1]{T^\uparrow_{#1}}
\nc{\sfFD}{{\sf F}_D}
\nc{\colindexR}[1]{\mathbf{c}(#1)}
\nc{\colindexL}[1]{\mathbf{c}_{\rm L}(#1)}
\nc{\rowindexH}[1]{\mathbf{r}(#1)}
\nc{\tHindex}[1]{\mathsf{h}(#1)}
\nc{\removablenodes}{\mathsf{RN}(\alpha)}
\nc{\sign}[1]{{\rm sign}(#1)}
\nc{\setLE}[1]{{L}(#1)}
\nc{\setNLP}[1]{{NL}(#1)}
\nc{\PI}{\mathtt{P}_n(I)}

\nc{\Puv}[2]{P_{(#1,#2)}}
\nc{\inv}[1]{\mathrm{Inv}(#1)}
\nc{\invL}[1]{\mathrm{Inv}_L(#1)}
\nc{\coinv}[1]{\mathrm{Coinv}(#1)}
\nc{\coinvL}[1]{\mathrm{Coinv}_L(#1)}
\nc{\minp}[1]{{#1}_{\rm min}}
\nc{\htt}[2]{\mathrm{ht}_{#1}(#2)}
\nc{\height}{\mathrm{ht}}
\nc{\pomega}{(P,\omega)}
\nc{\KPw}{K_{(P,\omega)}}
\nc{\KP}{K_{P}}
\nc{\balpha}{Y_\alpha}
\nc{\mbalpha}{\mathbf{Y}_\alpha}
\nc{\posetSource}[1]{{#1}^\rightarrow}
\nc{\posetSink}[1]{{#1}^\la}
\nc{\readingSourceP}[1]{{\tt read}_{\rm BTLR}(#1)}
\nc{\readingSinkP}[1]{{\tt read}_{\rm LRBT}(#1)}
\nc{\readingUL}[1]{{\tt read}^{\rm rev}_{\mathsf{ST}}(#1)}
\nc{\readingLR}[1]{{\tt read}_{\mathsf{ST}}(#1)}
\nc{\readingLL}[1]{{\tt read}_{\rm BTLR}(#1)}
\nc{\readingLLR}[1]{{\tt read}_{\rm LRBT}(#1)}
\nc{\positionX}[1]{x_{#1}}
\nc{\positionY}[1]{y_{#1}}
\nc{\labelSE}{\omega_{E}}
\nc{\posethSE}{P_{F}}
\nc{\posetRSG}{P_{G_*}}
\nc{\labelRSG}{\omega_{G_*}}
\nc{\posetDI}[1]{P_{\calV_{#1}}}
\nc{\posetRDI}[1]{P_{\calR\calV_{#1}}}
\nc{\posetX}[1]{P_{X_{#1}}}
\nc{\posetRX}[1]{P_{\mathcal{R}X_{#1}}}
\nc{\classQS}[1]{\mathcal{E}(#1)}
\nc{\classRQS}[1]{\mathbf{E}(#1)}
\nc{\classYQS}[1]{\widehat{\mathcal{E}}(#1)}
\nc{\classRS}[1]{\mathcal{E}_*(#1)}
\nc{\posetFromAlgo}[1]{P_{#1}}
\nc{\numberRN}[1]{{\sf U}(#1)}
\nc{\setR}[1]{R_{#1}}
\nc{\setC}[1]{{\sf C}_{#1}(E)}
\nc{\setV}[1]{{\sf V}_{#1}(E)}
\nc{\setPosetCons}{\mathfrak{P}(n)}
\nc{\wrestplus}[2]{{#1_{\geq #2}[+1]}}
\nc{\longelt}[1]{w_0^{(#1)}}
\nc{\subsetX}{D}
\nc{\posX}[2]{{\bf r}_{#1}(#2)}
\nc{\posY}[2]{{\bf c}_{#1}(#2)}
\nc{\mapf}{\boldsymbol{\Xi}}
\nc{\sinkAlgo}{\mathtt{s}(\tau)}
\nc{\sinkAlgotau}[1]{\mathtt{s}(#1)}
 
\newcommand*\circled[1]{\tikz[baseline=(char.base)]{
            \node[shape=circle,draw,inner sep=1.5pt] (char) {$#1$};}}

\title[Poset modules of the $0$-Hecke algebras]{Poset modules of the $0$-Hecke algebras and related quasisymmetric power sum expansions}

\author[S.-I. Choi]{Seung-Il Choi}
\address{Center for Quantum structures in Modules and Spaces, Seoul National University, Seoul 08826, Republic of Korea}
\email{ignatioschoi@snu.ac.kr}

\author[Y.-H. Kim]{Young-Hun Kim}
\address{Center for Quantum structures in Modules and Spaces, Seoul National University, Seoul 08826, Republic of Korea}
\email{ykim.math@gmail.com}

\author[Y.-T. Oh]{Young-Tak Oh}
\address{Department of Mathematics, Sogang University, Seoul 04107, Republic of Korea}
\email{ytoh@sogang.ac.kr}

\keywords{Partially ordered set (poset), $0$-Hecke algebra, Quasisymmetric function, Quasisymmetric power sum, Weak Bruhat order}

\date{\today}
\subjclass[2020]{20C08, 05E05, 05E10}

\begin{document}
\begin{abstract}
Duchamp--Hivert--Thibon introduced the construction of a right $H_n(0)$-module, denoted as $M_P$, for any partial order $P$ on the set $[n]$. This module is defined by specifying a suitable action of $H_n(0)$ on the set of linear extensions of $P$. In this paper, we refer to this module as the poset module associated with $P$.
Firstly, we show that $\bigoplus_{n \ge 0} G_0(\mathscr{P}(n))$ has a Hopf algebra structure that is isomorphic to the Hopf algebra of quasisymmetric functions, where $\mathscr{P}(n)$ is the full subcategory of $\textbf{mod-}H_n(0)$ whose objects are direct sums of finitely many isomorphic copies of poset modules and $G_0(\mathscr{P}(n))$ is the Grothendieck group of $\mathscr{P}(n)$. 
We also demonstrate how (anti-)automorphism twists interact with these modules, the induction product and restrictions. 
Secondly, we investigate the (type 1) quasisymmetric power sum expansion of some quasi-analogues $Y_\alpha$ of Schur functions, where $\alpha$ is a composition. 
We show that they can be expressed as the sum of the $P$-partition generating functions of specific posets, which allows us to utilize the result established by Liu--Weselcouch.
Additionally, we provide a new algorithm for obtaining these posets.
Using these findings, for the dual immaculate function and the extended Schur function, we express the coefficients appearing in the quasisymmetric power sum expansions in terms of border strip tableaux.
\end{abstract}

\maketitle

\tableofcontents

\section{Introduction}
The $0$-Hecke algebra $H_n(0)$ is obtained from the generic Hecke algebra $H_n(q)$ by specializing $q$ to $0$, but its representation theory is quite different from that of $H_n(q)$. For example, it is not representation-finite for $n \geq 4$ 
(see~\cite{11BG, 02DHT}).
So far, its representations have been constructed using a variety of combinatorial objects, including labeled posets, weak Bruhat intervals, ordered set partitions, and tableaux (see~\cite{02DHT,03Mc,22JKLO,18HR,15TW,15BBSSZ,19TW,21CKNO,22Searles,22BS,22NSvWVW}).
In this paper, we intensively study the $H_n(0)$-modules that arise from labeled posets. 

Let $P$ be a finite poset and $\omega$ a labeling of $P$.
In \cite{72Stan}, Stanley defined the $(P,\omega)$-partition generating function as 
\[
\KPw:=\sum_{\sigma \in \calA(P,\omega)}x_{\sigma(1)}x_{\sigma(2)}\cdots x_{\sigma(n)},
\]
where $\calA(P,\omega)$ is the set of all $\pomega$-partitions.
It was dealt with in more detail by Gessel~\cite{84Ges}. 
One notable feature of this function is that it appears as $F$-positive, 
that is, as a non-negative sum of fundamental quasisymmetric functions (see~\cref{Duchamp--Hivert--Thibon's poset modules}).
We mainly deal with the posets with underlying set $[n]:=\{1,2,\ldots,n\}$.
Such a poset $P$ can be naturally viewed as a labeled poset $(P, \omega)$, where the labeling $\omega: P \rightarrow [n]$ is given by $\omega(i) = i$ for $1 \leq i \leq n$.
Therefore, we simply write $K_P$ for $K_{\pomega}$.

On the other hand, it is well known that $F$-positive quasisymmetric functions play an important role in 
the representation theory of the 0-Hecke algebras of type $A$. 
Let $\Qsym$ be the algebra of quasisymmetric functions over $\Z$ and $G_0(\modR)$ the Grothendieck group of the category of finitely generated right $H_n(0)$-modules.
It was shown in~\cite{96DKLT, 97KT} that the {\em quasisymmetric characteristic}
\begin{align*}
\ch : \bigoplus_{n \ge 0}G_0(\modR) \ra \Qsym, \quad [\simFR{\alpha}] \mapsto F_{\alpha}
\end{align*}
is a ring isomorphism when the domain is equipped with the multiplication induced from the induction product. 
Here, $\alpha$ is a composition of $n$, $\simFR{\alpha}$ is the irreducible right $H_n(0)$-module attached to $\alpha$
and $[\simFR{\alpha}]$ the equivalence class of $\simFR{\alpha}$ inside $G_0(\modR)$ (see~\cref{subsec: 0-Hecke alg and QSym}).
Later, it was shown in \cite{09BL} that the quasisymmetric characteristic $\ch$ is an isomorphism of graded Hopf algebras.

Suppose we have a family $\{f_j \mid j \in J\}$ of $F$-positive quasisymmetric functions, where $J$ is an index set. There have been many studies to construct a family $\{M_j \mid j \in J\}$ of $H_n(0)$-modules with the property $\ch([M_j])=f_j$ by endowing a suitable $H_n(0)$-action on an existing combinatorial model for $f_j$, rather than simply adding irreducible modules. Duchamp--Hivert--Thibon~\cite{02DHT} conducted such a study for the $P$-partition generating functions of posets on $[n]$.
Let $\poset{n}$ be the set of all posets on $[n]$. 
For each $P\in \poset{n}$, define $\SGR{P}$ to be the set of permutations $\sigma \in \mathfrak S_n$ such that $\sigma^{-1}(i)< \sigma^{-1}(j)$ whenever $i \prec_P j$.
They construct a right $H_n(0)$-module $\DHTMP{P}$ by defining the following $H_n(0)$-action on the $\C$-span of $\SGR{P}$:
for $1 \leq i \leq n-1$ and $\gamma \in \SGR{P}$,
\begin{align*}
\gamma \cdot \opi_i:= 
\begin{cases}
-\gamma & \text{if } i \in \Des{R}{\gamma},\\
0 & \text{if $i \notin \Des{R}{\gamma}$ and $\gamma s_i \notin \SGR{P}$},\\
\gamma s_i & \text{if } i \notin \Des{R}{\gamma} \text{ and } \gamma s_i \in \SGR{P}.
\end{cases}
\end{align*}
Here, $\{\opi_i \mid 1 \leq i \leq n-1\}$ is a set of generators of $H_n(0)$ (see \cref{subsec: 0-Hecke alg and QSym}).
And, $\Des{R}{\gamma} = \{1 \leq i \leq n-1 \mid \ell(\gamma s_i) < \ell(\gamma)\}$, $s_i=(i \ \, i+1)$ is a simple transposition, and $\ell(\gamma)$ is the length of $\gamma$. 
In the present paper, we call it the \emph{poset module associated with $P$}.
This module has the property $\ch([\DHTMP{P}])=K_P$, which plays an important role in our arguments.

In \cref{Sec: poset module category}, we study poset module categories from the viewpoint of the Grothendieck group. 
For each nonnegative integer $n$,
let $\modR$ be the category of right $H_n(0)$-modules and $\scrP{}{n}$ the full subcategory of $\modR$ whose objects are direct-sums of finitely many isomorphic copies of poset modules. 
For convenience, although not entirely accurate, we will refer to $\scrP{}{n}$ as the \emph{poset module category of rank $n$}.
In \cref{Sec: The Hopf algebra structure}, we show that $\bigoplus_{n \ge 0} G_0(\scrP{}{n})$, with the product and coproduct inherited from the induced product and restriction, has a Hopf algebraic structure that is isomorphic to $\Qsym$ (\cref{Hopf algebra structure of Grothendieck group}).
The crucial step in doing this is to ensure that the restriction of a poset module appears as a direct sum of tensor products of poset modules, since it was already been proven in~\cite{02DHT} that the induction product of poset modules is again a poset module (\cref{Lem: restriction}).
In~\cref{Sec: automorphism twists}, we demonstrate how (anti-)automorphism twists interact with the poset modules.
For instance, for the (anti-)automorphism twists $\upphi[\DHTMP{P}]$, $\uptheta[\DHTMP{P}]$, and $\upchi[\DHTMP{P}]$, we have the right $H_n(0)$-module isomorphisms
\begin{align*}
\upphi[\DHTMP{P}] \cong   \DHTMP{\overline{P}^*},
\quad
\uptheta[\DHTMP{P}] \cong  \ourMP{P},
\quad \text{and} \quad
\upchi[\DHTMP{P}] \cong \ourMP{\overline{P}},
\end{align*}
where $\upphi,\uptheta$, and $\upchi$ are (anti-)involutions of $H_n(0)$ due to Fayers~\cite{05Fayers}, ${}^-$ and ${}^*$ are certain poset involutions on $\poset{n}$, and $\ourMP{P}$ is the right $H_n(0)$-module whose underlying space is  $\C\SGR{P}$ and whose $H_n(0)$-action is given by
\begin{align*}
\gamma \cdot \pi_i:= 
\begin{cases}
\gamma & \text{if $i \in \Des{R}{\gamma}$},\\
0 & \text{if $i \notin \Des{R}{\gamma}$ and $\gamma s_i \notin \SGR{P}$},\\
\gamma s_i & \text{if $i \notin \Des{R}{\gamma}$  and $\gamma s_i \in \SGR{P}$}
\end{cases}
\end{align*}
for $i \in [n-1]$ and $\gamma \in \SGR{P}$
(\cref{thm: auto-twists}). 
The interactions of (anti-)automorphism twists with the induction product and restrictions for poset modules are also dealt with (\cref{Coro: Automorphis of P}).

In \cref{Sec: Quasisymmetric power sum expansions I}, we investigate the expansion of some important bases for $\Qsym$ in terms of the (type 1) quasisymmetric power sum basis introduced by Ballantine--Daugherty--Hicks--Mason--Niese~\cite{20BDHMN}.
The bases we are working with, denoted as $\{\balpha \mid \text{$\alpha$ is a composition}\}$, were introduced as quasisymmetric analogues of the Schur basis
(see~\cref{Table: our consideration}).
Recently, Liu--Weselcouch~\cite{21LW} provided a combinatorial formula for the quasisymmetric power sum expansion of $K_{\pomega}$ for a labeled poset $\pomega$.
Such expansions have previously been considered in the naturally labeled case by Alexandersson--Sulzgruber \cite{21AS}.
For every composition $\alpha$ of $n$, we show that each $\balpha$ appears as the sum of the $P$-partition generating functions of some posets in $\poset{n}$, say $\balpha=\sum_{P}K_{P}$, which allows us to use Liu--Weselcouch's result.
In doing this, the results of Jung--Kim--Lee--Oh (\cite[Section 4]{22JKLO}) and Bj\"{o}rner--Wachs (\cite[Theorem 6.8]{91BW}) play a key role, where the former concerns left weak Bruhat interval modules and the latter characterizes the posets $P \in \poset{n}$ such that $\SGR{P}$ is a right weak Bruhat interval.
The overall process for deriving our result is outlined step by step in \cref{The bases for Qsym in consideration and our method}, and detailed descriptions of each step are provided throughout the rest of \cref{Sec: Quasisymmetric power sum expansions I}.

In \cref{Sec: Quasisymmetric power sum expansions II}, we study the posets $P$ that arise from the equality $\balpha=\sum_{P}K_{P}$, which was obtained in \cref{Sec: Quasisymmetric power sum expansions I}.
Let $\mbalpha$ be an arbitrary left $H_n(0)$-module from the list given in \cref{Table: our consideration}, and let $\setIndSummandY$ be the set of indecomposable direct summands of $\mathbf{\alpha}$.
For each $V\in \setIndSummandY$, we show in \cref{Subsec: pomega-partitions} that there is a right Bruhat interval $I_V$ such that $\DHTMP{P_{I_V}} \cong \calF_n(V)$ (as right $H_n(0)$-modules).
Here, $P_{I_V}$ is the regular poset corresponding to $I_V$ under Bj\"{o}rner--Wachs's algorithm
and $\calF_n$ is the quasisymmetric characteristic preserving contravariant functor from $\Rmod$ to $\modR$ given in \cref{Step 3 and Step 4}.
Since $\balpha=\ch([\mbalpha])$, it follows that
\[
\balpha=\sum_{V\in \setIndSummandY}K_{P_{I_V}}.
\]
In \cref{Sec51: canonical poset}, given any permutation $\rho \in \SG_n$ satisfying $w_0(\alpha) \preceq_L \rho$, we construct a poset $P$ such that $\SGR{P} = [w_0 \rho^{-1},w_0 \, w_0(\alpha)^{-1}]_R$ (\cref{Thm: Diagram Dalpharho}).
In \cref{Sec: The poset description of indecomposable direct summands}, explicit algorithms for obtaining $P_{I_V}$'s are given
(\cref{Thm: posets for V and X} and \cref{explicit description for S}).
Finally, in \cref{Tableau descriptions of the coefficients in the expansions},
we describe the coefficients appearing in the quasisymmetric power sum expansions of the dual immaculate function $\DIF{\alpha}$ and the extended Schur function $\ESF{\alpha}$
in terms of modified border strip tableaux (for details, see \cref{def of border strip tableaux} and \cref{tableau description}).

\section{Preliminaries}\label{Sec: Preliminaries}
Denote by $\N$ (resp. $\N_0$) the set of positive (resp. nonnegative) integers.
In this section, let $n \in \N_0$.
Define $[n]$ to be $\{1,2,\ldots, n\}$ if $n > 0$ or $\emptyset$ otherwise.
In addition, set $[-1]$ to be $\emptyset$.

A \emph{composition} $\alpha$ of a nonnegative integer $n$, denoted by $\alpha \models n$, is a finite ordered list of positive integers $(\alpha_1, \alpha_2, \ldots, \alpha_k)$ satisfying $\sum_{i=1}^k \alpha_i = n$.
We call $k$ the \emph{length} of $\alpha$ and denote it by $\ell(\alpha)$.
For convenience, we define the empty composition $\emptyset$ to be the unique composition of size and length $0$.
If $\alpha_1 \ge \alpha_2 \ge \cdots \ge \alpha_{\ell(\alpha)}$, then we say that $\alpha$ is a \emph{partition} of $n$ and denote this by $\alpha \vdash n$. 
The partition obtained by sorting the parts of $\alpha$ in the weakly decreasing order is denoted by $\widetilde{\alpha}$.

Given $\alpha = (\alpha_1, \alpha_2, \ldots,\alpha_{\ell(\alpha)}) \models n$ and $I = \{i_1 < i_2 < \cdots < i_p\} \subset [n-1]$,
let
\begin{align*}
&\set(\alpha) := \{\alpha_1,\alpha_1+\alpha_2,\ldots, \alpha_1 + \alpha_2 + \cdots + \alpha_{\ell(\alpha)-1}\}, \\
&\comp(I) := (i_1,i_2 - i_1,\ldots,n-i_p).
\end{align*}
The set of compositions of $n$ is in bijection with the set of subsets of $[n-1]$ under the correspondence $\alpha \mapsto \set(\alpha)$ (or $I \mapsto \comp(I)$).
For $\alpha, \beta \models n$, if $\set(\beta) \subseteq \set(\alpha)$, then we say that \emph{$\alpha$ refines $\beta$} and denote $\alpha \preceq \beta$.

For $\alpha = (\alpha_1,\ldots,\alpha_{\ell(\alpha)}) \models n$, let $\alpha^\rmr$ be the composition $(\alpha_{\ell(\alpha)}, \ldots, \alpha_1)$ and $\alpha^\rmc$ the composition satisfying $\set(\alpha^\rmc) = [n-1] \setminus \set(\alpha)$.
And let 
$\alphamax$ be the largest part of $\alpha$. 

\subsection{Quasisymmetric functions}
\label{preliminaries on Quasisymmetric functions}
Quasisymmetric functions are power series of bounded degree in variables $x_{1},x_{2},\ldots$  with coefficients in $\Z$, which are shift invariant in the sense that the coefficient of the monomial $x_{1}^{\alpha_{1}}x_{2}^{\alpha_{2}}\cdots x_{k}^{\alpha_{k}}$ is equal to the coefficient of the monomial $x_{i_{1}}^{\alpha_{1}}x_{i_{2}}^{\alpha_{2}}\cdots x_{i_{k}}^{\alpha _{k}}$ for any strictly increasing sequence of positive integers $i_{1} < i_{2} < \cdots < i_{k}$ indexing the variables and any positive integer sequence $(\alpha_{1},\alpha_{2},\ldots,\alpha_{k})$ of exponents.

The ring $\Qsym$ of quasisymmetric functions is a graded $\Z$-algebra, decomposing as
\[
\Qsym =\bigoplus _{n\geq 0} \Qsym_n,
\]
where $\Qsym_n$ is the $\Z$-module consisting of all quasisymmetric functions that are homogeneous of degree $n$.

Given a composition $\alpha = (\alpha_1, \alpha_2, \ldots, \alpha_k)$ of $n$, the \emph{monomial quasisymmetric function} $M_\alpha$ is defined by 
\[
M_\alpha = \sum_{i_1 < i_2 < \cdots < i_k} x_{i_1}^{\alpha_1} x_{i_2}^{\alpha_2}
\cdots
x_{i_k}^{\alpha_k}
\]
and the \emph{fundamental quasisymmetric function} $F_\alpha$ is defined by 
\[
F_\alpha = \sum_{\substack{1 \le i_1 \le i_2 \le \cdots \le i_n \\ i_j < i_{j+1} \text{ if } j \in \set(\alpha)}} x_{i_1} x_{i_2} \cdots x_{i_n}.
\]
For every nonnegative integer $n$, it is known that both $\{F_\alpha \mid \alpha \models n\}$ and $\{M_\alpha \mid \alpha \models n\}$ are bases for $\Qsym_n$.

For $\alpha, \beta \models n$ with $\beta \succeq \alpha$, let $\alpha^{(i)}$ be the composition consisting of the parts of $\alpha$ that combine to form $\beta_i$, so $\alpha^{(i)}\models \beta_i$. 
Let 
\[
\pi(\alpha):=\prod_{j=1}^{\ell(\alpha)}\sum_{k=1}^{j}\alpha_k
\quad \text{and} \quad
\pi(\alpha, \beta):=\prod_{i=1}^{\ell(\beta)}\pi(\alpha^{(i)}).
\]
With these notations, the type 1 quasisymmetric power sum $\Psi_\alpha$ is defined by
\[
z_\alpha \sum_{\beta \succeq \alpha} \frac{1}{\pi(\alpha,\beta)}M_\beta,
\]
where $z_\alpha = \prod_{i \geq 1} i^{m_i} \cdot m_i !$ and $m_i$ is the number of parts equal to $i$ in $\alpha$.
For example, let $\alpha = (2,3,1) \models 6$.
Then $\beta \succeq \alpha$ for $\beta \in \{(2,3,1),(5,1), (2,4), (6)\}$.
One sees that $z_\alpha = 6$,
\begin{align*}
&\pi(\alpha,(2,3,1)) = 2 \cdot 3 \cdot 1 = 6,\\
&\pi(\alpha,(5,1)) = (2 \cdot (2+3)) \cdot 1 = 10,\\
&\pi(\alpha,(2,4)) = 2 \cdot (3 \cdot (3+1)) = 24 , \text{ and}\\
&\pi(\alpha,(6)) = 2 \cdot (2 + 3) \cdot (2 + 3 + 1) = 60.
\end{align*}
Hence, one has 
\[
\Psi_\alpha 
=  M_{(2,3,1)} + \frac{3}{5}M_{(5,1)} + \frac{1}{4}M_{(2,4)} + \frac{1}{10}M_{(6)}.
\]
It is known that $\Psi_\alpha$'s form a basis for $\Q \otimes_\Z \Qsym$
(for more details, see~\cite{20BDHMN}).
From now on, we will omit `type 1' and 
simply refer to $\Psi_\alpha$ as the \emph{quasisymmetric power sum}.

\begin{proposition}{\rm (\cite[Proposition 3.19]{20BDHMN})}\label{lem: quasi power sum and power sum}
Let $\lambda \vdash n$ and $p_{\lambda}$ be the power sum symmetric function associated with $\lambda$.
Then, $p_\lambda = \sum_{\widetilde{\alpha} = \lambda} \Psi_\alpha$.
\end{proposition}

\subsection{The weak Bruhat orders on the symmetric group}\label{Sec: weak Bruaht order}
The symmetric group $\SG_n$ is generated by simple transpositions $s_i := (i,i+1)$ with $1 \le i \le n-1$. 
An expression of $\sigma \in \SG_n$ that uses the minimal number of simple transpositions is called a \emph{reduced expression for $\sigma$}.
This minimal number is denoted by $\ell(\sigma)$ and called the \emph{length} of $\sigma$. 

The {\em left descent set} and {\em right descent set} of a permutation $\sigma$ are defined by 
\begin{align*}
\Des{L}{\sigma} &:= \{i \in [n-1] \mid \ell(s_i \sigma) < \ell(\sigma)\}  \text{ and}\\ \Des{R}{\sigma} &:= \{i \in [n-1] \mid \ell(\sigma s_i) < \ell(\sigma)\},
    \end{align*}
respectively.
The \emph{left weak Bruhat order} $\preceq_L$ and 
\emph{right weak Bruhat order} $\preceq_R$ on $\SG_n$
are defined to be the partial order on $\SG_n$ whose covering relation $\preceq_L^c$ and $\preceq_R^c$ are given as follows: 
\begin{align*}
& \sigma \preceq_L^c s_i \sigma \ \text{ if and only if } \ i \notin \Des{L}{\sigma} \text{ and }\\
& \sigma \preceq_R^c \sigma s_i \ \text{ if and only if }\  i \notin \Des{R}{\sigma},
\end{align*}
respectively.
Given $\sigma, \rho \in \SG_n$,
if $\sigma \preceq_L \rho$, then 
the \emph{left weak Bruhat interval} from $\sigma$ to $\rho$ is defined by
\[
[\sigma,\rho]_L := \{\gamma \in \SG_n \mid \sigma \preceq_L \gamma \preceq_L \rho \},
\]
and if $\sigma \preceq_R \rho$, then 
the \emph{right weak Bruhat interval}  from $\sigma$ to $\rho$ is defined by
\[
[\sigma,\rho]_R := \{\gamma \in \SG_n \mid \sigma \preceq_R \gamma \preceq_R \rho \}.
\]

For $\sigma \in \SG_n$, the {\em inversion} set and {\em coinversion} set of $\sigma$ are defined by 
\begin{align*}
\inv{\sigma} &:= \{(\sigma(i), \sigma(j))  \mid 1 \leq i < j \leq n, \text{ but } \sigma(i) > \sigma(j) \} \quad \text{and} \\ 
\coinv{\sigma} &:= \{(\sigma(i), \sigma(j)) \mid 1 \leq i < j \leq n \, \text{ and } \sigma(i) < \sigma(j) \},
\end{align*}
respectively.
Note that $\sigma \preceq_R \rho$ if and only if $\sigma^{-1} \preceq_L \rho^{-1}$.
Combining this fact with \cite[Proposition 3.1.3]{06BB} or \cite[Proposition 3.1]{91BW} it is clear that for $\sigma, \rho \in \SG_n$,
\begin{equation}\label{Eq: weak Bruhat right order2}
\begin{aligned}
\sigma \preceq_R \rho \ \text{ if and only if }  \ \inv{\sigma} \subseteq \inv{\rho}
\text{ if and only if } \ \coinv{\sigma} \supseteq \coinv{\rho}.
\end{aligned}
\end{equation} 

Let $w_0^{(n)}$ denote the longest element in $\SG_n$.
For $I \subseteq [n-1]$, let $\SG_{I}$ be the parabolic subgroup of $\SG_n$ generated by $\{s_i \mid i\in I\}$ and $w_0^{(n)}(I)$ the longest element in $\SG_{I}$.
For $\alpha \models n$, let $w_0^{(n)}(\alpha) := w_0^{(n)}(\set(\alpha))$. 
When $n$ is clear in the context, we will omit the superscript $(n)$ from $w_0^{(n)}$ and $w^{(n)}_0(\alpha)$.

\subsection{The 0-Hecke algebra and the quasisymmetric characteristic}\label{subsec: 0-Hecke alg and QSym}
The $0$-Hecke algebra $H_n(0)$ is the associative $\C$-algebra with $1$ generated by the elements $\opi_1,\opi_2,\ldots,\opi_{n-1}$ subject to the following relations:
\begin{align*}
\opi_i^2 &= -\opi_i \quad \text{for $1\le i \le n-1$},\\
\opi_i \opi_{i+1} \opi_i &= \opi_{i+1} \opi_i \opi_{i+1}  \quad \text{for $1\le i \le n-2$},\\
\opi_i \opi_j &= \opi_j \opi_i \quad \text{if $|i-j| \ge 2$}.
\end{align*}
Another set of generators consists of $\pi_i:= \opi_i + 1$ for $i=1,2,\ldots,n-1$ with the same relations as above except that $\pi_i^2 = \pi_i$.

For any reduced expression $s_{i_1} s_{i_2} \cdots s_{i_p}$ for $\sigma \in \SG_n$, let $\opi_{\sigma} := \opi_{i_1} \opi_{i_2} \cdots \opi_{i_p}$ and $\pi_{\sigma} := \pi_{i_1} \pi_{i_2 } \cdots \pi_{i_p}$.
It is well known that these elements are independent of the choices of reduced expressions, and both $\{\opi_\sigma \mid \sigma \in \SG_n\}$ and $\{\pi_\sigma \mid \sigma \in \SG_n\}$ are $\mathbb C$-bases for $H_n(0)$.

According to \cite{79Norton}, there are $2^{n-1}$ pairwise inequivalent irreducible (right) $H_n(0)$-modules and $2^{n-1}$ pairwise inequivalent projective indecomposable $H_n(0)$-modules, which are naturally indexed by compositions of $n$.
For $\alpha \models n$, let $\mathbf{F}^R_{\alpha}$ denote the $1$-dimensional $\mathbb{C}$-vector space corresponding to $\alpha$, spanned by a vector $v_{\alpha}$ and endow it with the right $H_n(0)$-action as follows: for each $1 \le i \le n-1$,
\[
v_\alpha \cdot \opi_i = \begin{cases}
0 & i \notin \set(\alpha),\\
- v_\alpha & i \in \set(\alpha).
\end{cases}
\]
This module is the irreducible $1$-dimensional $H_n(0)$-module corresponding to $\alpha$.
And, the projective indecomposable $H_n(0)$-module corresponding to $\alpha$
is given by the submodule 
$\calP^R_\alpha := \pi_{w_0(\alpha^\rmc)} \opi_{w_0(\alpha)}  H_n(0)$
of the regular representation of $H_n(0)$.

Let $\modR$ be the category of right $H_n(0)$-modules.
Let $\calR(\modR)$ denote the $\Z$-span of the isomorphism classes of finite dimensional right $H_n(0)$-modules. 
We denote by $[M]$ the isomorphism class corresponding to an $H_n(0)$-module $M$. 
The \emph{Grothendieck group} $G_0(\modR)$ is the quotient of $\calR(\modR)$ modulo the relations $[M] = [M'] + [M'']$ whenever there exists a short exact sequence $0 \ra M' \ra M \ra M'' \ra 0$. 
By abusing notation, we denote by $[M]$ the element of $G_0(\modR)$ corresponding to an $H_n(0)$-module $M$.
The irreducible $H_n(0)$-modules form a free $\Z$-basis for $G_0(\modR)$. Let
\[
\calG(\modRb) := \bigoplus_{n \ge 0} G_0(\modR).
\]
Given $m,n \in \N_0$, let us view $H_m(0) \otimes H_n(0)$ as the subalgebra of $H_{m+n}(0)$ generated by $\{\pi_i \mid i \in [m+n-1] \setminus \{m\} \}$.
For a right $H_m(0)$-module $M$ and $H_n(0)$-module $N$, we define
\begin{align*}
[M] \boxtimes [N] = 
\left[
M \otimes N \uparrow_{H_m(0) \otimes H_n(0)}^{H_{m+n}(0)}
\right]
\quad \text{and} \quad
\Delta([M]) := \sum_{0 \le k \le m} 
\left[
M \downarrow_{H_k(0) \otimes H_{m-k}(0)}^{H_{m}(0)}
\right]. 
\end{align*}
It was shown in \cite{09BL} that $\calG(\modRb)$ has a Hopf algebra structure with the product $\boxtimes$ and coproduct $\Delta$.
It was shown in \cite{96DKLT, 97KT} that the linear map
\begin{equation}\label{quasi characteristic right}
\ch_R: \calG(\modRb) \ra \Qsym, \quad [\bfF^R_{\alpha}] \mapsto F_{\alpha},
\end{equation}
called the \emph{quasisymmetric characteristic} map, is a Hopf algebra isomorphism.

\subsection{Duchamp--Hivert--Thibon's poset modules}
\label{Duchamp--Hivert--Thibon's poset modules}
Let $P$ be a poset with underlying set $[n]$
and $\setLE{P}$ the set of all linear extensions of $P$.
To each $E \in \setLE{P}$ we assign the following two permutations in $\SG_n$ (written in one-line notation)
\begin{align}\label{eq: eta_R and eta_L}
E_R := i_1 i_2 \cdots i_n
\quad \text{and} \quad 
E_L := j_1 j_2 \cdots j_n.
\end{align}
Here $i_k$ denotes the $k$th smallest element in $E$
and $j_k$ the number of elements $m \in [n]$ with $m \preceq_E k$ for $1 \le k \le n$.
Let
\[
\SGR{P} := \{E_R \mid E \in \setLE{P}\} \subset \SG_n
\quad \text{and} \quad
\SGL{P} := \{E_L \mid E \in \setLE{P}\} \subset \SG_n.
\]
In this paper, following the convention of drawing Hasse diagrams of posets,
we use a bold edge (or strict edge) between $x$ and $y$ when $x \preceq_P y$, but $x > y$, while we use a plain edge (or natural edge) when $x \preceq_P y$ and $x < y$ in the Hasse diagram of $P$.
Here $\preceq_P$ is used to denote the order of $P$ and $\le$ the usual order on $[n]$.

\begin{example}\label{Ex: poset-51342}
If
\begin{tikzpicture}[baseline=3.5mm]
\def \hp {0.35}
\def \vp {0.45}
\def \ccc {1mm}
\node[left] at (\hp*-0.7, \vp*1.0) {$P = $};
\node[shape=circle,draw,minimum size=\ccc*3, inner sep=0pt] at (0, 0) (A5) {\tiny $5$};
\node[shape=circle,draw,minimum size=\ccc*3, inner sep=0pt] at (\hp*1, \vp) (A1) {\tiny $1$};
\node[shape=circle,draw,minimum size=\ccc*3, inner sep=0pt] at (0*\hp, \vp*2) (A3) {\tiny $3$};
\node[shape=circle,draw,minimum size=\ccc*3, inner sep=0pt] at (2*\hp, \vp*2) (A4) {\tiny $4$};
\node[shape=circle,draw,minimum size=\ccc*3, inner sep=0pt] at (3*\hp, \vp*0.5) (A2) {\tiny $2$};

\draw[line width = \lw] (A5) -- (A1);
\draw (A1) -- (A3);
\draw (A1) -- (A4) -- (A2);
\end{tikzpicture} \ ,
then we have
\begin{align*}
\SGR{P} & = \{ 25134,52134,25143,52143,51234,51324,51243\}, \\
\SGL{P} & =\{31452,32451,31542,32541,23451,24351,23541\}.
\end{align*}
\end{example}

Let $\poset{n}$ be the set of posets with $[n]$ as the underlying set.
For each $P \in \poset{n}$, Duchamp--Hivert--Thibon~\cite{02DHT} constructed a right $H_n(0)$-module $\DHTMP{P}$ by endowing $\C\SGR{P}$ with the right $H_n(0)$-action induced from 
\begin{align}\label{eq: action of M_P}
\gamma \cdot \opi_i:= 
\begin{cases}
-\gamma & \text{if } i \in \Des{R}{\gamma},\\
0 & \text{if $i \notin \Des{R}{\gamma}$ and $\gamma s_i \notin \SGR{P}$},\\
\gamma s_i & \text{if } i \notin \Des{R}{\gamma} \text{ and } \gamma s_i \in \SGR{P}
\end{cases}
\end{align}
for $i \in [n-1]$ and $\gamma \in \SGR{P}$.
In this paper, we call this module as the \emph{poset module associated with $P$}. 

Every poset $P\in \poset{n}$ can be viewed as the labeled poset $(P, \omega)$, where 
the labeling $\omega: P \ra [n]$ is given by $\omega(i) = i$ for $1\le i \le n$.
Under this consideration, a map $f: [n] \ra \N_0$ is a \emph{$P$-partition} 
if it satisfies the following conditions:
\begin{enumerate}[label = {\rm (\arabic*)}]
\item 
If $i \preceq_P j$, then $f(i) \le f(j)$.
\item 
If $i \preceq_P j$ and $i > j$, then $f(i) < f(j)$.
\end{enumerate}
Here $\preceq_P$ is used to denote the order of $P$ and $\le$ the usual order on $[n]$.
Hence the \emph{$P$-partition generating function $K_P$} of $P$ is written as 
\[
\sum_{f:  \text{$P$-partition}} \prod_{i \in [n]} x_{f(i)}.
\]
It was shown in~\cite{99Stanley} that 
\[
K_P= \sum_{\sigma \in \SGR{P}} F_{\comp(\Des{R}{\sigma})}.
\]

Let us recall the results on the quasisymmetric characteristic of poset modules and their induction product.
Given $P_1 = ([m], \preceq_{P_1})$ and $P_2 = ([n], \preceq_{P_2})$, define the poset $P_1 \sqcup P_2 := ([m+n], \preceq_{P_1 \sqcup P_2})$ by 
\[
i \prec_{P_1 \sqcup P_2} j \quad \text{if and only if} \quad i \prec_{P_1} j \ \ \text{or} \ \  i-m \prec_{P_2} j-m.
\]

\begin{proposition}{\rm (\cite[Proposition 3.21]{02DHT})}\label{prop: DHT proposition}
Let $m,n$ be positive integers. 
Then we have the following.
\begin{enumerate}[label = {\rm (\alph*)}]
\item Let $P\in \poset{n}$. Then $\ch_R(\DHTMP{P}) = K_P$.

\item Let $P_1\in \poset{m}$ and $P_2\in \poset{n}$. Then $M_{P_1 \sqcup P_2} = M_{P_1} \boxtimes M_{P_2}$, and thus  $\ch_R(M_{P_1 \sqcup P_2}) = K_{P_1} K_{P_2} = \ch_R(M_{P_1}) \ch_R(M_{P_2})$.
\end{enumerate}
\end{proposition}

\section{Poset module categories}
\label{Sec: poset module category}

For each nonnegative integer $n$, we define $\scrP{}{n}$ to be the full subcategory of $\modR$ whose objects are direct-sums of finitely many 
isomorphic copies of poset modules. 
The first aim of this section is to show that $\bigoplus_{n \ge 0} G_0(\scrP{}{n})$ has a Hopf algebraic structure that is isomorphic to $\Qsym$ with the induction product and restriction coproduct,
where $G_0(\scrP{}{n})$ denotes the Grothendieck group of $\scrP{}{n}$.
The second aim is to describe the $H_n(0)$-modules obtained by applying automorphisms on $H_n(0)$ in terms of poset modules.

\subsection{The Hopf algebra structure of \texorpdfstring{$\bigoplus_{n \geq 0} G_0(\scrP{}{n})$}{Lg}}
\label{Sec: The Hopf algebra structure}
The aim of the subsection is to show that $\bigoplus_{n \ge 0} G_0(\scrP{}{n})$ has a Hopf algebraic structure that is isomorphic to $\Qsym$ with the induction product and restriction coproduct.

Let $n$ be a nonnegative integer.
We first show that  
given any poset $P$ with $|\setLE{P}| > 1$, there exist $P',P'' \in \poset{n}$ such that
\[
\begin{tikzcd}
0 \arrow[r] & M_{P'} \arrow[r] & M_{P} \arrow[r] & M_{P''} \arrow[r] & 0
\end{tikzcd}
\]
is a short exact sequence of right $H_n(0)$-modules.
Using this result repeatedly yields that 
$G_0(\scrP{}{n})$ and $G_0(\modR)$ are isomorphic as abelian groups.

Given $P \in \poset{n}$, let 
$\inc{P}:= \{ (u,v) \in P \times P \mid \text{$u$ and $v$ are incomparable in $P$}\}$.
For $(u,v) \in \inc{P}$, let $\Puv{u}{v}$ be the poset in $\poset{n}$ whose covering relations are given as follows:
\[
y \text{ covers } x \text{ in } \Puv{u}{v}
\text{ if either } 
y \text{ covers } x \text{ in } P  \text{ or }  (x,y)=(u,v).
\]
From the definition of $\Puv{u}{v}$, it follows that $\SGR{P}$ is the disjoint union of $\SGR{\Puv{u}{v}}$ and $\SGR{\Puv{v}{u}}$ (see \cref{Ex: 51342 divided}).
Let $\upiota: M_{\Puv{v}{u}} \rightarrow M_{P}$ denote the natural injection defined by 
\[
\upiota(\gamma) = \gamma \quad \text{for } \gamma \in \SGR{\Puv{v}{u}},
\]
and let $\sfpr: M_{P} \rightarrow M_{\Puv{u}{v}}$ be the natural projection defined by
\[
\sfpr(\gamma) = \begin{cases}
\gamma & \text{if } \gamma \in \SGR{\Puv{u}{v}}, \\
0 & \text{otherwise,}
\end{cases}\quad \text{for } \gamma \in \SGR{P}.
\]
It holds that $\rmIm(\upiota) = \ker(\sfpr)$.

\begin{lemma}\label{Lem: short exact sequence MP}
Let $P \in \poset{n}$.
For $(u,v) \in \inc{P}$ with $u < v$ $($in the usual order$)$,
\[
\begin{tikzcd}
0 \arrow[r] & M_{\Puv{v}{u}} \arrow[r, "\upiota"] & M_{P} \arrow[r, "\sfpr"] & M_{\Puv{u}{v}} \arrow[r] & 0
\end{tikzcd}
\]
is a short exact sequence of $H_n(0)$-modules,
where $\upiota$ is the natural injection and $\sfpr$ the natural projection.
\end{lemma}
\begin{proof}
For the assertion, we have only to show that both $\C \SGR{P_{(v,u)}}$ and $\C \SGR{P_{(u,v)}}$ are closed under the $H_n(0)$-action. 
To do this, take arbitrary $\gamma \in \SGR{P_{(v,u)}}$ and $i \in [n-1]$.
We may assume that $\gamma \cdot \opi_i = \gamma s_i$, since otherwise $\gamma \cdot \opi_i \in \C\SGR{P_{(v,u)}}$ is clear.
Note that for any $\sigma \in \SGR{P}$
\begin{align}\label{eq: coinv and Sigma_R}
(v,u) \in \inv{\sigma} \quad \text{if and only if} \quad \sigma \in \SGR{P_{(v,u)}}.
\end{align}
Since $\gamma \in \SGR{P_{(v,u)}}$, we have $(v,u) \in \inv{\gamma}$.
In addition, the assumption $\gamma \cdot \opi_i = \gamma s_i$ implies that $\gamma \preceq_R \gamma s_i$.
Thus, by~\cref{Eq: weak Bruhat right order2}, we have that $(v,u) \in \inv{\gamma s_i}$.
This together with \cref{eq: coinv and Sigma_R} shows that $\gamma s_i \in \SGR{P_{(v,u)}}$, as desired.
In a similar manner, we can see that $\C \SGR{P_{(u,v)}}$ is closed under the $H_n(0)$-action. 
\end{proof}

\begin{example}\label{Ex: 51342 divided}
For
\begin{tikzpicture}[baseline=3.5mm]
\def \hp {0.35}
\def \vp {0.45}
\def \ccc {1mm}
\node[left] at (\hp*-0.7, \vp*1.0) {$P = $};
\node[shape=circle,draw,minimum size=\ccc*3, inner sep=0pt] at (0, 0) (A5) {\tiny $5$};
\node[shape=circle,draw,minimum size=\ccc*3, inner sep=0pt] at (\hp*1, \vp) (A1) {\tiny $1$};
\node[shape=circle,draw,minimum size=\ccc*3, inner sep=0pt] at (0*\hp, \vp*2) (A3) {\tiny $3$};
\node[shape=circle,draw,minimum size=\ccc*3, inner sep=0pt] at (2*\hp, \vp*2) (A4) {\tiny $4$};
\node[shape=circle,draw,minimum size=\ccc*3, inner sep=0pt] at (3*\hp, \vp*0.5) (A2) {\tiny $2$};

\draw[line width = \lw] (A5) -- (A1);
\draw (A1) -- (A3);
\draw (A1) -- (A4) -- (A2);
\end{tikzpicture},
we have $(1,2) \in \inc{P}$. 
Note that
\begin{tikzpicture}[baseline=3.5mm]
\def \hp {0.35}
\def \vp {0.45}
\def \ccc {1mm}
\node[left] at (\hp*-0.7, \vp*1.0) {$P_{(2,1)} = $};
\node[shape=circle,draw,minimum size=\ccc*3, inner sep=0pt] at (0, 0) (A5) {\tiny $5$};
\node[shape=circle,draw,minimum size=\ccc*3, inner sep=0pt] at (\hp*1, \vp) (A1) {\tiny $1$};
\node[shape=circle,draw,minimum size=\ccc*3, inner sep=0pt] at (0*\hp, \vp*2) (A3) {\tiny $3$};
\node[shape=circle,draw,minimum size=\ccc*3, inner sep=0pt] at (2*\hp, \vp*2) (A4) {\tiny $4$};
\node[shape=circle,draw,minimum size=\ccc*3, inner sep=0pt] at (3*\hp, \vp*0.2) (A2) {\tiny $2$};

\draw[line width = \lw] (A5) -- (A1);
\draw (A1) -- (A3);
\draw (A1) -- (A4) -- (A2);
\draw[line width = \lw] (A2) -- (A1);
\end{tikzpicture}
and
\begin{tikzpicture}[baseline=7mm]
\def \hp {0.35}
\def \vp {0.45}
\def \ccc {1mm}
\node[left] at (\hp*-0.7, \vp*1.8) {$P_{(1,2)} = $};
\node[shape=circle,draw,minimum size=\ccc*3, inner sep=0pt] at (0, 0) (A5) {\tiny $5$};
\node[shape=circle,draw,minimum size=\ccc*3, inner sep=0pt] at (\hp*1, \vp) (A1) {\tiny $1$};
\node[shape=circle,draw,minimum size=\ccc*3, inner sep=0pt] at (0*\hp, \vp*2) (A3) {\tiny $3$};
\node[shape=circle,draw,minimum size=\ccc*3, inner sep=0pt] at (3*\hp, \vp*3) (A4) {\tiny $4$};
\node[shape=circle,draw,minimum size=\ccc*3, inner sep=0pt] at (2*\hp, \vp*2) (A2) {\tiny $2$};

\draw[line width = \lw] (A5) -- (A1);
\draw (A1) -- (A3);
\draw (A1) -- (A2) -- (A4);
\end{tikzpicture}.
The sets
\begin{align*}
\SGR{P} & = \{ 25134,52134,25143,52143,51234,51324,51243\}, \\
\SGR{\Puv{1}{2}} & = \{51234,51324,51243 \},  \\
\SGR{\Puv{2}{1}} & = \{25134,52134,25143,52143\}
\end{align*}
are bases for $M_{P}$, $\upiota(M_{P_{(2,1)}}) \cong M_{P_{(2,1)}}$, and $M_{P} / \upiota(M_{P_{(2,1)}}) \cong M_{P_{(1,2)}}$, respectively.
The $H_5(0)$-actions on these bases are illustrated in \cref{fig: H5(0)-actions}. 
In this figure, $\cdot \opi_i$ ($1 \leq i \leq 4$) denotes the right $H_n(0)$-action and
the symbol
\begin{tikzpicture}[baseline=-1mm]
\def \hp {3em}
\def \vp {4em}
\node at (\hp*0, \vp*0) {} edge [out=40,in=320, loop] ();
\node at (\hp*0 + 0.7*\hp, \vp*0) {\tiny $\cdot \opi_i$};
\end{tikzpicture}
represents that $\opi_i$ acts as $-\id$.
For each basis element, we do not display the $\opi_i$-action on it if $\opi_i$ acts as $0$.
\begin{figure}[ht]
\centering
\captionsetup{justification=centering}
\[
\begin{tikzpicture}
\def \hp {5.3em}
\def \vp {4em}
\node at (\hp*0, \vp*0) (A1) {$25134$};
\node[right,xshift=\hp*0.15] at (A1) {} edge [out=40,in=320, loop] ();
\node[right,xshift=\hp*0.5] at (A1) {\small $\cdot \opi_2$};

\node at (\hp*2, \vp*0) (A2) {$51234$};
\node[right,xshift=\hp*0.15] at (A2) {} edge [out=40,in=320, loop] ();
\node[right,xshift=\hp*0.5] at (A2) {\small $\cdot \opi_1$};

\node at (\hp*-1, \vp*-1) (B1) {$52134$};
\node at (\hp*-1 + 0.25*\hp, \vp*-1) {} edge [out=40,in=320, loop] ();
\node[right,xshift=\hp*0.5] at (B1) {\small $\cdot \opi_1, \cdot \opi_2$};

\node at (\hp*1, \vp*-1) (B2) {$25143$};
\node at (\hp*1 + 0.25*\hp, \vp*-1) {} edge [out=40,in=320, loop] ();
\node[xshift=\hp*0.85] at (B2) {\small $\cdot \opi_2, \cdot \opi_4$};

\node at (\hp*2.5, \vp*-1) (B3) {$51324$};
\node at (\hp*2.5 + 0.25*\hp, \vp*-1) {} edge [out=40,in=320, loop] ();
\node[xshift=\hp*0.85] at (B3) {\small $\cdot \opi_1, \cdot \opi_3$};

\node at (\hp*4, \vp*-1) (B4) {$51243$};
\node at (\hp*4 + 0.25*\hp, \vp*-1) {} edge [out=40,in=320, loop] ();
\node[xshift=\hp*0.85] at (B4) {\small $\cdot \opi_1, \cdot \opi_4$};

\node at (\hp*0, \vp*-2) (C1) {$52143$};
\node at (\hp*0 + 0.25*\hp, \vp*-2) {} edge [out=40,in=320, loop] ();
\node[xshift=\hp*1] at (C1) {\small $\cdot \opi_1, \cdot \opi_2, \cdot \opi_4$};

\draw[->] (A1) -- (B1);
\node at (\hp*-0.6, \vp*-0.4) {\small $\cdot \opi_1$};
\draw[->] (A1) -- (B2) node[right,pos=0.2] {\small $\cdot \opi_4$};
\draw[->] (A2) -- (B3) node[right,pos=0.6]  {\small $\cdot \opi_3$};
\draw[->] (A2) -- (B4) node[right,pos=0.4] {\small $\cdot \opi_4$};
\draw[->] (A2) -- (B1) node[below,pos=0.03] {\small $\cdot \opi_2$};
\draw[->] (B1) -- (C1);
\node at (\hp*-0.6, \vp*-1.6) {\small $\cdot \opi_4$};
\draw[->] (B2) -- (C1);
\node at (\hp*0.6, \vp*-1.6) {\small $\cdot \opi_1$};
\draw[->] (B4) -- (C1) node[pos=0.43,below]  {\small $\cdot \opi_2$};

\draw[-,blue,dotted,thick]
(\hp*-1.5, \vp*0.3) -- (\hp*-1.5, \vp*-2.3) -- (\hp*2.1, \vp*-2.3) -- (\hp*2.1, \vp*-0.85) -- (\hp*0.7, \vp*0.3) -- (\hp*-1.5, \vp*0.3);
\node[below] at (\hp*0.4,\vp*-2.35) {$\SGR{P_{(2,1)}}$};

\draw[-,red,dotted,thick] 
(\hp*1.65, \vp*0.3) -- (\hp*1.65, \vp*-0.25) -- (\hp*2.4, \vp*-1.5) -- (\hp*5.2, \vp*-1.5) -- (\hp*5.2, \vp*-0.8) -- (\hp*2.7, \vp*0.3) -- (\hp*1.65, \vp*0.3);
\node[below] at (\hp*3.7,\vp*-1.6) {$\SGR{P_{(1,2)}}$};
\end{tikzpicture}
\]
\caption{The right $H_5(0)$-actions on $\SGR{P}$, $\SGR{P_{(2,1)}}$, and $\SGR{P_{(1,2)}}$}
\label{fig: H5(0)-actions}
\end{figure}
\end{example}

We then study the induction product and restriction coproduct on $\bigoplus_{n \geq 0}G_0(\scrP{}{n})$.
In \cite[Proposition 3.21]{02DHT}, Duchamp--Hivert--Thibon observed that the induction product of poset modules is again a poset module.
So we will focus on the restriction of poset modules.
Let us introduce necessary definitions and notation.

For $P \in \poset{n}$,
a subset $Q$ of $P$ is called an 
\emph{induced subposet}  (simply, {\em subposet}) of $P$ if $Q$ is equipped with a partial order $\preceq_Q$ satisfying that $x \preceq_Q y$ if and only if $x \preceq_P y$ for all $x,y \in Q$.
A \emph{lower subposet of $P$} is a subposet $Q$ of $P$ such that for all $x \in Q$ and $y \in P$, if $y \preceq_P x$ then $y \in Q$.
For $0 \le m \le n$, define
\[
\sfLS{P}{m} := \left\{Q \; \middle| \; 
\text{$Q$ is a lower subposet of $P$ and $|Q| = m$}
\right\}.
\]
For $Q \in \sfLS{P}{m}$, by abuse of notation, we will use $P \setminus Q$ to denote the subposet of $P$ whose underlying set is $P \setminus Q$.
Given a subposet $Q = \{k_1 < k_2 < \cdots < k_m\}$ of $P$, define $\rmst(Q)$ to be the poset $([m], \preceq_{\rmst(Q)})$ such that $i \preceq_{\rmst(Q)} j$ if $k_i \preceq_Q k_j$ for $i,j \in [m]$.

\begin{lemma}\label{Lem: restriction}
Let $P \in \poset{n}$ and $0 \leq m \leq n$. 
As $H_m(0) \otimes H_{n-m}(0)$-modules,
\begin{align*}
\DHTMP{P}\downarrow_{H_m(0) \otimes H_{n-m}(0)}^{H_{n}(0)} 
\hspace{1ex} \cong \hspace{-0.5ex} \bigoplus_{Q \in \sfLS{P}{m} }
\DHTMP{\rmst(Q)} \otimes \DHTMP{\rmst(P\setminus Q)}.
\end{align*}
\end{lemma}
\begin{proof}
For $1 \le r \le s \le n$, define a map 
$$
\rmst_{r,s}: \SG_n \ra \SG_{s-r+1}, 
\quad
\gamma \mapsto \rmst_{r,s}(\gamma),
$$ 
where $\rmst_{r,s}(\gamma)$ is a permutation in $\SG_{s-r+1}$ defined by 
$$
(\rmst_{r,s}(\gamma))(i) = |\{j \in [r,s] \mid \gamma(j) \le \gamma(i) \}|
\quad \text{for all $1\le i \le s-r+1$.}
$$
Let $f : \DHTMP{P} \ra 
\bigcup_{Q \in \sfLS{P}{m}} M_{\rmst(Q)} \otimes M_{\rmst(P \setminus Q)}$ be the $\C$-linear map defined by letting
\[
f(E_R) := 
\rmst_{1,m}(E_R) \otimes  \rmst_{m+1,n}(E_R)
\]
and extending it by linearity.
One can easily see that $f$ is a well-defined bijection.

To prove that $f$ is an $H_m(0) \otimes H_{n-m}(0)$-module homomorphism, it suffices to show that
\begin{align*}
f(E_R \cdot \opi_i) = \begin{cases}
f(E_R) \cdot \opi_i \otimes 1
& \text{for $1 \leq i \leq m-1$,} \\
f(E_R) \cdot 1 \otimes \opi_{i-m}
& \text{for $m+1 \leq i \leq n-1$.}
\end{cases}
\end{align*}
In fact, this follows from that 
\begin{align*}
f(E_R \cdot \opi_i) 
& = \begin{cases}
\rmst_{1,m}(E_R) \cdot \opi_i \otimes \rmst_{m+1,n}(E_R) & \text{if $1 \leq i \leq m-1$,} \\
\rmst_{1,m}(E_R) \otimes \rmst_{m+1,n}(E_R) \cdot \opi_{i-m} & \text{if $m+1 \leq i \leq n-1$.}
\end{cases}
\end{align*}
\end{proof}

\begin{example}\label{eg: restriction}
Given
\begin{tikzpicture}[baseline=3.5mm]
\def \hp {0.35}
\def \vp {0.45}
\def \ccc {1mm}
\node[left] at (\hp*-0.7, \vp*1.0) {$P = $};
\node[shape=circle,draw,minimum size=\ccc*3, inner sep=0pt] at (0, 0) (A5) {\tiny $5$};
\node[shape=circle,draw,minimum size=\ccc*3, inner sep=0pt] at (\hp*1, \vp) (A1) {\tiny $1$};
\node[shape=circle,draw,minimum size=\ccc*3, inner sep=0pt] at (0*\hp, \vp*2) (A3) {\tiny $3$};
\node[shape=circle,draw,minimum size=\ccc*3, inner sep=0pt] at (2*\hp, \vp*2) (A4) {\tiny $4$};
\node[shape=circle,draw,minimum size=\ccc*3, inner sep=0pt] at (3*\hp, \vp*0.5) (A2) {\tiny $2$};

\draw[line width=0.5mm] (A5) -- (A1);
\draw (A1) -- (A3);
\draw (A1) -- (A4) -- (A2);
\end{tikzpicture}\ ,
one has
$\sfLS{P}{3} = \left\{
\begin{tikzpicture}[baseline=3.5mm]
\def \hp {0.35}
\def \vp {0.45}
\def \ccc {1mm}
\node[left] at (\hp*-0.7, \vp*1.0) {$Q_1 = $};
\node[shape=circle,draw,minimum size=\ccc*3, inner sep=0pt] at (0, 0) (A5) {\tiny $5$};
\node[shape=circle,draw,minimum size=\ccc*3, inner sep=0pt] at (\hp*1, \vp) (A1) {\tiny $1$};
\node[shape=circle,draw,minimum size=\ccc*3, inner sep=0pt] at (0*\hp, \vp*2) (A3) {\tiny $3$};
\draw[line width=0.5mm] (A5) -- (A1);
\draw (A1) -- (A3);
\end{tikzpicture},
\ 
\begin{tikzpicture}[baseline=3.5mm]
\def \hp {0.35}
\def \vp {0.45}
\def \ccc {1mm}
\node[left] at (\hp*-0.7, \vp*1.0) {$Q_2 = $};
\node[shape=circle,draw,minimum size=\ccc*3, inner sep=0pt] at (0, 0) (A5) {\tiny $5$};
\node[shape=circle,draw,minimum size=\ccc*3, inner sep=0pt] at (\hp*1, \vp) (A1) {\tiny $1$};
\node[shape=circle,draw,minimum size=\ccc*3, inner sep=0pt] at (3*\hp, \vp*0.5) (A2) {\tiny $2$};

\draw[line width=0.5mm] (A5) -- (A1);
\end{tikzpicture}
\right\}.$
By \cref{Lem: restriction} we have that
\begin{align*}
\DHTMP{P}\downarrow_{H_3(0) \otimes H_2(0)}^{H_{5}(0)}
& \cong 
\left( \DHTMP{\rmst(Q_1)}\otimes\DHTMP{\rmst(P \setminus Q_1)} \right) \oplus 
\left( \DHTMP{\rmst(Q_2)}\otimes\DHTMP{\rmst(P \setminus Q_2)} \right).
\end{align*}
Indeed, if we set $\opi^{(1)}_1 := \opi_1 \otimes \id$, $\opi^{(1)}_2 := \opi_2 \otimes \id$, and $\opi^{(2)}_1 := \id \otimes \opi_1$, then the $H_3(0) \otimes H_2(0)$-action on $\DHTMP{P}\downarrow_{H_3(0) \otimes H_2(0)}^{H_{5}(0)}$ and $(\DHTMP{\rmst(Q_1)}\otimes\DHTMP{\rmst(P \setminus Q_1)}) \oplus (\DHTMP{\rmst(Q_2)}\otimes\DHTMP{\rmst(P \setminus Q_2)})$
can be illustrated as in \cref{fig: action on restriction} and \cref{fig: action on tensor for restriction}, respectively.

\begin{figure}[ht]
\begin{displaymath}
\begin{tikzpicture}
\def \hp {2}
\def \vp {1.5}
\node at (\hp*0, \vp*0) (A1) {$25134$};
\node at (0.2*\hp, 0*\vp) {} edge [out=40,in=320, loop] ();
\node at (-0.15*\hp + -0.5*\hp, 0.1*\vp -0.5*\vp) {\footnotesize $\opi_1^{(1)}$};
\node at (0.15*\hp + 0.5*\hp, 0.1*\vp -0.5*\vp) {\footnotesize $\opi_1^{(2)}$};
\node at (\hp*2, \vp*0) (A2) {$51234$};
\node at (2.2*\hp, 0*\vp) {} edge [out=40,in=320, loop] ();
\node at (2.7*\hp, 0.05*\vp) {\footnotesize $\opi_1^{(1)}$};
\node at (2.8*\hp + 0.5*\hp, 0*\vp -0.45*\vp) {\footnotesize $\opi_1^{(2)}$};

\node at (\hp*-1, \vp*-1) (B1) {$52134$};
\node at (0.2*\hp+\hp*-1, -1*\vp) {} edge [out=40,in=320, loop] ();
\node at (0.85*\hp+\hp*-1, -0.95*\vp) {\footnotesize $\opi_1^{(1)}, \opi_2^{(1)}$};
\node at (-0.2*\hp - 0.5*\hp, -0.1*\vp -1.5*\vp) {\footnotesize $\opi_1^{(2)}$};

\node at (\hp*1, \vp*-1) (B2) {$25143$};
\node at (0.2*\hp+\hp*1, -1*\vp) {} edge [out=40,in=320, loop] ();
\node at (0.85*\hp+\hp*1, -0.95*\vp) {\footnotesize $\opi_2^{(1)}, \opi_1^{(2)}$};
\node at (0.15*\hp + 0.5*\hp, -0.1*\vp -1.5*\vp) {\footnotesize $\opi_1^{(1)}$};

\node at (\hp*2.6, \vp*-1) (B3) {$51324$};
\node at (0.3*\hp+\hp*2.5, -1*\vp) {} edge [out=40,in=320, loop] ();
\node at (0.8*\hp+\hp*2.5, -0.95*\vp) {\footnotesize $\opi_1^{(1)}$};

\node at (\hp*4, \vp*-1) (B4) {$51243$};
\node at (0.2*\hp+\hp*4, -1*\vp) {} edge [out=40,in=320, loop] ();
\node at (0.85*\hp+\hp*4, -0.95*\vp) {\footnotesize $\opi_1^{(1)},\opi_1^{(2)}$};

\node at (\hp*0, \vp*-2) (C1) {$52143$};
\node at (0.2*\hp+\hp*0, -2*\vp) {} edge [out=40,in=320, loop] ();
\node at (1.05*\hp+\hp*0, -2*\vp) {\footnotesize $\opi_1^{(1)},\opi_2^{(1)},\opi_1^{(2)}$};
\draw[->] (B4) -- (C1) node[below,pos=0.2] {\footnotesize $\opi_2^{(1)}$};

\draw[->] (A1) -- (B1);
\draw[->] (A1) -- (B2);
\draw[->] (A2) -- (B1) node[below,pos=0.1] {\footnotesize $\opi_2^{(1)}$};

\draw[->] (A2) -- (B4);
\draw[->] (B1) -- (C1);
\draw[->] (B2) -- (C1);
\end{tikzpicture}
\end{displaymath}
\caption{The $H_3(0) \otimes H_2(0)$-action on $\DHTMP{P}\downarrow_{H_3(0) \otimes H_2(0)}^{H_{5}(0)}$ in \cref{eg: restriction}}
\label{fig: action on restriction}

\begin{tikzpicture}
\def \hp {2}
\def \vp {1.5}
\def \hhhh {7}
\node at (-3*\hp, \vp*-1) (B3) {$312 \otimes 12$};
\node at (0.2*\hp+\hp*-2.85, -1*\vp) {} edge [out=40,in=320, loop] ();
\node at (0.7*\hp+\hp*-2.85, -1*\vp) {\footnotesize $\opi_1^{(1)}$};
\node at (-1.75*\hp, \vp*-1) {$\oplus$};

\node at (\hp*0, \vp*0) (A1) {$231 \otimes 12$};
\node at (0.35*\hp, 0*\vp) {} edge [out=40,in=320, loop] ();
\node at (0.85*\hp, 0.05*\vp) {\footnotesize $\opi_2^{(1)}$};
\node at (-0.1*\hp + -0.5*\hp, 0.1*\vp -0.5*\vp) {\footnotesize $\opi_1^{(1)}$};
\node at (0.15*\hp + 0.5*\hp, 0.1*\vp -0.5*\vp) {\footnotesize $\opi_1^{(2)}$};

\node at (\hp*2, \vp*0) (A2) {$312 \otimes 12$};
\node at (2.35*\hp, 0*\vp) {} edge [out=40,in=320, loop] ();
\node at (2.85*\hp, 0.05*\vp) {\footnotesize $\opi_1^{(1)}$};
\node at (2.25*\hp + 0.5*\hp, 0*\vp -0.5*\vp) {\footnotesize $\opi_1^{(2)}$};

\node at (\hp*-1, \vp*-1) (B1) {$321 \otimes 12$};
\node at (0.35*\hp+\hp*-1, -1*\vp) {} edge [out=40,in=320, loop] ();
\node at (1*\hp+\hp*-1, -0.95*\vp) {\footnotesize $\opi_1^{(1)},\opi_2^{(1)}$};
\node at (-0.15*\hp - 0.5*\hp, -0.1*\vp -1.5*\vp) {\footnotesize $\opi_1^{(2)}$};
\node at (\hp*1, \vp*-1) (B2) {$231 \otimes 21$};
\node at (0.35*\hp+\hp*1, -1*\vp) {} edge [out=40,in=320, loop] ();
\node at (1*\hp+\hp*1, -0.95*\vp) {\footnotesize $\opi_2^{(1)},\opi_1^{(2)}$};
\node at (0.15*\hp + 0.5*\hp, -0.1*\vp -1.5*\vp) {\footnotesize $\opi_1^{(1)}$};
\node at (\hp*3, \vp*-1) (B4) {$312 \otimes 21$};
\node at (0.35*\hp+\hp*3, -1*\vp) {} edge [out=40,in=320, loop] ();
\node at (1*\hp+\hp*3, -0.95*\vp) {\footnotesize $\opi_1^{(1)},\opi_1^{(2)}$};
\node at (\hp*0, \vp*-2) (C1) {$321 \otimes 21$};
\node at (0.35*\hp+\hp*0, -2*\vp) {} edge [out=40,in=320, loop] ();
\node at (1.2*\hp+\hp*0, -1.95*\vp) {\footnotesize $\opi_1^{(1)},\opi_2^{(1)},\opi_1^{(2)}$};

\draw[->] (A1) -- (B1);
\draw[->] (A1) -- (B2);
\draw[->] (A2) -- (B1) node[below,pos=0.1] {\footnotesize $\opi_2^{(1)}$};;
\draw[->] (A2) -- (B4);
\draw[->] (B4) -- (C1) node[below,pos=0.2] {\footnotesize $\opi_2^{(1)}$};
\draw[->] (B1) -- (C1);
\draw[->] (B2) -- (C1);
\end{tikzpicture}
\caption{The $H_3(0) \otimes H_2(0)$-action on $(\DHTMP{\rmst(Q_1)}\otimes\DHTMP{\rmst(P \setminus Q_1)}) \oplus (\DHTMP{\rmst(Q_2)}\otimes\DHTMP{\rmst(P \setminus Q_2)})$ in \cref{eg: restriction}}
\label{fig: action on tensor for restriction}
\end{figure}
\end{example}

Note that each irreducible $H_n(0)$-module belongs to the category $\mathscr{P}(n)$. 
Let $\mathcal{R}(\mathscr{P}(n))$ denote the $\mathbb{Z}$-span of the isomorphism classes of finite-dimensional right $H_n(0)$-modules in $\mathscr{P}(n)$. 
We denote by $\langle M \rangle$ the isomorphism class corresponding to an $H_n(0)$-module $M$ in $\mathscr{P}(n)$. 
By employing \cref{Lem: short exact sequence MP} iteratively, it follows that $G_0(\mathscr{P}(n))$ is a free $\mathbb{Z}$-module with basis $\{\langle \mathbf{F}^R_{\alpha} \rangle \mid \alpha \models n \}$.
Consider the $\mathbb Z$-linear isomorphism 
\[
\Xi:\bigoplus_{n \ge 0} G_0(\scrP{}{n})\to \bigoplus_{n \ge 0} G_0(\modR), \quad \langle M \rangle \mapsto [M].
\]
Let $M$ be a right $H_m(0)$-module in $\mathscr{P}(m)$, and let $N$ be a right $H_n(0)$-module in $\mathscr{P}(n)$.
Due to \cref{prop: DHT proposition} and \cref{Lem: restriction}
we can define
\begin{align*}
\langle M \rangle \boxtimes \langle N \rangle := 
\left \langle
M \otimes N \uparrow_{H_m(0) \otimes H_n(0)}^{H_{m+n}(0)}
\right \rangle
\quad \text{and} \quad
\Delta(\langle M \rangle) := \sum_{0 \le k \le m} 
\left \langle
M \downarrow_{H_k(0) \otimes H_{m-k}(0)}^{H_{m}(0)}
\right \rangle. 
\end{align*}
It is obvious that 
$\Xi(\langle M \rangle \boxtimes \langle N \rangle)
=[M] \boxtimes [ N ]$ and 
$ \Xi(\Delta(\langle M \rangle))
=\Delta([ M ])$. 
It tells us that we can transport the Hopf algebra structure of 
$\bigoplus_{n \ge 0} G_0(\modR)$ to $\bigoplus_{n \ge 0} G_0(\scrP{}{n})$ via $\Xi$.
For the Hopf algebra structure obtained in this manner, it is evident that $\Xi$ is a Hopf algebra isomorphism.

Now, we are ready to state the main theorem of this subsection.

\begin{theorem}\label{Hopf algebra structure of Grothendieck group}
$\bigoplus_{n \ge 0} G_0(\scrP{}{n})$ is isomorphic to $\bigoplus_{n \ge 0} G_0(\modR)$ as a Hopf algebra,
and therefore it is isomorphic to $\Qsym$ as a Hopf algebra. 
\end{theorem}
\begin{proof}
The assertion can be derived by the fact that $\Xi$ and $\ch_R$ are Hopf algebra isomorphisms.
\end{proof}

\subsection{(Anti-)involution twists}
\label{Sec: automorphism twists}

In this subsection, we describe the (anti-)involution twists of poset modules and elucidate how (anti-)involution twists interact with the induction product and restriction coproduct.

Let us introduce the definitions of (anti-)involution twists.
Given an automorphism $\mu$ of $H_n(0)$ and a right $H_n(0)$-module $M$, we define $\mu[M]$ by the right $H_n(0)$-module with the same underlying space as $M$ and with the action $\cdot_\mu$ defined by
\[
v \cdot_\mu h :=  v \cdot \mu(h)  \quad \text{for $h \in H_n(0)$ and $v \in M$.}
\]
We define $\mathbf{T}^+_\mu : \modR \ra \modR$ to be the covariant functor, called the \emph{$\mu$-twist}, sending a right $H_n(0)$-module $M$ to $\mu[M]$ and an $H_n(0)$-module homomorphism $f:M \ra N$ to $\mathbf{T}^+_\mu(f): \mu[M] \ra \mu[N]$ defined by $\mathbf{T}^+_\mu(f)(v) = f(v)$ for $v \in M$.

Similarly, given an anti-automorphism $\nu$ of $H_n(0)$ and a right $H_n(0)$-module $M$, we define $\nu[M]$ by the right $H_n(0)$-module with $M^*$, the dual space of $M$, as the underlying space and with the action $\cdot^\nu$ defined by 
\begin{align}\label{Eq: T_nu^minus}
(\delta \cdot^\nu h)(v) := \delta(v \cdot \nu(h))
\quad \text{for $h \in H_n(0)$, $\delta \in M^*$, and $v \in M$.}
\end{align}
We define $\mathbf{T}^-_\nu: \modR \ra \modR$ to be the contravariant functor, called the \emph{$\nu$-twist}, sending an $H_n(0)$-module $M$ to $\nu[M]$ and an $H_n(0)$-module homomorphism $f:M \ra N$ to 
$\mathbf{T}^-_\nu(f): \nu[N] \ra \nu[M]$ defined by $\mathbf{T}^-_\nu(f)(\delta) = \delta \circ f$.

Next, let us introduce (anti-)involutions on $H_n(0)$, on $\poset{n}$ and on $\Qsym$ in our consideration.
As for (anti-)involutions on $H_n(0)$, we consider two involutions $\upphi, \uptheta$ and an anti-involution $\upchi$ on $H_n(0)$ defined by 
\[
\upphi(\opi_i) = \opi_{n-i}, 
\quad 
\uptheta (\opi_i) = -\pi_i, 
\quad \text{and} \quad 
\upchi(\opi_i) = \opi_i \quad (1 \le i \le n-1).
\]
These (anti-)involutions were introduced by Fayers \cite[Proposition 3.2]{05Fayers} and commute with each other.
As for involutions on $\poset{n}$, we consider 
\begin{align}\label{poset involutions}
& {}^-: \poset{n} \ra  \poset{n}, \  P \mapsto \overline{P}
\quad \text{and} \quad {}^*: \poset{n} \ra  \poset{n}, \ P \mapsto P^*,
\end{align}
where $\overline{P}$ and $P^*$ are posets in $\poset{n}$ whose orders are defined by
\begin{align*}
&u \preceq_{\overline{P}} v
 \ \Longleftrightarrow \ 
n + 1 - u \preceq_{P} n + 1 - v
\quad \text{and} \quad
u \preceq_{P^*} v
 \ \Longleftrightarrow \  
v \preceq_{P} u.
\end{align*}
These involutions also commute with each other.
As for involutions on $\Qsym$, we consider $\rho$ and $\psi$ defined by
$\rho(F_\alpha) = F_{\alpha^\rev}$ and $\psi(F_\alpha) = F_{\alpha^\rmc}$,
where $\alpha^\rev$ is the reverse composition of $\alpha$ and $\alpha^\rmc$ is the complement of $\alpha$.

Finally, for $P \in \poset{n}$, we introduce an $H_n(0)$-module $\ourMP{P}$ which is slightly different from $\DHTMP{P}$.
It is the right $H_n(0)$-module whose underlying space is  $\C\SGR{P}$ and whose $H_n(0)$-action is given by
\begin{align*}
\gamma \cdot \pi_i:= 
\begin{cases}
\gamma & \text{if $i \in \Des{R}{\gamma}$},\\
0 & \text{if $i \notin \Des{R}{\gamma}$ and $\gamma s_i \notin \SGR{P}$},\\
\gamma s_i & \text{if $i \notin \Des{R}{\gamma}$  and $\gamma s_i \in \SGR{P}$}
\end{cases}
\end{align*}
for $i \in [n-1]$ and $\gamma \in \SGR{P}$.

With the above preparation, the main result of this subsection can be stated as follows.

\begin{theorem}\label{thm: auto-twists}
For $P \in \poset{n}$, we have the following.
\begin{enumerate}[label = {\rm (\alph*)}]
\item As $H_n(0)$-modules, $\upphi[\DHTMP{P}] \cong   \DHTMP{\overline{P}^*}$, $\uptheta[\DHTMP{P}] \cong  \ourMP{P}$, and $\upchi[\DHTMP{P}] \cong \ourMP{\overline{P}}$.
\smallskip 

\item $\ch ([\upphi[\DHTMP{P}]]) = \rho (K_{P})$, $\ch ([\uptheta [\DHTMP{P}]]) = \psi (K_{P})$, and $\ch ([\upchi[\DHTMP{P}]]) = K_{P}$.
\end{enumerate}
\end{theorem}
\begin{proof}
Combining \cref{prop: DHT proposition} with \cref{thm: auto-twists} and \cite[Table 2]{22JKLO}, we immediately obtain the assertion (b).
We only show the assertion (a).

Let us consider the $\C$-linear isomorphisms defined by
\begin{align*}
f_1 &: \upphi[\DHTMP{P}] \ra   \DHTMP{\overline{P}^*}, \quad \gamma \mapsto w_0 \gamma w_0, \\
f_2 &: \uptheta[\DHTMP{P}] \ra   \ourMP{P}, \quad \gamma \mapsto (-1)^{\ell(\gamma)} \gamma,\\
f_3 &: \upchi[\DHTMP{P}] \ra   \ourMP{\overline{P}}, \quad \gamma^* \mapsto \gamma w_0,
\end{align*}
where $\gamma \in \SGR{P}$ and $\gamma^*$ denotes the dual of $\gamma$ with respect to the basis $\SGR{P}$ for $\DHTMP{P}$.
In a similar manner with \cite[Theorem 3.16]{22JKLO}, one can see that $f_1$, $f_2$, and $f_3$ are $H_n(0)$-module isomorphisms.
Thus, we omit the remaining part of the proof.
\end{proof}

\begin{example}
Let $P$ be the poset in \cref{Ex: poset-51342}.
Then
\[
\begin{tikzpicture}[baseline=3.5mm]
\def \hp {0.4}
\def \vp {0.5}
\def \ccc {1mm}
\node[left] at (\hp*-0.7, \vp*1.0) {$\overline{P}^* = $};
\node[shape=circle,draw,minimum size=\ccc*3, inner sep=0pt] at (0, 0) (A3) {\tiny $3$};
\node[shape=circle,draw,minimum size=\ccc*3, inner sep=0pt] at (\hp*1, \vp) (A5) {\tiny $5$};
\node[shape=circle,draw,minimum size=\ccc*3, inner sep=0pt] at (2*\hp, \vp*0) (A2) {\tiny $2$};
\node[shape=circle,draw,minimum size=\ccc*3, inner sep=0pt] at (2*\hp, \vp*2) (A1) {\tiny $1$};
\node[shape=circle,draw,minimum size=\ccc*3, inner sep=0pt] at (3*\hp, \vp*1) (A4) {\tiny $4$};

\draw[line width = \lw] (A5) -- (A1);
\draw (A3)-- (A5) -- (A2) -- (A4);
\end{tikzpicture}
\quad \text{and} \quad 
\begin{tikzpicture}[baseline=3.5mm]
\def \hp {0.4}
\def \vp {0.5}
\def \ccc {1mm}
\node[left] at (\hp*-0.7, \vp*1.0) {$\overline{P} = $};
\node[shape=circle,draw,minimum size=\ccc*3, inner sep=0pt] at (0, 0) (A1) {\tiny $1$};
\node[shape=circle,draw,minimum size=\ccc*3, inner sep=0pt] at (\hp*1, \vp) (A5) {\tiny $5$};
\node[shape=circle,draw,minimum size=\ccc*3, inner sep=0pt] at (0*\hp, \vp*2) (A3) {\tiny $3$};
\node[shape=circle,draw,minimum size=\ccc*3, inner sep=0pt] at (2*\hp, \vp*2) (A2) {\tiny $2$};
\node[shape=circle,draw,minimum size=\ccc*3, inner sep=0pt] at (3*\hp, \vp*1) (A4) {\tiny $4$};

\draw[line width = \lw] (A3) -- (A5) -- (A2) -- (A4);
\draw (A5) -- (A1);
\end{tikzpicture} \ .
\]
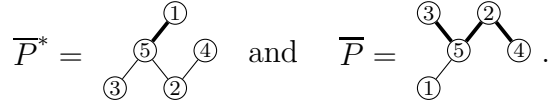
\cref{fig: H_5(0)-actions for twists} illustrates
the $H_5(0)$-actions on $\DHTMP{\overline{P}^*}(\cong \upphi[\DHTMP{P}])$, $\ourMP{P}(\cong \uptheta[\DHTMP{P}])$, and $\ourMP{\overline{P}}(\cong \upchi[\DHTMP{P}])$.
In this figure, the symbol
\begin{tikzpicture}[baseline=-1mm]
\def \hp {3em}
\def \vp {4em}
\node at (\hp*0, \vp*0) {} edge [out=40,in=320, loop] ();
\node at (\hp*0 + 0.7*\hp, \vp*0) {\tiny $\cdot \pi_i$};
\end{tikzpicture}
represents that $\pi_i$ acts as $\id$.
And, for each basis element, we do not display the $\pi_i$-action on it if $\pi_i$ acts as $0$.
\end{example}

\begin{figure}[ht]
\centering
\scalebox{0.7}{$
\begin{tikzpicture}
\def \hp {4.5em}
\def \vp {5em}
\node at (\hp*0, \vp*0) (A1) {$23514$};
\node[right,xshift=\hp*0.15] at (A1) {} edge [out=40,in=320, loop] ();
\node[xshift=\hp*0.8] at (A1) {\scriptsize $\cdot-\opi_3$};

\node at (\hp*2, \vp*0) (A2) {$23451$};
\node[right,xshift=\hp*0.15] at (A2) {} edge [out=40,in=320, loop] ();
\node[xshift=\hp*0.8] at (A2) {\scriptsize $\cdot-\opi_4$};

\node at (\hp*-1, \vp*-1) (B1) {$23541$};
\node[right,xshift=\hp*0.15] at (B1) {} edge [out=40,in=320, loop] ();
\node[xshift=\hp*0.8] at (B1) {\scriptsize 
$\begin{array}{l} 
\cdot-\opi_3, \\
\cdot-\opi_4
\end{array}$};

\node at (\hp*1, \vp*-1) (B2) {$32514$};
\node[right,xshift=\hp*0.15] at (B2) {} edge [out=40,in=320, loop] ();
\node[xshift=\hp*0.8] at (B2) {\scriptsize 
$\begin{array}{l} 
\cdot-\opi_1, \\
\cdot-\opi_3
\end{array}$};

\node at (\hp*2.5, \vp*-1) (B3) {$24351$};
\node[right,xshift=\hp*0.15] at (B3) {} edge [out=40,in=320, loop] ();
\node[xshift=\hp*0.8] at (B3) {\scriptsize 
$\begin{array}{l} 
\cdot-\opi_2, \\
\cdot-\opi_4
\end{array}$};

\node at (\hp*4, \vp*-1) (B4) {$32451$};
\node[right,xshift=\hp*0.15] at (B4) {} edge [out=40,in=320, loop] ();
\node[xshift=\hp*0.8] at (B4) {\scriptsize 
$\begin{array}{l}
\cdot-\opi_1, \\
\cdot-\opi_4
\end{array}$};

\node at (\hp*0, \vp*-2) (C1) {$32541$};
\node[right,xshift=\hp*0.15] at (C1) {} edge [out=40,in=320, loop] ();
\node[xshift=\hp*0.8] at (C1) {\scriptsize 
$\begin{array}{l}
\cdot-\opi_1,\\
\cdot-\opi_3, \\
\cdot-\opi_4
\end{array}$
};

\draw[->] (A1) -- (B1);
\node at (\hp*-0.6, \vp*-0.4) {\scriptsize  $\cdot\opi_4$};
\draw[->] (A1) -- (B2);
\node at (\hp*0.6, \vp*-0.4) {\scriptsize  $\cdot\opi_1$};
\draw[->] (A2) -- (B1);
\node at (\hp*2.35, \vp*-0.4) {\scriptsize  $\cdot\opi_2$};
\draw[->] (A2) -- (B3);
\node at (\hp*3, \vp*-0.4) {\scriptsize  $\cdot\opi_1$};
\draw[->] (A2) -- (B4);
\node at (\hp*1, \vp*-0.225) {\scriptsize  $\cdot\opi_3$};
\draw[->] (B1) -- (C1);
\node at (\hp*-0.6, \vp*-1.6) {\scriptsize  $\cdot\opi_1$};
\draw[->] (B2) -- (C1);
\node at (\hp*0.6, \vp*-1.6) {\scriptsize  $\cdot\opi_4$};
\draw[->] (B4) -- (C1);
\node at (\hp*2.1, \vp*-1.6) {\scriptsize  $\cdot\opi_3$};

\node at (\hp*2,\vp*-2.6) {$\DHTMP{\overline{P}^*} (\cong \upphi[\DHTMP{P}])$};
\end{tikzpicture}
\
\begin{tikzpicture}
\def \hp {4.5em}
\def \vp {5em}
\node at (\hp*0, \vp*0) (A1) {$25134$};
\node[right,xshift=\hp*0.15] at (A1) {} edge [out=40,in=320, loop] ();
\node[xshift=\hp*0.7] at (A1) {\scriptsize $\cdot \pi_2$};

\node at (\hp*2, \vp*0) (A2) {$51234$};
\node[right,xshift=\hp*0.15] at (A2) {} edge [out=40,in=320, loop] ();
\node[xshift=\hp*0.7] at (A2) {\scriptsize $\cdot \pi_1$};

\node at (\hp*-1, \vp*-1) (B1) {$52134$};
\node[right,xshift=\hp*0.15] at (B1) {} edge [out=40,in=320, loop] ();
\node[xshift=\hp*0.7] at (B1) {\scriptsize 
$\begin{array}{l} 
\cdot \pi_1, \\
\cdot \pi_2
\end{array}$};

\node at (\hp*1, \vp*-1) (B2) {$25143$};
\node[right,xshift=\hp*0.15] at (B2) {} edge [out=40,in=320, loop] ();
\node[xshift=\hp*0.7] at (B2) {\scriptsize 
$\begin{array}{l} 
\cdot \pi_2, \\
\cdot \pi_4
\end{array}$};

\node at (\hp*2.5, \vp*-1) (B3) {$51324$};
\node[right,xshift=\hp*0.15] at (B3) {} edge [out=40,in=320, loop] ();
\node[xshift=\hp*0.7] at (B3) {\scriptsize 
$\begin{array}{l} 
\cdot \pi_1, \\
\cdot \pi_3
\end{array}$};

\node at (\hp*4, \vp*-1) (B4) {$51243$};
\node[right,xshift=\hp*0.15] at (B4) {} edge [out=40,in=320, loop] ();
\node[xshift=\hp*0.7] at (B4) {\scriptsize 
$\begin{array}{l}
\cdot \pi_1, \\
\cdot \pi_4
\end{array}$};

\node at (\hp*0, \vp*-2) (C1) {$52143$};
\node[right,xshift=\hp*0.15] at (C1) {} edge [out=40,in=320, loop] ();
\node[xshift=\hp*0.7] at (C1) {\scriptsize 
$\begin{array}{l}
\cdot \pi_1,\\
\cdot \pi_2, \\
\cdot \pi_4
\end{array}$
};

\draw[->] (A1) -- (B1);
\node at (\hp*-0.6, \vp*-0.4) {\scriptsize  $\cdot \pi_1$};
\draw[->] (A1) -- (B2);
\node at (\hp*0.6, \vp*-0.4) {\scriptsize  $\cdot \pi_4$};
\draw[->] (A2) -- (B1);
\node at (\hp*2.35, \vp*-0.4) {\scriptsize  $\cdot \pi_3$};
\draw[->] (A2) -- (B3);
\node at (\hp*3, \vp*-0.4) {\scriptsize  $\cdot \pi_4$};
\draw[->] (A2) -- (B4);
\node at (\hp*1, \vp*-0.225) {\scriptsize  $\cdot \pi_2$};
\draw[->] (B1) -- (C1);
\node at (\hp*-0.6, \vp*-1.6) {\scriptsize  $\cdot \pi_4$};
\draw[->] (B2) -- (C1);
\node at (\hp*0.6, \vp*-1.6) {\scriptsize  $\cdot \pi_1$};
\draw[->] (B4) -- (C1);
\node at (\hp*2.1, \vp*-1.6) {\scriptsize  $\cdot \pi_2$};

\node at (\hp*2,\vp*-2.6) {
$\ourMP{P}(\cong \uptheta[\DHTMP{P}])$};
\end{tikzpicture}
$}
\\
\scalebox{0.7}{$
\begin{tikzpicture}
\def \hp {4.5em}
\def \vp {5em}
\node at (\hp*0, \vp*-2) (A1) {$41532$};
\node[right,xshift=\hp*0.15] at (A1) {} edge [out=40,in=320, loop] ();
\node[xshift=\hp*0.7] at (A1) {\scriptsize $\begin{array}{l} 
\cdot \pi_1, \\
\cdot \pi_3, \\
\cdot \pi_4
\end{array}$};

\node at (\hp*2, \vp*-2) (A2) {$15432$};
\node[right,xshift=\hp*0.15] at (A2) {} edge [out=40,in=320, loop] ();
\node[xshift=\hp*0.7] at (A2) {\scriptsize 
$\begin{array}{l} 
\cdot \pi_2, \\
\cdot \pi_3, \\
\cdot \pi_4
\end{array}$};

\node at (\hp*-1, \vp*-1) (B1) {$14532$};
\node[right,xshift=\hp*0.15] at (B1) {} edge [out=40,in=320, loop] ();
\node[xshift=\hp*0.7] at (B1) {\scriptsize 
$\begin{array}{l} 
\cdot \pi_3, \\
\cdot \pi_4
\end{array}$};

\node at (\hp*1, \vp*-1) (B2) {$41523$};
\node[right,xshift=\hp*0.15] at (B2) {} edge [out=40,in=320, loop] ();
\node[xshift=\hp*0.7] at (B2) {\scriptsize 
$\begin{array}{l} 
\cdot \pi_1, \\
\cdot \pi_3
\end{array}$};

\node at (\hp*2.5, \vp*-1) (B3) {$15342$};
\node[right,xshift=\hp*0.15] at (B3) {} edge [out=40,in=320, loop] ();
\node[xshift=\hp*0.7] at (B3) {\scriptsize 
$\begin{array}{l} 
\cdot \pi_2, \\
\cdot \pi_4
\end{array}$};

\node at (\hp*4, \vp*-1) (B4) {$15423$};
\node[right,xshift=\hp*0.15] at (B4) {} edge [out=40,in=320, loop] ();
\node[xshift=\hp*0.7] at (B4) {\scriptsize 
$\begin{array}{l}
\cdot \pi_2, \\
\cdot \pi_3
\end{array}$};

\node at (\hp*0, \vp*0) (C1) {$14523$};
\node[right,xshift=\hp*0.15] at (C1) {} edge [out=40,in=320, loop] ();
\node[xshift=\hp*0.7] at (C1) {\scriptsize 
$\begin{array}{l}
\cdot \pi_3
\end{array}$
};

\draw[<-] (A1) -- (B1);
\node at (\hp*-0.6, \vp*-1.6) {\scriptsize  $\cdot \pi_1$};
\draw[<-] (A1) -- (B2);
\node at (\hp*0.6, \vp*-1.6) {\scriptsize  $\cdot \pi_4$};
\draw[<-] (A2) -- (B1);
\node at (\hp*1, \vp*-1.6) {\scriptsize  $\cdot \pi_2$};
\draw[<-] (A2) -- (B3);
\node at (\hp*2.05, \vp*-1.6) {\scriptsize  $\cdot \pi_3$};
\draw[<-] (A2) -- (B4);
\node at (\hp*3, \vp*-1.6) {\scriptsize  $\cdot \pi_4$};
\draw[<-] (B1) -- (C1);
\node at (\hp*-0.6, \vp*-0.4) {\scriptsize  $\cdot \pi_4$};
\draw[<-] (B2) -- (C1);
\node at (\hp*0.6, \vp*-0.4) {\scriptsize  $\cdot \pi_1$};
\draw[<-] (B4) -- (C1);
\node at (\hp*2., \vp*-0.4) {\scriptsize  $\cdot \pi_2$};
\node at (\hp*2,\vp*-2.6) {$\ourMP{\overline{P}}(\cong \upchi[\DHTMP{P}])$};
\end{tikzpicture}
$}
\caption{The $H_5(0)$-actions on $\DHTMP{\overline{P}^*}$, $\ourMP{P}$, and $\ourMP{\overline{P}}$}
\label{fig: H_5(0)-actions for twists}
\end{figure}

Let $m, n \in \N_0$.
For any (anti-)automorphism $\zeta: H_{m+n}(0) \ra H_{m+n}(0)$ and an $H_m(0) \otimes H_n(0)$-module $M$, we denote $\zeta|_{H_m(0) \otimes H_n(0)}[M]$ by $\zeta[M]$.
We close this section by providing the following corollary.
\begin{corollary}\label{Coro: Automorphis of P}
{\rm (a)} Let $P_1 \in \poset{m}$ and $P_2 \in \poset{n}$.
As $H_{m+n}(0)$-modules,
\[
\begin{array}{c}
\upphi[M_{P_1} \boxtimes M_{P_2}] \cong M_{\overline{P_2}^*} \boxtimes M_{\overline{P_1}^*}, 
\ \
\uptheta[M_{P_1} \boxtimes M_{P_2}] \cong \ourMP{P_1} \boxtimes \ourMP{P_2}, \ \
\\[1.5ex]
\text{and} \quad
\upchi[M_{P_1} \boxtimes M_{P_2}] \cong \ourMP{\overline{P_2}} \boxtimes \ourMP{\overline{P_1}}.
\end{array}
\]

{\rm (b)} Let $P \in \poset{m+n}$. 
As $H_{m}(0) \otimes H_{n}(0)$-modules,
\[
\begin{array}{c}
\upphi[M_{P}\downarrow^{H_{m+n}(0)}_{H_{n}(0) \otimes H_{m}(0)}] \cong M_{\overline{P}^*}\downarrow^{H_{m+n}(0)}_{H_{m}(0) \otimes H_{n}(0)},
\quad
\uptheta[M_{P}\downarrow^{H_{m+n}(0)}_{H_{m}(0) \otimes H_{n}(0)}] \cong 
\ourMP{P}\downarrow^{H_{m+n}(0)}_{H_{m}(0) \otimes H_{n}(0)},
\\[1.5ex]
\text{and} \quad
\upchi[M_{P}\downarrow^{H_{m+n}(0)}_{H_{m}(0) \otimes H_{n}(0)}] \cong 
\ourMP{\overline{P}} \downarrow^{H_{m+n}(0)}_{H_{m}(0) \otimes H_{n}(0)}.
\end{array}
\]
\end{corollary}
\begin{proof}
Combining \cref{thm: auto-twists}(a) with~\cite[Theorem 4.4]{22JO} yields the assertions (a) and (b).
\end{proof}

\section{Quasisymmetric power sum expansions I: A theoretical method}
\label{Sec: Quasisymmetric power sum expansions I}

\subsection{The bases for \texorpdfstring{$\Qsym$}{Lg} in consideration and our method}
\label{The bases for Qsym in consideration and our method}
We start by introducing the bases for $\Qsym$ that will be covered in this section.
For each $\alpha \models n$, let $\balpha$ be one of the following quasisymmetric functions: 

\begin{enumerate}[label = {$\bullet$}]
\item
the quasisymmetric Schur function $\QS{\alpha}$~(\cite{11HLMW})

\item
the dual immaculate function $\DIF{\alpha}$~(\cite{14BBSSZ})

\item
the extended Schur function $\ESF{\alpha}$~(\cite{19AS})

\item 
the images of $\QS{\alpha}, \DIF{\alpha}, \ESF{\alpha}$ under 
the involutions $\rho$, $\psi$, $\rho \circ \psi$ on $\Qsym$ 
\end{enumerate}
These functions are quite significant in that they were introduced as quasisymmetric analogues of the Schur function and  
$\{\balpha \mid \alpha \models n\}$ is a $\mathbb Z$-basis for $\Qsym_n$.
The purpose of this section is to present a theoretical method to expand $\balpha$ in the quasisymmetric power sum basis $\{\Psi_\alpha: \alpha \models n\}$.

In the past few years, many noteworthy results on quasisymmetric power sums have been obtained (for instance, see \cite{21AS}, \cite{23AWW}, \cite{20BDHMN}, \cite{22A}, \cite{21LW}).
Among them, the present paper heavily depends on the result of Liu--Weselcouch~\cite{21LW}, which tells us how to expand  
the $\pomega$-partition generating function $\KPw$ into quasisymmetric power sums
for an arbitrary finite labeled poset $(P, \omega)$.
To be precise, our method can be performed in the following steps.

\begin{roadmap} 
Let $\alpha$ be a composition of $n$. 
\begin{enumerate}[label = {\it Step \arabic*.}]
\item
We observe that    
$\balpha$ appears as the image of some left $H_n(0)$-module, say $\mbalpha$, under the left quasisymmetric characteristic.
In fact, $\mbalpha$ is constructed using the appropriate tableaux.  

\item
We observe that every indecomposable summand of $\mbalpha$ is equipped with the structure of a left weak Bruhat interval module introduced in \cite{22JKLO}.

\item
Modifying the construction in \cite{22JKLO}, we introduce right weak Bruhat interval modules and show that these are all poset modules.

\item
We construct a contravariant functor 
from the category of left $H_n(0)$-modules to that of right $H_n(0)$-modules
\[
\calF_n: \Rmod \ra \modR
\]
that preserves the quasisymmetric characteristic and sends a left weak Bruhat interval module to a right weak Bruhat interval module. 

\item
From the result of Duchamp--Hivert--Thibon~\cite{02DHT} 
we derive that $\balpha$ appears as 
the $P$-partition generating function of a poset $P\in \poset{n}$.  
Then, apply the result of Liu--Weselcouch~\cite{21LW} to expand 
$Y_\alpha$ in the quasisymmetric power sum basis.
\end{enumerate}
\end{roadmap}

\subsection{\texorpdfstring{{\it Step 1 $\&$ Step 2}}{Lg}}
\label{Subsec: Step 1 an Step 2}
For a detailed description of {\it Step 1}, we first introduce the left irreducible modules (up to isomorphism), the
Grothendieck group of $\Rmod$ (=the category of left $H_n(0)$-modules), and 
the left quasisymmetric characteristic.

For $\alpha \models n$, the left irreducible $H_n(0)$-module, denoted by  $\simFL{\alpha}$, 
corresponding to $\alpha$ is defined to be $\C v_{\alpha}$ endowed with the following left $H_n(0)$-action: for $1 \le i \le n-1$, 
\[
\pi_i \cdot v_\alpha = \begin{cases}
0 & i \in \set(\alpha),\\
v_\alpha & i \notin \set(\alpha).
\end{cases}
\]

Let $\calR_L(\Rmod)$ denote the $\Z$-span of the isomorphism classes of finite dimensional left $H_n(0)$-modules. 
The \emph{Grothendieck group} $G_0(\Rmod)$ can be defined in the same way as $G_0(\Rmod)$ and 
the nonisomorphic irreducible $H_n(0)$-modules induce a free $\Z$-basis for $G_0(\Rmod)$.
Let
\[
\calG(\Rbmod) := \bigoplus_{n \ge 0} G_0(\Rmod).
\]
Then the map 
\begin{equation}\label{quasi characteristic left}
\ch_L : \calG(\Rbmod) \ra \Qsym, \quad [\bfF^L_{\alpha}] \mapsto F_{\alpha},
\end{equation}
is also a Hopf algebra isomorphism.
From now on, for clarity, we will refer to $\ch_L$ as the \emph{left quasisymmetric characteristic}
and $\ch_R$ in \cref{quasi characteristic right} as the \emph{right quasisymmetric characteristic}.

For all of $Y_\alpha$'s introduced in the preceding subsection, 
certain left $H_n(0)$-modules $\mbalpha$ with $\ch_L(\mbalpha)=Y_\alpha$ have been constructed 
using appropriate tableaux. The complete list can be shown in \cref{Table: our consideration}.

\begin{table}[t]	
\centering
\tabulinesep=1.2mm
\small
\begin{tabu}{c|c|c|c}
$\balpha$ &  $\mbalpha$ & $\begin{array}{c}
\text{Decomposition} \\ \text{into indecomposables}
\end{array}$
& Basis of $\mathbf{Y}_\alpha$
\\ \hline \hline
$\displaystyle \QS{\alpha}$ (\cite{11HLMW})
& $\mQS{\alpha}$ (\cite{15TW}) 
& $\mQS{\alpha} = \bigoplus_E \mQS{\alpha,E}$ (\cite{15TW,19Konig})
&
$\substack{\text{standard reverse} \\ \text{composition tableaux}}$ 
\\ \hline
$\YQS{\alpha} = \rho(\QS{\alpha^\rmr})$ (\cite{13LMvW})
&  
$\mYQS{\alpha} 
\cong \upphi[\mQS{\alpha^\rmr}]$ (\cite{22CKNO1})
& $\mYQS{\alpha} = \bigoplus_E \mYQS{\alpha,E}$ 
(\cite{22CKNO1})
&
$\substack{\text{standard Young} \\ \text{composition tableaux}}$ 
\\ \hline
$\YRQS{\alpha} = \psi(\QS{\alpha})$ (\cite{14MR})
&  
$\mYRQS{\alpha} := \uptheta \circ \upchi[\mQS{\alpha}]$ (\cite{22JKLO})
& $\mYRQS{\alpha} = \bigoplus_E \mYRQS{\alpha,E}$ (\cite{22JKLO})
&
$\substack{\text{standard reverse} \\ \text{composition tableaux}}$ 
\\ \hline
$\RQS{\alpha} =  \rho \circ \psi(\QS{\alpha^\rmr})$ (\cite{15MN})   
& $\mRQS{\alpha} \cong \upphi \circ \uptheta \circ \upchi [\mQS{\alpha^\rmr}]$  (\cite{22BS})
& $\mRQS{\alpha} = \bigoplus_E \mRQS{\alpha,E}$ 
(\cite{22JKLO})
&
$\substack{\text{standard Young} \\ \text{row-strict tableaux}}$ 
\\ \tabucline[1.1pt]{-}
$\DIF{\alpha}$ (\cite{14BBSSZ})
& $\mDIF{\alpha}$ (\cite{15BBSSZ})
& $\mDIF{\alpha}$ (\cite{15BBSSZ})
& $\substack{\text{standard dual} \\ \text{immaculate tableaux}}$ 
\\ \hline
$\RDIF{\alpha} = \psi(\DIF{\alpha})$ (\cite{22NSvWVW})
&  $\mRDIF{\alpha} \cong \uptheta \circ \upchi[\mDIF{\alpha}]$ (\cite{22NSvWVW})
& $\mRDIF{\alpha}$ (\cite{22NSvWVW})
& $\substack{\text{standard dual} \\ \text{ immaculate tableaux}}$ 
\\ \tabucline[1.1pt]{-}
$\ESF{\alpha}$ (\cite{19AS})
& $\mESF{\alpha}$ (\cite{19Searles})
& $\mESF{\alpha}$ (\cite{19Searles})
& $\substack{\text{standard} \\ \text{extended tableaux}}$ 
\\ \hline
$\calR\calE_\alpha = \psi(\calE_\alpha)$ (\cite{22NSvWVW})
& $\mRESF{\alpha} \cong \uptheta \circ \upchi[\mESF{\alpha}]$ (\cite{22NSvWVW})
& $\mRESF{\alpha}$ (\cite{22NSvWVW})
& $\substack{\text{standard} \\ \text{extended tableaux}}$ 
\end{tabu}
\caption{The $0$-Hecke module $\mbalpha$ corresponding to $\balpha$ and its decomposition into indecomposables, where $\ch_L(\mbalpha)=\balpha$.}
\label{Table: our consideration}
\end{table}

Next, let us explain {\it Step 2}. 
In~\cite{22JKLO}, Jung--Kim--Lee--Oh introduced the left weak Bruhat interval modules to deal with $\mbalpha$'s, more precisely, 
the indecomposable direct summands of $\mbalpha$'s 
in \cref{Table: our consideration} in a unified way.
\footnote{
The modules $\calR\mDIF{\alpha}$ and $\calR \mESF{\alpha}$ in \cite{22NSvWVW} 
were constructed by
defining new  $H_n(0)$-actions on the tableau-bases for $\mDIF{\alpha}$ and $\mESF{\alpha}$, respectively.
These modules, however, can be easily obtained 
by taking some involution twists introduced in \cite{22JKLO} on $\mDIF{\alpha}$ and 
$\mESF{\alpha}$.
To be precise, $\calR\mDIF{\alpha}$ and $\calR \mESF{\alpha}$ can be recovered as the $\uptheta \circ \upchi$-twists of $\mDIF{\alpha}$ and $\mESF{\alpha}$, respectively.  
In fact, various involution twists of all left Bruhat interval modules can be found in \cite[Section 3.4]{22JKLO}.}

\begin{definition}{\rm (\cite[Definition 1]{22JKLO})} \label{left WBI}
Let $\sigma, \rho \in \SG_n$ with $\sigma \preceq_L \rho$.
The \emph{left weak Bruhat interval module associated with $[\sigma,\rho]_L$}, denoted by $\sfB_L(\sigma,\rho)$, is the left $H_n(0)$-module with  $\C[\sigma,\rho]_L$ as the underlying space and with the $H_n(0)$-action defined by
\begin{equation*}\label{Hecke algebra action: left and pi}
\pi_i \cdot \gamma  := 
\begin{cases}
\gamma & \text{if $i \in \Des{L}{\gamma}$}, \\
0 & \text{if $i \notin \Des{L}{\gamma}$ and $s_i\gamma \notin [\sigma,\rho]_L$,} \\
s_i \gamma & \text{if $i \notin \Des{L}{\gamma}$ and $s_i\gamma \in [\sigma,\rho]_L$}
\end{cases} 
\end{equation*}
for $1 \leq i \leq n-1$ and $\gamma \in [\sigma,\rho]_L$.
\end{definition}

We denote by $\osfB_L(\sigma,\rho)$ the $\uptheta$-twist of $\sfB_L(\sigma,\rho)$, which is called the negative weak Bruhat interval module associated with $[\sigma,\rho]_L$ in \cite[Definition 1]{22JKLO}.
One sees that it is the left $H_n(0)$-module with  $\C[\sigma,\rho]_L$ as the underlying space and with the $H_n(0)$-action defined by
\begin{equation*}\label{Hecke algebra action: left and pi bar}
\opi_i \star \gamma  := 
\begin{cases}
-\gamma & \text{if $i \in \Des{L}{\gamma}$}, \\
0 & \text{if $i \notin \Des{L}{\gamma}$ and $s_i\gamma \notin [\sigma,\rho]_L$,} \\
s_i \gamma & \text{if $i \notin \Des{L}{\gamma}$ and $s_i\gamma \in [\sigma,\rho]_L$}
\end{cases} 
\end{equation*}
for $1 \leq i \leq n-1$ and $\gamma \in [\sigma,\rho]_L$.

Quite interestingly, every indecomposable direct summand $V$ of $\mbalpha$ appearing 
in \cref{Table: our consideration} turns out to have 
the structure of a left weak Bruhat interval module.
To be precise, Jung--Kim--Lee--Oh provided a left weak Bruhat interval $[\sigma_V, \rho_V]_L \subseteq \SG_n$ satisfying the following conditions:
\begin{enumerate}[label = {\rm (\alph*)}]
\item 
$V$ is isomorphic to $\sfB_L(\sigma_V, \rho_V)$ as a left $H_n(0)$-module.
\item 
$\sigma_V$ is of the form $w_0(\beta)$ for some $\beta \models n$.
\end{enumerate}
For the precise description of $\sigma_V$ and $\rho_V$, see \cref{Lemma: V and X are interval}  and \cref{Lemma: S is interval}.

\begin{remark}\label{correction of WBI paper}
It was stated in the introduction of \cite{22JKLO} that 
the category $\bigoplus_{n \ge 0} \mathcal{B}_n$ is closed under induction product, restriction coproduct, and (anti-)involution twists, where $\mathcal{B}_n$ is the full subcategory of $\Rmod$ whose objects are direct sums of finitely many isomorphic copies of left weak Bruhat interval modules and negative left weak Bruhat interval modules.
However, it is incorrect because the induction product of a left weak Bruhat interval module and a negative left weak Bruhat interval module is not an object in $\bigoplus_{n \ge 0} \mathcal{B}_n$.
In fact, $\mathcal{B}_n$ should appear as the full subcategory of $\Rmod$ whose objects are direct sums of finitely many isomorphic copies of only left weak Bruhat interval modules.
\end{remark}

\subsection{\texorpdfstring{{\it Step 3 $\&$ Step 4}}{Lg}}
\label{Step 3 and Step 4}
The first thing we need to do in step 3 is to introduce the right module version of \cref{left WBI}.  
\begin{definition}
Let $\sigma, \rho \in \SG_n$ with $\sigma \preceq_R \rho$.
The \emph{right weak Bruhat interval module associated with $[\sigma,\rho]_R$}, denoted by $\osfB_R(\sigma,\rho)$, is the right $H_n(0)$-module with $\C[\sigma,\rho]_R$ as the underlying space and with the $H_n(0)$-action defined by
\begin{align*}
\gamma \cdot \opi_i= 
\begin{cases}
-\gamma & \text{if } i \in \Des{R}{\gamma},\\
0 & \text{if $i \notin \Des{R}{\gamma}$ and $\gamma s_i \notin [\sigma,\rho]_R$},\\
\gamma s_i & \text{if } i \notin \Des{R}{\gamma} \text{ and } \gamma s_i \in [\sigma,\rho]_R
\end{cases}
\end{align*}
for $1 \leq i \leq n-1$ and $\gamma \in [\sigma,\rho]_R$.
\end{definition}

We denote by $\sfB_R(\sigma,\rho)$ the $\uptheta$-twist of $\osfB_R(\sigma,\rho)$.
One sees that it is the right $H_n(0)$-module with  $\C[\sigma,\rho]_R$ as the underlying space and with the $H_n(0)$-action defined by
\begin{align*}
\gamma \cdot \pi_i= 
\begin{cases}
\gamma & \text{if } i \in \Des{R}{\gamma},\\
0 & \text{if $i \notin \Des{R}{\gamma}$ and $\gamma s_i \notin [\sigma,\rho]_R$},\\
\gamma s_i & \text{if } i \notin \Des{R}{\gamma} \text{ and } \gamma s_i \in [\sigma,\rho]_R
\end{cases}
\end{align*}
for $1 \leq i \leq n-1$ and $\gamma \in [\sigma,\rho]_R$.

In the following, we will show that every right weak Bruhat interval module is a poset module.
Indeed, this can be justified using Bj\"{o}rner--Wachs's result in~\cite[Section 6]{91BW}, which characterizes the posets 
$P \in \poset{n}$ such that $\SGR{P}$ is a right weak Bruhat interval.
Before stating their result, we introduce the necessary definitions and notations.
Given $P \in \poset{n}$,
we say that $P$ is {\em regular} if there is no triple $(u,v,w)$ such that 
\begin{align}\label{Def: regular condition of P}
\begin{aligned}
& v \preceq_P w, \qquad 
(u,v) , (u,w) \in \inc{P}, 
\quad \text{ and } \quad 
w < u < v, \quad \text{ or}, \\
& v \preceq_P w, \qquad 
(u,v), (u,w) \in \inc{P}, \quad \text{ and } \quad 
v < u < w.
\end{aligned}
\end{align}
Let $\calR_{\rm st}(P) := \{(x,y) \in [n]^2 \mid x \preceq_P y, x \neq y \}$.
The {\em dimension} of $P$ is defined by
\[
\min \left\{ |U| \; \middle| \; U \subseteq \SGR{P} \text{ and } \bigcap_{\sigma \in U} \left( \inv{\sigma} \cup \coinv{\sigma} \right) = \calR_{\rm st}(P) \right\}.
\]

\begin{lemma}{\rm (\cite[Theorem 6.8]{91BW})}\label{Thm: BW equivalent}
Let $U \subseteq \SG_n$ with $|U| > 1$. Then the following are equivalent.
\begin{enumerate}[label = {\rm (\alph*)}]
    \item $U \subseteq \SG_n$ is a right weak Bruhat interval.

    \item $U = \SGR{P}$ for some regular poset $P$.
    
    \item $U = \SGR{P}$ for some two-dimensional regular poset $P$.
\end{enumerate}
\end{lemma}

Note that if $P,P' \in \poset{n}$ are different, then $\SGR{P} \neq \SGR{P'}$.
Therefore \cref{Thm: BW equivalent} implies that for any right weak Bruhat interval $I$ in $\SG_n$, there exists a unique poset $P_I \in \poset{n}$ such that $\SGR{P_I} = I$.
It should be emphasized that the proof of \cite[Theorem 6.8]{91BW} shows not only the existence of $P_I$, but also how to construct it.

\begin{proposition}\label{Prop: Interval module is poset module}
Every right weak Bruhat interval module is a poset module. 
\end{proposition}
\begin{proof}
Let $I$ be a right weak Bruhat interval of $\mathfrak S_n$.
Since $\SGR{P_I} = I$, $\osfB_R(I)$ and $\DHTMP{P_I}$ share the same basis $I$. Therefore the assertion is immediate from the fact that the $H_n(0)$-action on $I$ in $\osfB_R(I)$ is identical to the $H_n(0)$-action on $\SGR{P_I}$ in $\DHTMP{P_I}$.
\end{proof}

\begin{remark}
The converse of \cref{Prop: Interval module is poset module} may not be true.
More generally, the poset module is not necessarily isomorphic to the direct sum of some right weak Bruhat interval modules.
For example, let 
$P =  
\begin{tikzpicture}[baseline=5mm,xscale=1.3]
\node[shape=circle,draw,minimum size=\ccc, inner sep=0pt] at (-\hp*0.3, 0) (A1) {\tiny $1$};
\node[shape=circle,draw,minimum size=\ccc, inner sep=0pt] at (\hp*1, \vp*2) (A2) {\tiny $2$};
\node[shape=circle,draw,minimum size=\ccc, inner sep=0pt] at (\hp*1, \vp*1) (A3) {\tiny $3$};
\node[shape=circle,draw,minimum size=\ccc, inner sep=0pt] at (\hp*2, \vp*2) (A4) {\tiny $4$};
\node[shape=circle,draw,minimum size=\ccc, inner sep=0pt] at (\hp*2, \vp*1) (A5) {\tiny $5$};

\draw (A2) -- (A1) -- (A5);
\draw (A3) -- (A4);
\draw[line width = \lw] (A2) -- (A3); 
\draw[line width = \lw] (A5) -- (A4); 
\end{tikzpicture}$.
We note that the projective cover of $\DHTMP{P}$ is the projective indecomposable module $\calP^R_{(2,2,1)}$, thus $\DHTMP{P}$ is indecomposable.
We also note that $\dim \DHTMP{P} = 8$ and $$\ch_R([\DHTMP{P}]) = F_{(3,2)} + F_{(3,1,1)}
 + 2F_{(2,2,1)} + F_{(2,1,2)} + F_{(1,3,1)} + F_{(1,2,2)} + F_{(1,2,1,1)}.$$
Using \textsc{SageMath}, the authors verified that there are no permutations $\sigma, \rho \in \SG_5$ satisfying that $\dim \osfB_R(\sigma,\rho) = 8$ and $\ch_R([\osfB_R(\sigma,\rho)]) = \ch_R([\DHTMP{P}])$.
\end{remark}

Next, let us deal with {\it Step 4},
which concerns a functor from $\Rmod$ to $\modR$ that preserves the quasisymmetric characteristic and sends a left weak Bruhat interval module to a right weak Bruhat interval module.
For $n \in \N_0$, we consider the contravariant functor 
\[
\calF_n: \Rmod \ra \modR
\]
which assigns to each object $M$ the dual space $M^*:= \Hom(M,\C)$ endowed with the right $H_n(0)$-action given by
\begin{equation*}
(\phi \cdot \opi_\sigma)(v) = \phi(\opi_{\sigma^{-1}} \cdot v) \quad \text{ for } \phi \in M^* \text{ and } v \in M,
\end{equation*}
and to each $H_n(0)$-module homomorphism $h:M \ra N$ the dual $\C$-homomorphism $\calF_n(h) = \Hom_\C(h,K): \calF_n(N) \ra \calF_n(M)$ given by $\phi \mapsto \phi \circ h$.
Then, consider the map 
\[
\ocalF: \calG(\Rbmod) \ra  \calG(\modRb), \quad [M] \mapsto [\calF_n(M)],
\]
where $M$ is a left $H_n(0)$-module.
It is not difficult to show that it is a Hopf algebra isomorphism.
With this preparation, we can state the following proposition.
\begin{proposition}\label{Prop: B and bB}

{\rm (a)} 
Let $\sigma, \rho \in \SG_n$ with $\sigma \preceq_L \rho$.
As right $H_n(0)$-modules,
\begin{equation*}
\calF_n(\sfB_L(\sigma,\rho)) \cong \osfB_R(w_0\rho^{-1}, w_0\sigma^{-1}).
\end{equation*}

{\rm (b)}
For every left $H_n(0)$-module $M$,
\[
\ch_L([M]) =\ch_R \circ \ocalF([M]).
\]
\end{proposition}
\begin{proof}
(a) Let us consider the linear map 
\[
f: \calF_n(\sfB_L(\sigma,\rho))
\ra  \osfB_R(w_0\rho^{-1}, w_0\sigma^{-1}), \quad \gamma^* \mapsto w_0 \gamma^{-1},
\]
where $\gamma \in [\sigma,\rho]_L$ and $\gamma^*$ denotes the dual of $\gamma$ with respect to the basis $[\sigma, \rho]_L$ for $\sfB_L(\sigma,\rho)$.
Since $\gamma \in [\sigma,\rho]_L$ if and only if $w_0 \gamma^{-1} \in [w_0 \rho^{-1}, w_0 \sigma^{-1}]_R$, $f$ is bijective.
We claim that $f$ is a right $H_n(0)$-module homomorphism.

Let $1 \leq i \le n-1$ and $\gamma \in [\sigma,\rho]_L$.
By the definition of $\calF_n$,
\[
 \gamma^* \star \opi_i 
= \begin{cases}
-\gamma^* & \text{if $i \notin \Des{L}{\gamma}$,} \\
0 & \text{if $i \in \Des{L}{\gamma}$ and $s_i\gamma \notin [\sigma, \rho]_L$,} \\
(s_i \gamma)^* & \text{if $i \in \Des{L}{\gamma}$ and $s_i\gamma \in [\sigma, \rho]_L$},
\end{cases}
\]
where $\star$ denotes the the right $H_n(0)$-action on $\calF_n(\sfB_L(\sigma,\rho))$.
Thus,
\begin{align}\label{eq: f(opi action on gamma*)}
f(\gamma^* \star \opi_i) 
= \begin{cases}
-w_0 \gamma^{-1} & \text{if $i \notin \Des{L}{\gamma}$,} \\
0 & \text{if $i \in \Des{L}{\gamma}$ and $s_i\gamma \notin [\sigma, \rho]_L$,} \\
w_0 \gamma^{-1} s_i & \text{if $i \in \Des{L}{\gamma}$ and $s_i\gamma \in [\sigma, \rho]_L$.}
\end{cases}
\end{align}
On the other hand, we have
\begin{align}\label{eq: opi action on f(gamma*)}
f(\gamma^*) \cdot \opi_i 
= \begin{cases}
-w_0 \gamma^{-1} & \text{if $i \in \Des{R}{w_0 \gamma^{-1}}$,} \\
 0 & \text{if $i \notin \Des{R}{w_0 \gamma^{-1}}$ and $w_0 \gamma^{-1} s_i \notin [w_0 \rho^{-1}, w_0 \sigma^{-1}]_R$,} \\
w_0 \gamma^{-1} s_i & \text{if $i \notin \Des{R}{w_0 \gamma^{-1}}$ and $w_0 \gamma^{-1} s_i \in [w_0 \rho^{-1}, w_0 \sigma^{-1}]_R$.}
\end{cases}
\end{align}
Note that $i \notin \Des{L}{\gamma}$ if and only if $i \in \Des{R}{w_0 \gamma^{-1}}$.
Combining this with \cref{eq: f(opi action on gamma*)} and \cref{eq: opi action on f(gamma*)} shows that $f$ is a right $H_n(0)$-module isomorphism. 

(b) In \cref{Subsec: Step 1 an Step 2}, we remarked that the quasisymmetric characteristic
\begin{equation*}
\ch_L : \calG(\Rbmod) \ra \Qsym, \quad [\bfF^L_{\alpha}] \mapsto F_{\alpha}.
\end{equation*}
is a Hopf algebra isomorphism.
And, from (a) it follows that  $\calF_n(\bfF^L_\alpha) = \bfF^R_\alpha$. 
Consequently, we derive that  
\[
\ch_R([\calF_n(\bfF^L_\alpha)]) = F_\alpha = \ch_L([\bfF^L_\alpha])
\]
for every $\alpha \models n$.
\end{proof}

The rest of this subsection is devoted to investigating the compatibility of $\calF_n$ and the $\upchi$-twists on weak Bruhat interval modules.
In \cref{Sec: automorphism twists}, we explained that any (anti-)automorphism $f$ of $H_n(0)$ induces an endofunctor on $\modR$, called the $f$-twist, and denoted by $\mathbf{T}^{+}_{f}$ if $f$ is an automorphism and $\mathbf{T}^{-}_{f}$ otherwise.
In the same manner, one can define a similar endofunctor on $\Rmod$ without difficulty. 
For simplicity of notation, this functor is also written as $\mathbf{T}^{\pm}_{f}$.

Recall that $\osfB_L(\sigma,\rho)$ and $\sfB_R(\sigma,\rho)$ are 
defined as the $\theta$-twists of $\sfB_L(\sigma,\rho)$ and $\osfB_R(\sigma,\rho)$, respectively. However, $\mathbf{T}^{+}_{\uptheta}$ is rather unsatisfactory in that it does not preserve the quasisymmetric characteristic. 
Indeed, it holds that   $\ch_L([\mathbf{T}^{+}_{\uptheta}(\bfF^L_\alpha)])=\ch_R([\mathbf{T}^{+}_{\uptheta}(\bfF^R_{\alpha})])=F_{\alpha^\rmc}$.
Contrary to $\mathbf{T}^{+}_{\uptheta}$, $\mathbf{T}^{-}_{\upchi}$ 
preserves the quasisymmetric characteristic since $\ch_L([\mathbf{T}^{-}_{\upchi}(\bfF^L_\alpha)])=\ch_R([\mathbf{T}^{-}_{\upchi}(\bfF^R_\alpha)])=F_{\alpha}$.
Referring to \cref{correction of WBI paper}, 
we let 
\begin{itemize}
\item $\mathcal{B}^L_n$:= the full subcategory of $\Rmod$ whose objects are direct sums of finitely many isomorphic copies of left weak Bruhat interval modules,
    
\item ${\overline {\mathcal B}}^R_n$:= the full subcategory of $\modR$ whose objects are direct sums of finitely many isomorphic copies of right weak Bruhat interval modules,

\item
${\overline {\mathcal B}}^L_n$:= $\mathbf{T}^{+}_{\uptheta}(\mathcal{B}^L_n)$, \text{ and} 

\item
$\mathcal{B}^R_n$:= $\mathbf{T}^{+}_{\uptheta}({\overline {\mathcal B}}^R_n)$.
\end{itemize}
In \cite[Theorem 4(3)]{22JKLO}, it was shown that $\mathbf{T}^-_{\autochi}(\sfB_L(\sigma,\rho)) \cong \osfB_L(\rho w_0, \sigma w_0)$.
Similarly, one can easily see that $\mathbf{T}^-_{\autochi}(\osfB_R(\sigma,\rho)) \cong \sfB_R(w_0\rho, w_0\sigma)$.
Therefore we can derive the following four characteristic-preserving equivalences 
\begin{align*}
 & \calF_n| : \mathcal{B}^L_n \to  {\overline {\mathcal B}}^R_n, \qquad 
 \calF_n| : {\overline {\mathcal B}}^L_n \to  \mathcal{B}^R_n\\
 &\mathbf{T}^-_{\autochi}|: \mathcal{B}^L_n \to {\overline {\mathcal B}}^L_n, \qquad 
 \mathbf{T}^-_{\autochi}|: {\overline {\mathcal B}}^R_n \to \mathcal{B}^R_n
\end{align*}
such that the diagram 
\begin{equation*}\label{inductuion0}
\begin{CD}
\mathcal{B}^L_n @>{\calF_n|}>> {\overline {\mathcal B}}^R_n\\
@V{\mathbf{T}^-_{\autochi}|}VV @V{\mathbf{T}^-_{\autochi}|}VV\\
{\overline {\mathcal B}}^L_n @>{\calF_n|}>> \mathcal{B}^R_n
\end{CD}
\end{equation*}
is commutative.

\subsection{{\it Step 5}}
\label{Subsec: pomega-partitions}
Let $\balpha$ and $\mbalpha$ be one of the quasisymmetric functions and $H_n(0)$-modules in the list of~\cref{Table: our consideration}, respectively.
Let $\setIndSummandY$ denote the set of indecomposable direct summands of $\mbalpha$. 
We remarked in \cref{Subsec: Step 1 an Step 2} that each $V \in \setIndSummandY$ is isomorphic to $\sfB_L(\sigma_V, \rho_V)$.
Considering \cref{Prop: B and bB}, we set
$I_V := [w_0 \rho_V^{-1}, w_0 \sigma_V^{-1}]_R$.
For later use, consider the bijection \begin{align}\label{eq: bold f map}
\mapf: \SG_n \ra \SG_n, \quad   \gamma \mapsto w_0 \gamma^{-1}.
 \end{align}
Then it holds that $I_V = \mapf([\sigma_V,\rho_V]_L)$.

Through {\it Step 1}  to {\it Step 4}, we can derive that $\DHTMP{P_{I_{V}}}$ is isomorphic to $\calF_n(V)$.
For the simplicity of notation, we write $\posetFromAlgo{V}$ for $P_{I_{V}}$.
This, together with \cref{prop: DHT proposition}, implies that $\balpha$ appears as the sum of $P$-partition generating functions of certain posets $P$, more precisely, 
\begin{equation}\label{Eq: Yalpha = sum KP}
\balpha = \sum_{V \in \setIndSummandY} K_{\posetFromAlgo{V}}.
\end{equation}

In this subsection, applying Liu--Weselcouch's result \cite[Theorem 6.9]{21LW}, 
we obtain the quasisymmetric power sum expansion of $\balpha$.
Let us consider the totally ordered sets $\bfP' = (\{\pm i \mid i \in \N\},  \le_{\bfP'})$ and $\bfP^* = (\{\pm i, i^* \mid i \in \N\},  \le_{\bfP^*})$, where
\begin{align*}
& -1 <_{\bfP'} 1 <_{\bfP'} -2 <_{\bfP'} 2 <_{\bfP'} \cdots \ \text{ and }
\\
&-1 <_{\bfP^*} 1^* <_{\bfP^*} 1 <_{\bfP^*} -2 <_{\bfP^*} 2^* <_{\bfP^*} 2 <_{\bfP^*} \cdots.
\end{align*}
For $i \in \N$, set $|-i|:=i$, $|i|:= i$, and $|i^*| := i$.
Given a poset $P$, an {\em enriched $P$-partition} is a map $f: P \ra \bfP'$ satisfying the following conditions:
\begin{enumerate}[label = {\rm (\roman*)}]
\item 
$\{|f(x)| : x \in P\} = [k]$ 
for some $k$.
\item Let $x,y \in P$ with $x \preceq^{c}_P y$.
Then
\begin{enumerate}[label = $\bullet$]
\item $f(x) \leq_{\bfP'} f(y)$,
\item if $|f(x)| = |f(y)|$ and $x < y$, then $f(y) > 0$,
\item if $|f(x)| = |f(y)|$ and $x > y$, then $f(x) < 0$.
\end{enumerate}
\end{enumerate}

For $x \in P$ and an enriched $P$-partition $f$, define the map $f_x:P \ra \bfP'$ by
\[
f_x(y) = \begin{cases}
-f(x) & \text{if } x = y, \\
f(y) & \text{otherwise.}
\end{cases}
\]
We say that an element $x \in P$ is {\em ambiguous with respect to $f$} if $f_x$ is still an enriched $P$-partition.
For an enriched $P$-partition $f$, define the map  $f^*:P \ra \bfP^*$ by
\[
f^*(x) = \begin{cases}
|f(x)|^* & \text{if $x$ is ambiguous with respect to $f$},\\
f(x) & \text{otherwise.}
\end{cases}
\]
It is called a {\em starred $P$-partition}.
We need three statistics related to $f^*$.
The {\em ambiguity} of $f^*$ is defined by the sequence 
$$
{\rm amb}(f^*) := (a_1,a_2,\ldots,a_k),
$$ 
where $a_i=|(f^*)^{-1}(i^*)|$, the number of elements labeled $i^*$ by $f^*$. 
The {\em sign} of $f^*$ is defined by 
\[
{\rm sign}(f^*) := (-1)^{|\{x \mid -f^*(x) \in \N\}|}.
\]
Finally, the {\em weight} of $f^*$ is defined by the sequence 
$\wt(f^*) := (b_1,b_2,\ldots,b_k)$, where $b_i$ is the number of elements of labeled $-i,i^*$, or $i$ by $f$.
Given $P \in \poset{n}$ and $\beta \models n$, let
\begin{align}\label{eq: sfpt_P}
\mathsf{pt}_{P}(\beta) := 
\{ 
f^* \mid \text{$f^*$ is a starred $P$-partition with $\mathrm{amb}(f^*) = (1^{\ell(\beta)})$ and  $\wt(f^*) = \beta$}
\}.
\end{align}

\begin{lemma}{\rm (\cite[Theorem 6.9]{21LW})}\label{Lem: Liu--Weselcouch result}
For $P \in \poset{n}$,
\[
K_P = \sum_{\beta} \left(\sum_{f^* \in \mathsf{pt}_{P}(\beta)}\sign{f^*} \right)\frac{\Psi_\beta}{z_\beta}.
\]
\end{lemma}

Applying \cref{Lem: Liu--Weselcouch result} to \cref{Eq: Yalpha = sum KP} leads us to the following formula.
\begin{theorem}\label{Thm: Yalpha to quasipowersum}
Given $\alpha \models n$, let $\balpha$ be any quasisymmetric function listed in \cref{Table: our consideration}.
Then 
\[
\balpha =  \sum_{V \in \setIndSummandY}
\sum_{\beta \models n}  
\sum_{f^* \in \mathsf{pt}_{P_V}(\beta)} 
\sign{f^*}
\frac{\Psi_\beta}{z_\beta}.
\]
\end{theorem}

\section{Quasisymmetric power sum expansions II: Explicit computations}
\label{Sec: Quasisymmetric power sum expansions II}
As in the previous section, for each $\alpha \models n$, let $\mbalpha$ be an arbitrary $H_n(0)$-module chosen from the list in \cref{Table: our consideration}. For every indecomposable direct summand $V$ of $\mbalpha$, we introduced the poset $P_V$ in \cref{Subsec: pomega-partitions}, which has the property that $\DHTMP{P_V}$ is isomorphic to $\mathcal{F}_n(V)$.

In this section, we provide a way of drawing the Hasse diagram of $P_V$. 
In the cases where $V = \mDIF{\alpha}$ or $\mESF{\alpha}$, we also describe the coefficients that appear in the quasisymmetric power sum expansion of $K_{P_V}$ in terms of certain tableaux. 
Let us explain our results in more detail.

In \cref{Sec51: canonical poset}, we assign a poset $P_D \in \poset n$ to each $n$-element subset $D$ of $\mathbb N^2$. Then, for any permutation $\rho \in \SG_n$ satisfying $w_0(\alpha) \preceq_L \rho$, we algorithmically construct an $n$-element subset $D_{\alpha;\rho}$ such that $\SGR{\ourPoset{D_{\alpha;\rho}}} = [w_0 \rho^{-1},w_0 \, w_0(\alpha)^{-1}]_R$.
For the definition of $w_0(\alpha)$, refer to \cref{Sec: weak Bruaht order}.
In \cref{Sec: The poset description of indecomposable direct summands}, we utilize the results from \cref{Sec51: canonical poset} to provide a specific description of $P_V$ in the cases where $\mbalpha = \mDIF{\alpha}$, $\mESF{\alpha}$, or $\mQS{\alpha}$. Our approach can be extended to other $\mbalpha$'s since they are obtained by applying (anti-)automorphism twists to one of $\mDIF{\alpha}, \mESF{\alpha},\mQS{\alpha}$.
Finally, in \cref{Tableau descriptions of the coefficients in the expansions}, we describe the coefficients in the quasisymmetric power sum expansions of the dual immaculate functions and the extended Schur functions in terms of certain border strip tableaux. 
To do this, we use Liu--Weselcouch's theorem \cite[Theorem 6.9]{21LW}, 
which is applicable due to the Hasse diagram obtained from our algorithm.

\subsection{The canonical posets associated with subsets of \texorpdfstring{$\N^2$}{Lg}}
\label{Sec51: canonical poset}
In this section, we consider an $n$-element subset of $\mathbb{N}^2$ as a diagram containing $n$ boxes located in the first quadrant. 
Let $\mathfrak{D}_n$ be the set of $n$-element subsets of $\N^2$ satisfying the following conditions:
\begin{enumerate}[label = $\bullet$]
\item $\{j \in \N \mid (i,j) \in D \ \text{for some $i \in \N$}\} = [k]$ for some $k \in \N$, and
\item $\{i \in \N \mid (i,j) \in D \ \text{for some $j \in \N$}\} = [l]$ for some $l \in \N$.
\end{enumerate}
The conditions stated above imply that the diagram $D$ has neither empty rows nor empty columns.
Although we could proceed without imposing these conditions, we have included them to prevent excessive redundancy.

\begin{definition}
For $D \in \mathfrak{D}_n$, we define $\ourPoset{D}$ to be the poset in $\poset{n}$ whose partial order $\preceq_{\ourPoset{D}}$ is given by
\begin{align}\label{eq: def of P_D}
i \preceq_{P_{D}} j \quad \text{if and only if} \quad
\text{$x_i \leq x_j$ and $y_i \leq y_j$.}
\end{align}
Here, for $1 \leq i \leq n$, $(x_i,y_i)$ is the $i$th element when enumerating the elements in $D$ along the rows from left to right starting with the uppermost row.
We call $P_D$ \emph{the canonical poset associated with $D$}.
\end{definition}

\begin{example}\label{Ex: D and diagram} 
Consider the subsets of $\N^2$
\begin{align*}
D & = \left\{(1,2), (1,4), (2,4), (3,1), (3,3), (4,4)\right\} \text{ and}\\     
D' & = \left\{(1,2), (1,4), (2,4), (3,1), (3,3), (3,4)\right\},
\end{align*}
which are viewed as the following diagrams:
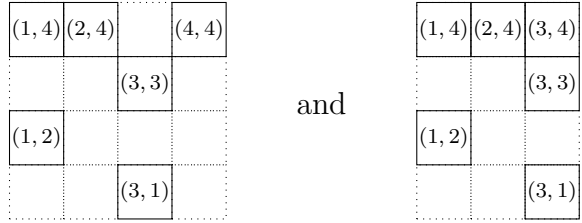
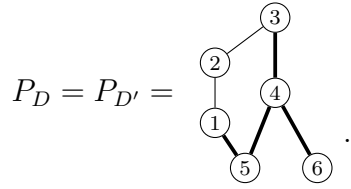
\begin{equation*}
\begin{tikzpicture}[baseline=14mm]
\def \ud {7.2mm}
\draw (\ud*0,\ud*1) rectangle (\ud*1,\ud*2) node[xshift=-\ud*0.5,yshift=-\ud*0.5] {\tiny $(1,2)$};
\draw (\ud*0,\ud*3) rectangle (\ud*1,\ud*4) node[xshift=-\ud*0.5,yshift=-\ud*0.5] {\tiny $(1,4)$};
\draw (\ud*1,\ud*3) rectangle (\ud*2,\ud*4) node[xshift=-\ud*0.5,yshift=-\ud*0.5] {\tiny $(2,4)$};
\draw (\ud*2,\ud*2) rectangle (\ud*3,\ud*3) node[xshift=-\ud*0.5,yshift=-\ud*0.5] {\tiny $(3,3)$};
\draw (\ud*2,0) rectangle (\ud*3,\ud) node[xshift=-\ud*0.5,yshift=-\ud*0.5] {\tiny $(3,1)$};
\draw (\ud*3,\ud*3) rectangle (\ud*4,\ud*4) node[xshift=-\ud*0.5,yshift=-\ud*0.5] {\tiny $(4,4)$};
\foreach \c in {0,1,3}{
    \draw[dotted] (\ud*\c,\ud*0) rectangle (\ud*\c+\ud*1,\ud*1);
}
\foreach \c in {1,2,3}{
    \draw[dotted] (\ud*\c,\ud*1) rectangle (\ud*\c+\ud*1,\ud*2);
}
\foreach \c in {0,1,3}{
    \draw[dotted] (\ud*\c,\ud*2) rectangle (\ud*\c+\ud*1,\ud*3);
}
\foreach \c in {0,2}{
    \draw[dotted] (\ud*\c,\ud*3) rectangle (\ud*\c+\ud*1,\ud*4);
}
\end{tikzpicture}
\qquad \text{and} \qquad 
\begin{tikzpicture}[baseline=14mm]
\def \ud {7.2mm}
\draw (\ud*0,\ud*1) rectangle (\ud*1,\ud*2) node[xshift=-\ud*0.5,yshift=-\ud*0.5] {\tiny $(1,2)$};
\draw (\ud*0,\ud*3) rectangle (\ud*1,\ud*4) node[xshift=-\ud*0.5,yshift=-\ud*0.5] {\tiny $(1,4)$};
\draw (\ud*1,\ud*3) rectangle (\ud*2,\ud*4) node[xshift=-\ud*0.5,yshift=-\ud*0.5] {\tiny $(2,4)$};
\draw (\ud*2,\ud*2) rectangle (\ud*3,\ud*3) node[xshift=-\ud*0.5,yshift=-\ud*0.5] {\tiny $(3,3)$};
\draw (\ud*2,0) rectangle (\ud*3,\ud) node[xshift=-\ud*0.5,yshift=-\ud*0.5] {\tiny $(3,1)$};
\draw (\ud*2,\ud*3) rectangle (\ud*3,\ud*4) node[xshift=-\ud*0.5,yshift=-\ud*0.5] {\tiny $(3,4)$};
\foreach \c in {0,...,2}{
    \draw[dotted] (\ud*\c,\ud*0) rectangle (\ud*\c+\ud*1,\ud*1);
}
\foreach \c in {1,2}{
    \draw[dotted] (\ud*\c,\ud*1) rectangle (\ud*\c+\ud*1,\ud*2);
}
\foreach \c in {0,1}{
    \draw[dotted] (\ud*\c,\ud*2) rectangle (\ud*\c+\ud*1,\ud*3);
}
\foreach \c in {0,2}{
    \draw[dotted] (\ud*\c,\ud*3) rectangle (\ud*\c+\ud*1,\ud*4);
}
\end{tikzpicture}
\end{equation*}
Then 
\[
\begin{tikzpicture}[baseline=3.5mm]
\def \hp {8mm}
\def \vp {10mm}
\def \ccc {1.3mm}
\node[left] at (\hp*-0.5,\vp*1) {$P_D = P_{D'} = $};

\node[shape=circle,draw,minimum size=\ccc*3, inner sep=0pt] at (\hp*0.5, 0) (A5) {\tiny $5$};
\node[shape=circle,draw,minimum size=\ccc*3, inner sep=0pt] at (\hp*0, \vp*1.4) (A2) {\tiny $2$};
\node[shape=circle,draw,minimum size=\ccc*3, inner sep=0pt] at (\hp*0, \vp*0.6) (A1) {\tiny $1$};
\node[shape=circle,draw,minimum size=\ccc*3, inner sep=0pt] at (\hp*1, \vp*1) (A4) {\tiny $4$};
\node[shape=circle,draw,minimum size=\ccc*3, inner sep=0pt] at (\hp*1.7, \vp*0) (A6) {\tiny $6$};
\node[shape=circle,draw,minimum size=\ccc*3, inner sep=0pt] at (\hp, \vp*2) (A3) {\tiny $3$};

\draw (A1) -- (A2) -- (A3);
\draw[line width=0.5mm] (A1) -- (A5) -- (A4) -- (A3);
\draw[line width=0.5mm] (A4) -- (A6);
\end{tikzpicture} \ .
\]
As illustrated in this example, two distinct diagrams can yield the same poset.
\end{example}

Let $D \in \mathfrak{D}_n$.
A filling $T$ on $D$ with entries $1,2,\ldots,n$ is called a \emph{standard tableau} if it satisfies the following two conditions:
\begin{enumerate}[label = {\rm (\arabic*)}]
\item the entries of $T$ are all distinct,
\item 
$T(i,j) \le T(k,l)$ if $i \le k$ and $j \le l$.
\end{enumerate}
Here $T(i,j)$ is the entry of $T$ at $(i,j)$.
We denote by $\sfST{D}$ the set of standard tableaux on $D$ and by $\sourceTD{D}$ (resp. $\sinkTD{D}$) the standard tableau on $D$ obtained by filling integers $1, 2, \ldots, n$ without repetition along the rows from left to right, starting with the lowermost row 
(resp. along the columns from bottom to top, starting with the leftmost column).
Given $T \in \sfST{D}$,  let $\readingLR{T}$ be the word obtained by reading the entries of $T$ along the rows from right to left starting with the lowermost row.
It should be emphasized that in this paper, every word containing distinct entries from $1$ to $n$ is regarded as a permutation in $\SG_n$.
Let us consider the bijection:
\begin{align*}
\mapf: \SG_n \rightarrow \SG_n, \quad \gamma \mapsto w_0 \gamma^{-1},
\end{align*}
which has already been defined in \cref{eq: bold f map}.
We have the following theorem.

\begin{theorem}\label{Lem: SGTPD = w0readSink and w0readSource to the right}
Let $D \in \mathfrak{D}_n$. Then we have 
\[
\SGR{\ourPoset{D}} = \mapf([\readingLR{\sourceTD{D}},\readingLR{\sinkTD{D}}]_L).
\]
\end{theorem}
\begin{proof}
We first claim that $\SGR{\ourPoset{D}}$ is a right weak Bruhat interval in $\SG_n$.
Due to~\cref{Thm: BW equivalent}, this claim can be verified by showing that $\ourPoset{D}$ is regular,
i.e., $\ourPoset{D}$ has no triples $(u,v,w)$'s satisfying \cref{Def: regular condition of P}, which is immediate from the definition of $\ourPoset{D}$.

We next claim that $\mapf(\readingLR{\sourceTD{D}})$ and $\mapf(\readingLR{\sinkTD{D}})$ are in $\SGR{\ourPoset{D}}$.
To do this, let $\sfFD$ denote the filling on $D$ by placing $i$ $(1 \leq i \leq n)$ in the $i$th box when enumerating the boxes in $D$ along the rows from left to right, starting at the top row.
Then we consider the following two words obtained from $\sfFD$:
\begin{enumerate}[label = -, leftmargin = 3ex]
\item 
the word $\bfwD{D}$ obtained by reading the entries of $\sfFD$ along the columns from bottom to top starting with the leftmost column and 

\item 
the word $\bfwPD{D}$ obtained by reading the entries of $\sfFD$ along the rows from left to right starting with the lowermost row. 
\end{enumerate}
By definition, $\bfwD{D}$ and $\bfwPD{D}$ are in $\SGR{\ourPoset{D}}$.
We assert that the following equalities hold:
\[
w_0 \, \readingLR{\sinkTD{D}}^{-1} = \bfwD{D}
\quad \text{and} \quad
w_0 \, \readingLR{\sourceTD{D}}^{-1}= \bfwPD{D}.
\]
We will only prove the first equality using mathematical induction on $|D|$. 
The second equality can be obtained in the same manner as the first one.

It is clear that the equality $w_0 \, \readingLR{\sinkTD{D}}^{-1} = \bfwD{D}$ holds when $|D| = 1$. 
Given a positive integer $m$, assume that the equality holds for all elements in $\mathfrak{D}_m$.
Choose any $D$ in $\mathfrak{D}_{m+1}$.
Let $(x, y)$ be the uppermost box in the rightmost column of $D$.
We notice that the box $(x, y)$ is filled with $m+1$ in $\sinkTD{D}$.
Let $D':= D \setminus \{(x, y)\}$ and let $z$ be the number of boxes in $D$ that are weakly above and strictly left of $(x, y)$.
In addition, set 
\[
\readingUL{\sinkTD{D}} := \readingLR{\sinkTD{D}} \, w_0
\]
that is the word obtained by the entries of $\sinkTD{D}$ along the rows from left to right starting with the uppermost row.
Observe that if $\readingUL{\sinkTD{D'}} = v_1 v_2 \cdots v_m$, then
\[ 
\readingUL{\sinkTD{D}} = v_1 v_2 \cdots v_z \  m+1 \ v_{z+1} v_{z+2} \cdots v_{m}.
\]
This implies that
\[ 
\readingUL{\sinkTD{D}}^{-1}(i) =
\begin{cases}
\readingUL{\sinkTD{D'}}^{-1}(i) & \text{if $i < m+1$ and $\readingUL{\sinkTD{D'}}^{-1}(i) \le z$,} \\
\readingUL{\sinkTD{D'}}^{-1}(i) + 1 & \text{if $i < m+1$ and $\readingUL{\sinkTD{D'}}^{-1}(i) > z$,} \\
z+1 & \text{if $i = m+1$}
\end{cases}
\]
for $1 \le i \le m+1$.
By the induction hypothesis, we have $\readingUL{\sinkTD{D'}}^{-1} = \bfwD{D'}$. 
In addition, we have $\bfwD{D}(m+1) = z+1$ by the definition of $\sfFD$.
Therefore, we have
\[
w_0 \, \readingLR{\sinkTD{D}}^{-1} = \readingUL{\sinkTD{D}}^{-1} = \bfwD{D}.
\]

Finally, we claim that $ w_0 \, \readingLR{\sinkTD{D}}^{-1}$ is minimal and $w_0 \, \readingLR{\sourceTD{D}}^{-1}$ is maximal in $\SGR{\ourPoset{D}}$.
This claim can be verified by showing that 
\begin{enumerate}[label = {\rm (\roman*)}]
\item $\bfwD{D} s_i \notin \SGR{\ourPoset{D}}$ for all $i \in \Des{R}{\bfwD{D}}$, and
\item $\bfwPD{D} s_j \notin \SGR{\ourPoset{D}}$ for all $j \notin \Des{R}{\bfwPD{D}}$.
\end{enumerate}
We only prove (i) since (ii) can be proved in the same way.
Choose an arbitrary $i \in \Des{R}{\bfwD{D}}$.
Combining the inequality $\bfwD{D}(i) > \bfwD{D}(i+1)$ with the \cref{eq: def of P_D}, we have 
\[
\bfwD{D}(i) \preceq_{\ourPoset{D}} \bfwD{D}(i+1).
\]
Therefore, $\bfwD{D} s_i \notin \SGR{\ourPoset{D}}$, as desired.
\end{proof}

Let $D \in \mathfrak{D}_n$.
By the definitions of $\sourceTD{D}$ and ${\tt read}_{\mathsf{ST}}$, it is straightforward to see that $\readingLR{\sourceTD{D}}=w_0((r_1,r_2, \ldots,r_{k})^{\rmc})$, where $r_i\,\,(1\le i \le k)$ is the number of boxes in the $i$th row of $D$ (from bottom to top) and $k$ is the number of rows of $D$.
Hence \cref{Lem: SGTPD = w0readSink and w0readSource to the right} implies that $\SGR{\ourPoset{D}}$ is of the form 
$\mapf([w_0(\alpha), \rho]_L)$ for some $\alpha \models n$ and $ w_0(\alpha) \preceq_L \rho$.
Furthermore, the converse of this implication also holds.
In the following, given $\alpha \models n$ and $\rho \in \SG_n$ with $w_0(\alpha) \preceq_L \rho$, 
we will construct a diagram $D \in \mathfrak{D}_n$ such that 
$\SGR{\ourPoset{D}} = \mapf([w_0(\alpha),\rho]_L)$. 
The subsequent algorithm plays a crucial role in this construction. 

\begin{algorithm}\label{Algo: Construction of  D_alpha_rho}
Let $\alpha \models n$ and $w_0(\alpha) \preceq_L \rho$.

\begin{enumerate}[label = {\it Step \arabic*.}]
\item
Let $\set(\alpha^\rmc) = \{z_1,z_2,\ldots,z_l\}$, and $z_0 := 0$ and $z_{l+1}:=n$.
For $1 \leq j \leq l+1$, 
let 
\[
R_j(\alpha;\rho) := \{\rho(r) \mid z_{j-1}+1 \leq r \leq z_j\}.
\]

\item 
Let $e:= |\Des{L}{\rho}|$ and let $\Des{L}{\rho} = \{k_1 < k_2 < \cdots < k_e\}$, $k_0:=0$, and $k_{e+1}:=n$.
Then, for $i = 1,2,\ldots,e+1$, set
\[
C_i(\alpha;\rho) := \{c \mid k_{i-1}+1 \leq c \leq k_i\}.
\]

\item 
Let 
\[
D := \{(i,j) \in \N^2 \ \mid R_j(\alpha;\rho) \cap C_i(\alpha;\rho) \neq \emptyset \}.
\]
Return $D$.
\end{enumerate}  
\end{algorithm}

We denote by $D_{\alpha;\rho}$ the set $D$ obtained from \cref{Algo: Construction of D_alpha_rho}.

\begin{example}
Let $\alpha = (1,1,2,2,1,1,1) \models 9$ and $\rho = 8 \ 4 \ 1 \ 5 \ 3 \ 9 \ 7 \ 6 \ 2$.
We start by observing that $w_0(\alpha) = 3 \ 2 \ 1 \ 5 \ 4 \ 9 \ 8 \ 7 \ 6$ and $w_0(\alpha) \preceq_L \rho$, indicating that we can proceed with the construction of the diagram $D_{\alpha;\rho}$ using \cref{Algo: Construction of D_alpha_rho}.

Since $\set(\alpha^\rmc) = \{3,5\}$, {\it Step 1} yields that 
\[
R_1(\alpha;\rho) = \{1,4,8\}, \quad R_2(\alpha;\rho) = \{3,5\} \quad \text{and} \quad R_3(\alpha;\rho) = \{2,6,7,9\}.
\]
Since $\Des{L}{\rho} =\{2,3,6,7\}$, {\it Step 2} yields that
\begin{align*}
&C_{1}(\alpha;\rho) = \{1,2\}, \quad C_{2}(\alpha;\rho) = \{3\}, \quad C_3(\alpha;\rho) = \{4,5,6\}, \\
&C_{4}(\alpha;\rho) = \{7\} \quad \text{and} \quad  C_{5}(\alpha;\rho) = \{8,9\}.
\end{align*}
Finally, {\it Step 3} gives us 
\[
D_{\alpha;\rho} = \{(1,1),(3,1),(5,1),(2,2),(3,2),(1,3),(3,3),(4,3),(5,3)\}.
\]
\end{example}

\begin{theorem}\label{Thm: Diagram Dalpharho}
Given $\alpha \models n$ and $\rho \in \SG_n$ with $w_0(\alpha) \preceq_L \rho$, the diagram $D_{\alpha;\rho}$ obtained from \cref{Algo: Construction of D_alpha_rho} satisfies that 
$\SGR{\ourPoset{D_{\alpha;\rho} }} = \mapf([w_0(\alpha),\rho]_L)$.
\end{theorem}
\begin{proof}
Note that $D_{\alpha;\rho}$ is designed to satisfy
\[
D_{\alpha;\rho} \in \mathfrak{D}_n,
\quad 
\readingLR{\sourceTD{D_{\alpha;\rho}}} = w_0(\alpha), 
\quad \text{and} \quad \readingLR{\sinkTD{D_{\alpha;\rho}}} = \rho.
\]
Therefore, the assertion follows from \cref{Lem: SGTPD = w0readSink and w0readSource to the right}. 
\end{proof}

By merging the results of \cref{Lem: SGTPD = w0readSink and w0readSource to the right} and \cref{Thm: Diagram Dalpharho}, we obtain the equality
\[\left\{ \SGR{\ourPoset{D}} \mid D \in \mathfrak{D}_n \right\} = \left\{ \mapf([w_0(\alpha), \rho]_L) \mid \alpha \models n,  w_0(\alpha) \preceq_L \rho \right\}.\]

\subsection{The Hasse diagram of \texorpdfstring{$\posetFromAlgo{V}$}{Lg} for \texorpdfstring{$V \in \setIndSummandY$}{Lg}}
\label{Sec: The poset description of indecomposable direct summands}
Let us begin by introducing the necessary notation. Throughout this section, we fix a composition $\alpha = (\alpha_1,\ldots, \alpha_{\ell(\alpha)})$ of $n$.
We define the \emph{composition diagram} $\tcd(\alpha)$ of $\alpha$ as a left-justified array of $n$ boxes where the $i$th row from the top has $\alpha_i$ boxes for $1 \leq i \leq \ell(\alpha)$.
A box in $\tcd(\alpha)$ is said to be in the $i$th row if it is in the $i$th row from the top and in the $j$th column if it is in the $j$th column from the left.
We use $(i,j)$ to denote the box in the $i$th row and $j$th column.
Given a filling $\tau$ of $\tcd(\alpha)$ and $(i,j) \in \tcd(\alpha)$, we denote by $\tau_{i,j}$ the entry at $(i,j)$ in $\tau$.

\subsubsection{\texorpdfstring{$\mbalpha = \mDIF{\alpha}$ or $\mESF{\alpha}$}{Lg}} 
A \emph{standard immaculate tableau of shape $\alpha$} is a filling $\calT$ of $\tcd(\alpha)$ with entries in $\{1,2,\ldots,n\}$ such that the entries are all distinct, the entries in each row increase from left to right, and the entries in the first column increase from top to bottom.
A \emph{standard extended tableau of shape $\alpha$} is a filling $\sfT$ of the composition diagram $\tcd(\alpha)$ with $\{1,2,\ldots,n\}$ such that the entries are all distinct, the entries in each row increase from left to right, and the entries in each column increase from top to bottom.
Let $\SIT(\alpha)$ be the set of all standard immaculate tableaux of shape $\alpha$ and $\SET(\alpha)$ the set of all standard extended tableaux of shape $\alpha$.
The $H_n(0)$-module $\mDIF{\alpha}$ was introduced by Berg--Bergeron--Saliola--Serrano--Zabrocki~\cite{15BBSSZ} by establishing a suitable $H_n(0)$-action on the $\C$-span of $\SIT(\alpha)$.
Similarly, the $H_n(0)$-module $\mESF{\alpha}$ was introduced by Searles \cite{19Searles} by establishing an $H_n(0)$-action on the $\C$-span of $\SET(\alpha)$.
Indeed, both $\mDIF{\alpha}$ and $\mESF{\alpha}$ are well-known to be indecomposable.

To introduce our result, we require the following tableaux of $\tcd(\alpha)$:
\begin{enumerate}[label = $\bullet$, leftmargin=4ex]
\item 
$\sourceSIT{\alpha}$ is the $\SIT$ of shape $\alpha$ obtained by filling $\tcd(\alpha)$ with $1,2,\ldots,n$ from left to right and from top to bottom.

\item 
$\sinkSIT{\alpha}$ is the $\SIT$ of shape $\alpha$ obtained by filling $\tcd(\alpha)$ in the following steps:
\begin{enumerate}[label = {\rm (\roman*)}, leftmargin=3ex]
\item Fill the first column with entries $1,2,\ldots, \ell(\alpha)$ from top to bottom.
\item Fill the remaining boxes with entries $\ell(\alpha) + 1, \ell(\alpha) + 2, \ldots, n$ from left to right from bottom to top.
\end{enumerate}

\item 
$\sfT_\alpha$ is the $\SET$ of shape $\alpha$ obtained by filling $\tcd(\alpha)$ with $1,2,\ldots,n$ from left to right and from top to bottom.

\item 
$\sfT'_\alpha$ is the $\SET$ of shape $\alpha$ obtained by filling $\tcd(\alpha)$ with the entries $1,2,\ldots, n$ from top to bottom and from left to right.
\end{enumerate}
Given any filling $T$ of a composition diagram, let $\rmread(T)$ be the word obtained from $T$ by reading the entries from right to left starting with the top row.

\begin{lemma}{\rm (\cite[Theorem 5]{22JKLO})}
\label{Lemma: V and X are interval}
For every $\alpha \models n$, we have the following $H_n(0)$-module isomorphisms.
\begin{enumerate}[label = {\rm (\alph*)},itemsep=2pt]
    \item $\mDIF{\alpha} \cong \sfB_L(\sigma_{\mDIF{\alpha}},\rho_{\mDIF{\alpha}})$,
where $\sigma_{\mDIF{\alpha}} = \rmread(\sourceSIT{\alpha})$ and $
\rho_{\mDIF{\alpha}} = \rmread(\sinkSIT{\alpha})$.

\item $\mESF{\alpha} \cong \sfB_L(\sigma_{\mESF{\alpha}},\rho_{\mESF{\alpha}})$, 
where $\sigma_{\mESF{\alpha}} = \rmread(\sourceSET{\alpha})$ and $\rho_{\mESF{\alpha}}
= \rmread(\sinkSET{\alpha})$.
\end{enumerate}
\end{lemma}

Using \cref{Lemma: V and X are interval}, we can derive the following theorem. 

\begin{theorem}\label{Thm: posets for V and X}
For each $\alpha \models n$, we have the following.
\begin{enumerate}[label = {\rm (\alph*)},itemsep=3pt]
\item $P_{\mDIF{\alpha}} = P_{D_{\alpha^\rmc; \rmread(\sinkSIT{\alpha})}}$ and $P_{\mESF{\alpha}} = P_{D_{\alpha^\rmc; \rmread(\sfT'_\alpha)}}$.

\item For $1 \le i \le \ell(\alpha)$, let $k_i = \sum_{i+1 \le j \le \ell(\alpha)}(\alpha_j-1)$. 
Then 
\begin{align*}
& D_{\alpha^\rmc; \rmread(\sinkSIT{\alpha})} =
\{(1,i) \mid (i,1) \in \tcd(\alpha)\} \cup \{(j+k_i,i) \mid (i,j) \in \tcd(\alpha) \text{ with } j \geq 2\} \text{ and } \\ 
& D_{\alpha^\rmc; \rmread(\sfT'_\alpha)} = \{(i,j) \mid (j,i) \in \tcd(\alpha)\}.
\end{align*}
\end{enumerate}
\end{theorem}
\begin{proof}
(a) Recall that 
$P_{\mDIF{\alpha}}$ is the abbreviation of $P_{I_{\mDIF{\alpha}}}$. 
\cref{Lemma: V and X are interval} says that  
\[
I_{\mDIF{\alpha}}=\mapf([\rmread(\sourceSIT{\alpha}),\rmread(\sinkSIT{\alpha})]_L).
\]
On the other hand, \cref{Thm: Diagram Dalpharho} says that 
\[
\SGR{P_{D_{\alpha^\rmc; \rmread(\sinkSIT{\alpha})}}}=\mapf([w_0(\alpha^\rmc),\rmread(\sinkSIT{\alpha})]_L).
\]
In the proof of \cref{Lemma: V and X are interval}, it was pointed out that 
$w_0(\alpha^\rmc)=\rmread(\sourceSIT{\alpha})$.
Consequently, $I_{\mDIF{\alpha}}=\SGR{P_{D_{\alpha^\rmc; \rmread(\sinkSIT{\alpha})}}}$, which provides the first equality in the assertion.

The second equality in the assertion can be derived using the same approach as described above.

(b) Apply \cref{Algo: Construction of  D_alpha_rho} to the pair $(\alpha^\rmc, \rmread(\sinkSIT{\alpha}))$ instead of $(\alpha,\rho)$.
The first equality can be obtained by directly computing $R_j(\alpha^\rmc, \rmread(\sinkSIT{\alpha}))$ and $C_i(\alpha^\rmc, \rmread(\sinkSIT{\alpha}))$. 
The second equality can also be derived using the same approach.
\end{proof}

\begin{example}
For $\alpha = (3,2,4) \models 9$, 
consider the following tableaux and their corresponding reading words: 
\begin{align*}
\sinkSIT{\alpha} & = 
\begin{ytableau}
1 & 8 & 9 \\
2 & 7 \\ 
3 & 4 & 5 & 6
\end{ytableau} 
\qquad 
\rmread(\sinkSIT{\alpha}) = 9 \ 8 \ 1 \ 7 \ 2 \ 6 \ 5 \ 4 \ 3 \\[3mm]
\sinkSET{\alpha} & = 
\begin{ytableau}
1 & 4 & 7 \\
2 & 5 \\ 
3 & 6 & 8 & 9
\end{ytableau}
\qquad 
\rmread(\sinkSET{\alpha}) = 7 \ 4 \ 1 \ 5 \ 2 \ 9 \ 8 \ 6 \ 3
\end{align*}
Since $\alpha^\rmc = (1,1,2,2,1,1,1)$, one observes that $w_0(\alpha^\rmc) = 3 \ 2 \ 1 \ 5 \ 4 \ 9 \ 8 \ 7 \ 6$ and $ \rmread(\sourceSIT{\alpha}) = \rmread(\sourceSET{\alpha}) = w_0(\alpha^\rmc)$.
Applying \cref{Algo: Construction of  D_alpha_rho} to $(\alpha^\rmc,\rmread(\sinkSIT{\alpha}))$ or $(\alpha^\rmc,\rmread(\sinkSET{\alpha}))$ yields that 
\begin{align*}
D_{\alpha^\rmc; \rmread(\sinkSIT{\alpha})} &=
\{(1,1),(6,1),(7,1),(1,2),(5,2),(1,3),(2,3),(3,3),(4,3)\}
, \\ 
D_{\alpha^\rmc; \rmread(\sfT'_\alpha)} &= \{(1,1),(2,1),(3,1),(1,2),(2,2),(1,3),(2,3),(3,3),(4,3)\}.
\end{align*}
Therefore, the Hasse diagrams of the desired posets are 
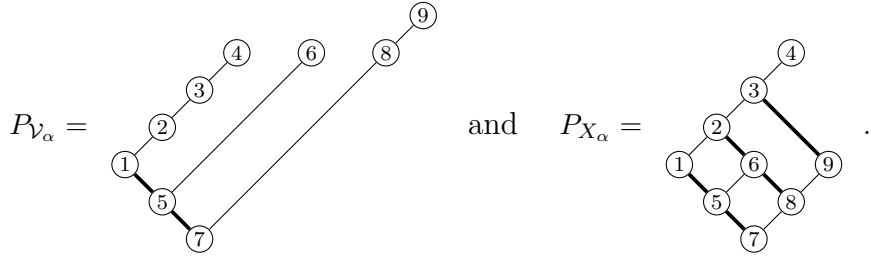
\[
\def \ccc {3.5mm}
\def \hhh {1.2em}
\def \vvv {1.2em}
\begin{tikzpicture}[baseline = 14mm]
\node[left] at (\hhh*-0.7,\vvv*3) {\small $P_{\mDIF{\alpha}} =$};
\node[shape=circle,draw,minimum size=\ccc, inner sep=0pt] at (\hhh*3,\vvv*5) (D4) {\tiny $4$};
\node[shape=circle,draw,minimum size=\ccc, inner sep=0pt] at (\hhh*2,\vvv*4) (D3) {\tiny $3$};
\node[shape=circle,draw,minimum size=\ccc, inner sep=0pt] at (\hhh*1,\vvv*3) (D2) {\tiny $2$};
\node[shape=circle,draw,minimum size=\ccc, inner sep=0pt] at (\hhh*0,\vvv*2) (D1) {\tiny $1$};
\node[shape=circle,draw,minimum size=\ccc, inner sep=0pt] at (\hhh*1,\vvv*1) (D5) {\tiny $5$};
\node[shape=circle,draw,minimum size=\ccc, inner sep=0pt] at (\hhh*5,\vvv*5) (D6) {\tiny $6$};
\node[shape=circle,draw,minimum size=\ccc, inner sep=0pt] at (\hhh*2,\vvv*0) (D7) {\tiny $7$};
\node[shape=circle,draw,minimum size=\ccc, inner sep=0pt] at (\hhh*7,\vvv*5) (D8) {\tiny $8$};
\node[shape=circle,draw,minimum size=\ccc, inner sep=0pt] at (\hhh*8,\vvv*6) (D9) {\tiny $9$};

\draw (D1) -- (D2) -- (D3) -- (D4);
\draw (D5) -- (D6);
\draw (D7) -- (D8) -- (D9);
\draw[line width=\lw] (D1) -- (D5) -- (D7);
\end{tikzpicture} 
\quad \text{and} \quad  
\def \ccc {3.5mm}
\def \hhh {1.2em}
\def \vvv {1.2em}
\begin{tikzpicture}[baseline = 14mm]
\node[left] at (\hhh*-0.7,\vvv*3) {\small $P_{\mESF{\alpha}} =$};
\node[shape=circle,draw,minimum size=\ccc, inner sep=0pt] at (\hhh*3,\vvv*5) (D4) {\tiny $4$};
\node[shape=circle,draw,minimum size=\ccc, inner sep=0pt] at (\hhh*2,\vvv*4) (D3) {\tiny $3$};
\node[shape=circle,draw,minimum size=\ccc, inner sep=0pt] at (\hhh*1,\vvv*3) (D2) {\tiny $2$};
\node[shape=circle,draw,minimum size=\ccc, inner sep=0pt] at (\hhh*0,\vvv*2) (D1) {\tiny $1$};
\node[shape=circle,draw,minimum size=\ccc, inner sep=0pt] at (\hhh*1,\vvv*1) (D5) {\tiny $5$};
\node[shape=circle,draw,minimum size=\ccc, inner sep=0pt] at (\hhh*2,\vvv*2) (D6) {\tiny $6$};
\node[shape=circle,draw,minimum size=\ccc, inner sep=0pt] at (\hhh*4,\vvv*2) (D9) {\tiny $9$};
\node[shape=circle,draw,minimum size=\ccc, inner sep=0pt] at (\hhh*3,\vvv*1) (D8) {\tiny $8$};
\node[shape=circle,draw,minimum size=\ccc, inner sep=0pt] at (\hhh*2,\vvv*0) (D7) {\tiny $7$};

\draw (D1) -- (D2) -- (D3) -- (D4);
\draw (D5) -- (D6);
\draw (D7) -- (D8) -- (D9);
\draw[line width=\lw] (D1) -- (D5) -- (D7);
\draw[line width=\lw] (D2) -- (D6) -- (D8);
\draw[line width=\lw] (D3) -- (D9);
\end{tikzpicture} \ \ .
\]
\end{example}

\subsubsection{\texorpdfstring{$\mbalpha = \mQS{\alpha}$}{Lg}} \label{explcit algorithm for S}
A \emph{standard reverse composition tableau} ($\SRCT$) of shape $\alpha$ is a filling $\tau$ of $\tcd(\alpha)$ with entries in $\{1,2,\ldots,n\}$ such that the following conditions hold:
\begin{enumerate}[label = {\rm (\arabic*)}]
\item The entries are all distinct.

\item The entries in each row are decreasing when read from left to right.

\item The entries in the first column are increasing when read from top to bottom.

\item (The triple condition) If $i<j$ and $\tau_{i,k} > \tau_{j,k+1}$, then $(i,k+1) \in \tcd(\alpha)$ and $\tau_{i,k+1} > \tau_{j,k+1}$.
\end{enumerate}
Let $\SRCT(\alpha)$ be the set of SRCTx of shape $\alpha$.
For $\tau \in \SRCT(\alpha)$ and $1 \leq i \leq n-1$, we say that $i$ is a {\em descent} of $\tau$ if $i+1$ appears weakly right of $i$ in $\tau$.
Denote by $\Des{}{\tau}$ the set of descents in $\tau$.
For $i,j \in [n]$ with $i < j$, we say that $i$ {\em attacks} $j$ in $\tau$ if either
\begin{equation}\label{Def: attacking condition}
\begin{aligned}
&\text{$i$ and $j$ are in the same column in $\tau$, or}\\
&\text{$i$ and $j$ are in adjacent columns in $\tau$, with $j$ positioned lower-right of $i$.}
\end{aligned}
\end{equation}
In the case where $i$ is a descent of $\tau$ and it attacks (resp. does not attack) $i+1$ in $\tau$, we say that $i$ is an {\em attacking descent} (resp. a {\em nonattacking descent}) of $\tau$.

For $1 \le i \le \alphamax$, we define the \emph{$i$th column word} $\sfc^i(\tau)$ of $\tau$ to be the word obtained from $\tab$ by reading the entries in the $i$th column from top to bottom.
The \emph{standardized $i$th column word} of $\tau$, denoted by $\mathrm{st}_i(\tau)$, is the word $p_1 p_2 \cdots p_{l_i}$ uniquely determined by the conditions 
\begin{enumerate}[label = {\rm (\roman*)}]
\item $\{p_1, p_2, \ldots, p_{l_i}\} = \{1,2, \ldots, l_i\}$ and
\item 
$p_j > p_{j'}$ if and only if $\sfc^i(\tau)_{j} > \sfc^i(\tau)_{j'}$ for all $1 \le j < j' \le l_i$,
\end{enumerate}
where $\sfc^i(\tau)_{j}$ is the $j$th entry of $\sfc^i(\tau)$ and $l_i$ is the length of the word $\sfc^i(\tau)$.
Let $\tau_1,\tau_2 \in \SRCT(\alpha)$.
Define the equivalence relation $\sim_{\alpha}$ on $\SRCT(\alpha)$ by
\[
\tau_1 \sim_{\alpha} \tau_2 \quad \text{if and only if} \quad \rmst_i(\tau_1) = \rmst_i(\tau_2) \quad \text{for all } 1 \leq i \leq \alphamax.
\]
Let $\classQS{\alpha}$ be the set of all equivalence classes under $\sim_{\alpha}$.
For $E \in \classQS{\alpha}$,  let $\mQS{\alpha,E}$ be the submodule of $\mQS{\alpha}$ whose underlying space is the $\C$-span of $E$.

In \cite[Theorem 3.2]{15TW}, Tewari--van Willigenburg introduced a left $H_n(0)$-module denoted as $\mQS{\alpha}$. This module was defined by considering an appropriate $H_n(0)$-action on the $\C$-span of $\SRCT(\alpha)$.
They demonstrated that $\mQS{\alpha}$ is decomposed as 
\[
\bigoplus_{E \in \classQS{\alpha}} \mQS{\alpha,E}.
\]
where $\mQS{\alpha,E}$ is the submodule of $\mQS{\alpha}$ whose underlying space is the $\C$-span of $E$.
Later, K\"{o}nig \cite{19Konig} showed that every direct summand $\mQS{\alpha,E}$ is indecomposable.

An $\SRCT$ $\tau$ is said to be a \emph{source tableau} if, for every $i \notin \Des{}{\tau}$ where $i \neq n$, $i+1$ lies to the immediate left of $i$.
An $\SRCT$ $\tau$ is said to be a \emph{sink tableau} if, for every $i \in \Des{}{\tau}$, $i$ is an attacking descent.
In \cite[Section 6]{15TW}, it was shown that there is a unique source tableau $\sourcetauE{E}$ and a unique sink tableau $\sinktauE{E}$ in each equivalence class $E$.
These tableaux satisfy the following property: for any $\tau \in E$,
\begin{align*}
\tau = \pi_\xi \cdot \sourcetauE{E} \quad \text{for some } \xi \in \SG_n \quad \text{and} \quad \pi_{\xi} \cdot \tau = \sinktauE{E} \quad \text{for some } \xi \in \SG_n.
\end{align*}

From now on, we fix an equivalence class $E \in \classQS{\alpha}$ and let $V = \mQS{\alpha,E} \ (= H_n(0) \cdot \sourcetauE{E})$.
Let 
\[
\Des{}{\sourcetauE{E}} = \left\{d_1 < d_2 < \cdots < d_{m_E^{~}}\right\},
\]
where $m_E := |\Des{}{\sourcetauE{E}}|$.
For convenience, we set $d_0 := 0$ and $d_{m_E^{~}+1} := n$.
In \cite{22CKNO1}, the composition diagram $\tcd(\alpha)$ was decomposed into $\tH_j$'s ($1 \leq j \leq m_E^{~}+1$), where $\tH_j$ is the horizontal strip occupied by the boxes with entries from $d_{j-1}+1$ to $d_j$ in $\sourcetauE{E}$.
Employing this decomposition, Jung--Kim--Lee--Oh~\cite{22JKLO} introduced a reading word of $\tau \in E$, denoted by $\tread{\tau}$, in the following way:
For $1 \le j \le m_E^{~} + 1$, let $\rmw^{(j)}(\tau)$ be the word obtained by reading the entries in $\tau$ filling $\tH_j$ from left to right.
Define $\tread{\tau}$ to be the word obtained by concatenating $\rmw^{(i)}(\tau)$'s:
\[
\tread{\tau} = \rmw^{(1)}(\tau) \ \rmw^{(2)}(\tau) \  \cdots \  \rmw^{(m_E^{~} + 1)}(\tau).
\]

\begin{example}
For $\alpha = (2,3,2,4) \models 11$, let  
\[
\tau_1=
\begin{ytableau}
{\color{red} 3}  &  2 \\
{\color{red} 6}  &  5  &  4 \\ 
{\color{red} 7}  &  {\color{red} 1} \\ 
11 &  10 &  9  &  8
\end{ytableau}, \quad 
\tau_2=
\begin{ytableau}
{\color{red} 4}  &  2 \\
{\color{red} 7}  &  5  &  3 \\ 
{\color{red} 8}  &  {\color{red} 1} \\ 
11 &  10 &  9  &  6
\end{ytableau} \quad \text{and} \quad 
\tau_3=
\begin{ytableau}
{\color{red} 5}  &  {\color{red} 2} \\
{\color{red} 8}  &  {\color{red} 6}  &  3 \\ 
{\color{red} 9}  &  {\color{red} 1} \\ 
11 &  10 &  7  &  4
\end{ytableau}.
\]
Here, for $k = 1,2,3$, the red entries in $\tau_k$ indicate that they belong to $\Des{}{\tau_k}$.
These tableaux are contained in the same class, denoted as $E$, of $\mQS{\alpha}$. Specifically, $\tau_1$ is the source tableau in $E$, and $\tau_3$ is the sink tableau in $E$. Furthermore, we have the following relations:
\[
\tau_2 = \pi_{6}\pi_{7} \pi_{3} \cdot \tau_1
\quad \text{and} \quad 
\tau_3 = \pi_{7} \pi_{8}\pi_{4} \pi_{5} \cdot \tau_2.
\]
Since $\Des{}{\tau_1} = \{1 < 3 < 6 < 7\}$, the decomposition of $\tcd(\alpha)$ into the horizontal strips $\tH_j$ $(1 \leq j \leq 5)$ is 
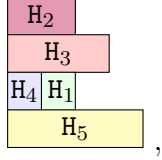
\[
\begin{tikzpicture}
\def \hhh {4.5mm}
\def \vvv {5mm}
\filldraw[color=purple!40] (\hhh*0,\vvv*1) rectangle (\hhh*2,\vvv*2);
\draw (\hhh*0,\vvv*1) rectangle (\hhh*2,\vvv*2);
\filldraw[color=red!20] (\hhh*0,\vvv*0) rectangle (\hhh*3,\vvv*1);
\draw (\hhh*0,\vvv*0) rectangle (\hhh*3,\vvv*1);
\filldraw[color=blue!10] (\hhh*0,\vvv*-1) rectangle (\hhh*1,\vvv*0);
\draw (\hhh*0,\vvv*-1) rectangle (\hhh*1,\vvv*0);
\filldraw[color=green!10] (\hhh*1,\vvv*-1) rectangle (\hhh*2,\vvv*0);
\draw (\hhh*1,\vvv*-1) rectangle (\hhh*2,\vvv*0);
\filldraw[color=yellow!30] (\hhh*0,\vvv*-2) rectangle (\hhh*4,\vvv*-1);
\draw (\hhh*0,\vvv*-2) rectangle (\hhh*4,\vvv*-1);

\node at (\hhh*1.6,\vvv*-0.5) {\small $\tH_1$};
\node at (\hhh*0.5,\vvv*-0.5) {\small $\tH_4$};
\node at (\hhh*1.5,\vvv*0.5) {\small $\tH_3$};
\node at (\hhh*2,\vvv*-1.5) {\small $\tH_5$};
\node at (\hhh*1,\vvv*1.5) {\small $\tH_2$};
\end{tikzpicture} \ ,
\]
which yields the reading words of these tableaux 
\begin{align*}
\tread{\tau_1} & = 1 \ 3 \ 2 \ 6 \ 5 \ 4 \ 7 \ 11 \ 10 \ 9 \ 8,\\
\tread{\tau_2} & = 1 \ 4 \ 2 \ 7 \ 5 \ 3 \ 8 \ 11 \ 10 \ 9 \ 6,\\
\tread{\tau_3} & = 1 \ 5 \ 2 \ 8 \ 6 \ 3 \ 9 \ 11 \ 10 \ 7 \ 4.
\end{align*}
\end{example}

\begin{lemma}{\rm (\cite[Theorem 6]{22JKLO})}
\label{Lemma: S is interval}
For $\alpha \models n$ and $E \in \classQS{\alpha}$, we have the $H_n(0)$-module isomorphism
\[
\mQS{\alpha,E} \cong \sfB_L(\sigma_{\mQS{\alpha,E}},\rho_{\mQS{\alpha,E}}),
\]
where $\sigma_{\mQS{\alpha,E}} =  \tread{\sourcetauE{E}}$ and $
\rho_{\mQS{\alpha,E}} = \tread{\sinktauE{E}}$.
\end{lemma}

With this preparation, we can state the main result of this subsection.

\begin{theorem} \label{explicit description for S}
For $\alpha \models n$ and $E \in \classQS{\alpha}$, we have
\[
P_{\mQS{\alpha,E}}= P_{D_{\comp(\Des{}{\sourcetauE{E}}^{\rmc});\tread{\sinktauE{E}}}}.
\]
\end{theorem}
\begin{proof}
Recall that 
$P_{\mQS{\alpha,E}}$ is the abbreviation of $P_{I_{\mQS{\alpha,E}}}$. 
\cref{Lemma: S is interval} says that  
\[
I_{\mQS{\alpha,E}}=\mapf([\tread{\sourcetauE{E}},\tread{\sinktauE{E}}]_L).
\]
On the other hand, \cref{Thm: Diagram Dalpharho} says that 
\[
\SGR{P_{D_{\comp(\Des{}{\sourcetauE{E}}^{\rmc}); \tread{\sinktauE{E}}}}}=\mapf([w_0(\comp(\Des{}{\sourcetauE{E}}^{\rmc})),\tread{\sinktauE{E}}]_L).
\]
In the proof of \cref{Lemma: S is interval}, it was pointed out that 
$w_0(\comp(\Des{}{\sourcetauE{E}}^{\rmc}))=\tread{\sourcetauE{E}}$.
Consequently, $I_{\mQS{\alpha,E}}=\SGR{P_{D_{\comp(\Des{}{\sourcetauE{E}}^{\rmc});\tread{\sinktauE{E}}}}}$, which provides the assertion.
\end{proof}

\cref{explicit description for S} shows that one can construct $P_{\mQS{\alpha,E}}$ if both $\sourcetauE{E}$ and $\sinktauE{E}$ are available.
However, unlike the cases of $\mDIF{\alpha}$ or $\mESF{\alpha}$, finding $\sourcetauE{E}$ and $\sinktauE{E}$ for a given $E\in \classQS{\alpha}$ is not straightforward.
In the following discussion, we will proceed with the assumption that only source tableaux are known.
\footnote{In general, finding all source tableaux (or sink tableaux) is a nontrivial problem.
Even the cardinality of $\classQS{\alpha}$ is still open.}
This assumption is more natural than assuming knowledge of all sink tableaux since $\sourcetauE{E}$ serves as a generator of $\mQS{\alpha,E}$.
In this situation, the crucial step in utilizing \cref{explicit description for S} is to find $\sinktauE{E}$.
Theoretically, $\sinktauE{E}$ can be obtained from $\sourcetauE{E}$ by applying $H_n(0)$-actions repeatedly. 
However, in practice, this process is not efficient due to its complexity and time-consuming nature.

In what follows, given a tableau $\tau$ (not necessarily $\sourcetauE{E}$) in $\SRCT(\alpha)$, 
we will give a concrete algorithm to produce the sink tableau in $E_\tau$ from $\tau$, where $E_\tau$ is the equivalence class containing $\tau$.

\begin{algorithm}
\label{Algo: produce the sink tableau}
Let $\alpha \models n$ and $\tau \in \SRCT(\alpha)$.
\begin{enumerate}[label = {\it Step \arabic*.}]
\item
Let $(x^{(n)},y^{(n)}) := (\ell(\alpha),1)$ and $A: = \{(x^{(n)},y^{(n)})\}$. 
Define $\tau^{(n-1)}$ to be the filling of $\tcd(\alpha)$ obtained from $\tau$ by setting
\[
\tau^{(n-1)}_{i,j} := \begin{cases}
\tau_{i,j} & \text{if } (i,j) \notin A, \\
* & \text{otherwise.}
\end{cases}
\]

\item
For each $m = n-1,n-2,\ldots,1$, 
define the sequence $\{a^{(m)}_k\}$ using the integer entries in $\tau^{(m)}$ as follows:
\begin{enumerate}[label = (\roman*)]
\item
Consider the columns which contain at least one positive integer.
Let $a^{(m)}_1$ be the greatest integer entry in the leftmost column of $\tau^{(m)}$ among these columns. 

\item
Given $a^{(m)}_k$, let $j$ be the column index of $a^{(m)}_k$ in $\tau^{(m)}$.
Check if there is an integer entry greater than $a^{(m)}_k$ in column $j+1$.
Furthermore, make sure that this entry is placed below $a^{(m)}_k$.
If such an entry exists, define $a^{(m)}_{k+1}$ to be the greatest integer entry in column $j+1$.
Otherwise, terminate the sequence.
\end{enumerate}
From the sequence $\{a^{(m)}_1,a^{(m)}_2,\ldots,a^{(m)}_l\}$ defined in this way, let $(x^{(m)},y^{(m)})$  be the position of $a^{(m)}_l$ in $\tau^{(m)}$.
After inserting $(x^{(m)},y^{(m)})$ into $A$, then define $\tau^{(m-1)}$ to be the filling of $\tcd(\alpha)$ obtained from $\tau$ by setting
\[
\tau^{(m-1)}_{i,j} := \begin{cases}
\tau_{i,j} & \text{if } (i,j) \notin A, \\
* & \text{if }  (i,j) \in A.
\end{cases}
\]

\item 
Given the set $\{(x^{(i)},y^{(i)}) \mid  1 \leq i \leq n\}$ defined by the previous step,
let $\dot{\uptau}$ be the tableau by setting
\[
\dot{\uptau}_{(x^{(i)},y^{(i)})} := i \quad \text{for all } 1 \leq i \leq n.
\]
Then return $\dot{\uptau}$.
\end{enumerate}
\end{algorithm}

We denote by $\sinkAlgo$ the filling $\dot{\uptau}$ obtained from \cref{Algo: produce the sink tableau}.

\begin{proposition}\label{Prop: Stau is sink}
Let $\alpha \models n$ and $\tau \in \SRCT(\alpha)$.
Then the filling $\sinkAlgo$ obtained from \cref{Algo: produce the sink tableau} 
is the sink tableau in $E_\tau$, where $E_\tau$ is the equivalence class containing $\tau$.
\end{proposition}
\begin{proof}
Let us first show that $\sinkAlgo \in \SRCT(\alpha)$.
From the construction of $\sinkAlgo$ using \cref{Algo: produce the sink tableau}, it is clear that $\sinkAlgo$ satisfies (1), (2), and (3) in the definition of $\SRCT$.
Now, to check the condition (4), take any two entries $n \geq i > j \geq 1$ in $\sinkAlgo$ such that 
$i$ and $j$ are in adjacent columns in $\sinkAlgo$, with $j$ positioned lower-right of $i$.
As {\it Step 2} of \cref{Algo: produce the sink tableau}, these $i$ and $j$ are positioned at $(x^{(i)},y^{(i)})$ and $(x^{(j)},y^{(j)})$, respectively.
Then, by the assumption, we have that $x^{(i)} < x^{(j)}$ and $y^{(i)}+1 = y^{(j)}$.
We claim that
\[
(x^{(i)},y^{(j)}) \in \tcd(\alpha) \quad \text{and} \quad \sinkAlgo_{x^{(i)},y^{(j)}} > \sinkAlgo_{x^{(j)},y^{(j)}}.
\]
In {\it Step 2} of \cref{Algo: produce the sink tableau}, $(x^{(i)},y^{(i)})$ is determined as the position of the box containing the last element of the sequence $\{a_k^{(i)}\}$.
This means that each entry in column $y^{(j)}$ of $\tau^{(i)}$ placed below $(x^{(i)},y^{(i)})$ is less than $\tau^{(i)}_{x^{(i)},y^{(i)}}$.
Therefore, we have that  $\tau_{x^{(i)},y^{(i)}} > \tau_{x^{(j)},y^{(j)}}$.
Since $\tau$ is an SRCT, it follows from the definition of SRCTx that 
\[
(x^{(i)},y^{(j)}) \in \tcd(\alpha) \quad \text{and} \quad \tau_{x^{(i)},y^{(j)}} > \tau_{x^{(j)},y^{(j)}}.
\]
The right inequality preserves in $\tau^{(i)}$, so we know that $\tau^{(i)}_{x^{(i)},y^{(j)}} > \tau^{(i)}_{x^{(j)},y^{(j)}}$.
By applying {\it Step 2} of \cref{Algo: produce the sink tableau} to $\tau^{(i)}$ we conclude that $(x^{(i)},y^{(j)})$ is chosen before $(x^{(j)},y^{(j)})$, that is, $\sinkAlgo_{x^{(i)},y^{(j)}} > \sinkAlgo_{x^{(j)},y^{(j)}}$.

We next show that $\sinkAlgo \in E_\tau$.
This can be done by showing
\begin{equation}\label{Claim equivalence}
\rmst_c(\tau) = \rmst_c(\sinkAlgo) \quad \text{for } 1 \leq c \leq \alphamax.
\end{equation}
Given any $1 \leq c \leq \alphamax$, one observes from {\it Step 2} that if $\tau_{i,c} > \tau_{j,c}$, then $
\sinkAlgo_{i,c} > \sinkAlgo_{j,c}$.
This observation yields \cref{Claim equivalence}.

We finally show that $\sinkAlgo$ is the sink tableau in $E_\tau$.
Let us choose $p \in \Des{}{\sinkAlgo}$.
Consider the tableaux $\tau^{(p+1)}$ and $\tau^{(p)}$ defined in \cref{Algo: produce the sink tableau}.
Suppose that the sequence $\{a_k^{(p+1)}\}$ terminates at $l$ for some $l \geq 1$.
We first deal with the case where $l=1$.
In this case $a_1^{(p+1)}$ is in column $y^{(p+1)}$.
And, from (i) of {\it Step 2} we see that the column index of the column containing $a_1^{(p)}$ is greater than or equal to $y^{(p+1)}$.
In addition, the assumption $p \in \Des{}{\sinkAlgo}$ implies that $y^{(p)} \le y^{(p+1)}$.
Thus, $a_1^{(p)}$ is in column $y^{(p+1)}$, and $\{a_k^{(p)}\}$ terminates at $k=1$.
It follows that $p$ is an attacking descent of $\sinkAlgo$.
Next, we deal with the case where $l>1$.
In this case, we have that $y^{(p+1)}-1 \le y^{(p)}$.
And, the assumption $p \in \Des{}{\sinkAlgo}$ implies that $y^{(p)} \le y^{(p+1)}$.
Combining these observations, we conclude $y^{(p)} = y^{(p+1)}$ or $y^{(p+1)}-1$.
The case $y^{(p)} = y^{(p+1)}-1$ and $x^{(p)} < x^{(p+1)}$ cannot occur due to (ii) of {\it Step 2}, and hence $p$ is an attacking descent of $\sinkAlgo$.
Therefore, $\sinkAlgo$ is the sink tableau in $E_\tau$.
\end{proof}

\begin{example}
Let $\alpha = (2,3,2,4) \models 11$ and let  
$E_i$ $(i=1,2,3)$ be the equivalence class containing the source tableau $\tau_i$: 
\[
\tau_1=
\begin{ytableau}
2  &  1 \\ 
5  &  4  & 3 \\
7  &  6 \\ 
11  & 10  & 9 &  8
\end{ytableau} \qquad 
\tau_2=
\begin{ytableau}
2  &  1 \\ 
6  &  5  &  4 \\
7  &  3 \\
11 &  10  &  9  &  8
\end{ytableau} \qquad 
\tau_3=
\begin{ytableau}
3  &  2 \\
6  &  5  &  4 \\
7  &  1 \\
11 &  10 &  9  &  8
\end{ytableau}
\]
It can be easily verified that $\classQS{\alpha} = \{E_1,E_2,E_3\}$.
Now, let us apply \cref{Algo: produce the sink tableau} to $\tau_i$ $(i=1,2,3)$. 
Then, we obtain the sink tableaux: 
\[
\sinkAlgotau{\tau_1} =
\begin{ytableau}
4  &  1 \\ 
6&  5  & 2 \\
9  &  7 \\ 
11  & 10  & 8 &  3
\end{ytableau} \qquad 
\sinkAlgotau{\tau_2} =
\begin{ytableau}
4  &  1 \\
8  &  6  &  2 \\
9  &  5 \\
11 &  10  &  7  &  3
\end{ytableau} \qquad  
\sinkAlgotau{\tau_3} =
\begin{ytableau}
5 & 2 \\ 
8 & 6 & 3 \\ 
9 & 1 \\ 
11 & 10 & 7 & 4
\end{ytableau}
\]
The following table shows $\comp(\Des{}{\tau_i}^{\rmc})$ and $\tread{\sinkAlgotau{\tau_i}}$ for the pairs $(\tau_i,\sinkAlgotau{\tau_i})$ ($i=1,2,3$).
\begin{table}[ht]	
\centering
\tabulinesep=1.2mm
\begin{tabular}{ c|c|c }
\text{class} & $\comp(\Des{}{\tau_i}^{\rmc})$ & $\tread{\sinkAlgotau{\tau_i}}$ \\[1mm] \hline
$E_1$ & $(1,2,1,2^2,1^3)$ & 4 \, 1 \, 6 \, 5 \, 2 \, 9 \, 7 \, 11 \, 10 \, 8 \, 3 \\ 
$E_2$ & $(1,3,1,3,1^3)$ & 4 \, 1 \, 5 \, 8 \, 6 \, 2 \, 9 \, 11 \, 10 \, 7 \, 3  \\ 
$E_3$ & $(2^2,1,3,1^3)$ & 1 \, 5 \, 2 \, 8 \, 6 \, 3 \, 9 \, 11 \, 10 \, 7 \, 4  
\end{tabular} 
\end{table}
\\
Applying \cref{Algo: Construction of  D_alpha_rho} to these data, we obtain  
the Hasse diagram of the poset $P_{\mQS{\alpha,E_i}}$ for each $i=1,2,3$.
\[
\def \ccc {3.5mm}
\def \hhh {1.2em}
\def \vvv {1.2em}
\begin{tikzpicture}[baseline = 14mm]
\node[left] at (\hhh*-1.5,\vvv*2) {\small $P_{\mQS{\alpha,E_1}} =$};
\node[shape=circle,draw,minimum size=\ccc, inner sep=0pt] at (\hhh*3,\vvv*5) (D4) {\tiny $4$};
\node[shape=circle,draw,minimum size=\ccc, inner sep=0pt] at (\hhh*2,\vvv*4) (D3) {\tiny $3$};
\node[shape=circle,draw,minimum size=\ccc, inner sep=0pt] at (\hhh*1,\vvv*3) (D2) {\tiny $2$};
\node[shape=circle,draw,minimum size=\ccc, inner sep=0pt] at (\hhh*-1,\vvv*1) (D1) {\tiny $1$};
\node[shape=circle,draw,minimum size=\ccc, inner sep=0pt] at (\hhh*2,\vvv*2) (D5) {\tiny $5$};
\node[shape=circle,draw,minimum size=\ccc, inner sep=0pt] at (\hhh*3,\vvv*3) (D6) {\tiny $6$};
\node[shape=circle,draw,minimum size=\ccc, inner sep=0pt] at (\hhh*3,\vvv*1) (D9) {\tiny $9$};
\node[shape=circle,draw,minimum size=\ccc, inner sep=0pt] at (\hhh*2,\vvv*0) (D8) {\tiny $8$};
\node[shape=circle,draw,minimum size=\ccc, inner sep=0pt] at (\hhh*1,\vvv*-1) (D7) {\tiny $7$};
\node[shape=circle,draw,minimum size=\ccc, inner sep=0pt] at (\hhh*3,\vvv*-1) (D11) {\tiny $11$};
\node[shape=circle,draw,minimum size=\ccc, inner sep=0pt] at (\hhh*2,\vvv*-2) (D10) {\tiny $10$};

\draw (D1) -- (D2) -- (D3) -- (D4);
\draw (D5) -- (D6);
\draw (D7) -- (D8) -- (D9);
\draw (D10) -- (D11);
\draw[line width=\lw] (D1) -- (D7) -- (D10);
\draw[line width=\lw] (D8) -- (D11);
\draw[line width=\lw] (D2) -- (D5) -- (D9);
\draw[line width=\lw] (D3) -- (D6);
\end{tikzpicture}
\qquad 
\begin{tikzpicture}[baseline = 14mm]
\node[left] at (\hhh*-0.7,\vvv*2) {\small $P_{\mQS{\alpha,E_2}} =$};
\node[shape=circle,draw,minimum size=\ccc, inner sep=0pt] at (\hhh*3,\vvv*5) (D4) {\tiny $4$};
\node[shape=circle,draw,minimum size=\ccc, inner sep=0pt] at (\hhh*2,\vvv*4) (D3) {\tiny $3$};
\node[shape=circle,draw,minimum size=\ccc, inner sep=0pt] at (\hhh*1,\vvv*3) (D2) {\tiny $2$};
\node[shape=circle,draw,minimum size=\ccc, inner sep=0pt] at (\hhh*0,\vvv*2) (D1) {\tiny $1$};
\node[shape=circle,draw,minimum size=\ccc, inner sep=0pt] at (\hhh*3,\vvv*3) (D5) {\tiny $5$};
\node[shape=circle,draw,minimum size=\ccc, inner sep=0pt] at (\hhh*2,\vvv*0) (D6) {\tiny $6$};
\node[shape=circle,draw,minimum size=\ccc, inner sep=0pt] at (\hhh*3,\vvv*1) (D7) {\tiny $7$};
\node[shape=circle,draw,minimum size=\ccc, inner sep=0pt] at (\hhh*4,\vvv*2) (D8) {\tiny $8$};
\node[shape=circle,draw,minimum size=\ccc, inner sep=0pt] at (\hhh*4,\vvv*0) (D9) {\tiny $9$};
\node[shape=circle,draw,minimum size=\ccc, inner sep=0pt] at (\hhh*4,\vvv*-2) (D10) {\tiny $10$};
\node[shape=circle,draw,minimum size=\ccc, inner sep=0pt] at (\hhh*5,\vvv*-1) (D11) {\tiny $11$};

\draw (D1) -- (D2) -- (D3) -- (D4);
\draw (D6) -- (D7) -- (D8);
\draw (D10) -- (D11);
\draw[line width=\lw] (D1) -- (D6) -- (D10);
\draw[line width=\lw] (D2) -- (D7) -- (D9) -- (D11);
\draw[line width=\lw] (D3) -- (D5) -- (D8);
\end{tikzpicture}
\qquad 
\begin{tikzpicture}[baseline = 14mm]
\node[left] at (\hhh*-0.7,\vvv*2) {\small $P_{\mQS{\alpha,E_3}} =$};
\node[shape=circle,draw,minimum size=\ccc, inner sep=0pt] at (\hhh*3,\vvv*5) (D4) {\tiny $4$};
\node[shape=circle,draw,minimum size=\ccc, inner sep=0pt] at (\hhh*2,\vvv*4) (D3) {\tiny $3$};
\node[shape=circle,draw,minimum size=\ccc, inner sep=0pt] at (\hhh*1,\vvv*3) (D2) {\tiny $2$};
\node[shape=circle,draw,minimum size=\ccc, inner sep=0pt] at (\hhh*0,\vvv*2) (D1) {\tiny $1$};
\node[shape=circle,draw,minimum size=\ccc, inner sep=0pt] at (\hhh*3,\vvv*3) (D5) {\tiny $5$};
\node[shape=circle,draw,minimum size=\ccc, inner sep=0pt] at (\hhh*2,\vvv*0) (D6) {\tiny $6$};
\node[shape=circle,draw,minimum size=\ccc, inner sep=0pt] at (\hhh*3,\vvv*1) (D7) {\tiny $7$};
\node[shape=circle,draw,minimum size=\ccc, inner sep=0pt] at (\hhh*4,\vvv*2) (D8) {\tiny $8$};
\node[shape=circle,draw,minimum size=\ccc, inner sep=0pt] at (\hhh*3,\vvv*-1) (D9) {\tiny $9$};
\node[shape=circle,draw,minimum size=\ccc, inner sep=0pt] at (\hhh*4,\vvv*0) (D10) {\tiny $10$};
\node[shape=circle,draw,minimum size=\ccc, inner sep=0pt] at (\hhh*4,\vvv*-2) (D11) {\tiny $11$};

\draw (D1) -- (D2) -- (D3) -- (D4);
\draw (D6) -- (D7) -- (D8);
\draw (D9) -- (D10);
\draw[line width=\lw] (D1) -- (D6) -- (D9) -- (D11);
\draw[line width=\lw] (D2) -- (D7) -- (D10);
\draw[line width=\lw] (D3) -- (D5) -- (D8);
\end{tikzpicture}
\]
\end{example}

\begin{remark}
(a) In \cite[Section 4.1]{22BS}, Bardwell--Searles introduced an $H_n(0)$-action on the set of all standard Young row-strict composition tableaux $\mathrm{SYRT}(\alpha)$ of shape $\alpha$. The resulting $H_n(0)$-module was denoted as $\mRQS{\alpha}$.
Furthermore, they established an equivalence relation on $\mathrm{SYRT}(\alpha)$ and closely examined the class $E_0$, which comprises all tableaux $T \in \mathrm{SYRT}(\alpha)$ satisfying the property that entries increase from bottom to top in every column of $T$.
Notably, they provided a method to find the source tableau in $E_0$ using the notion of \emph{boundary boxes} and \emph{threads}.
Indeed, \cref{Algo: produce the sink tableau} can be viewed as a generalization of this method. 
To be precise, under the isomorphism $\mathcal{W}: \mRQS{\alpha} \ra \upphi \circ \uptheta \circ \upchi [\mQS{\alpha^\rmr}]$ given in \cite[Proposition 1]{22JKLO}, the source tableau $T_{\rm sup}$ in $E_0$ is mapped to $\sinktauE{C_{\alpha^\rmr}}$, where $C_{\alpha^\rmr} \in \classQS{\alpha^\rmr}$ is the class which comprises all tableaux $\tau \in \SRCT(\alpha^\rmr)$ satisfying the property that entries increase from top to bottom in each column of $\tau$.
Concretely, the way of defining the sequence ${a^{(m)}_i}$ in \cref{Algo: produce the sink tableau} is precisely analogous to the method of defining threads.

(b) 
Tewari--van Willigenburg~\cite{19TW} introduced a left $H_n(0)$-module denoted as $\mPQS{\sigma}{\alpha}$ for each permutation $\sigma \in \SG_{\ell(\alpha)}$. This module is constructed by considering a modified version of SRCTx, called {\em standard permuted composition tableaux} (SPCTx).
It is noteworthy that the procedure outlined in \cref{Algo: produce the sink tableau} can be effectively applied to SPCTx.
Utilizing the result from \cite[Theorem 6]{22JKLO}, we can deduce that the assertion in \cref{explicit description for S} is true for $\mPQS{\sigma}{\alpha}$ as well.
To be more precise, for any indecomposable direct summand $\mPQS{\sigma}{\alpha,E}$ of  $\mPQS{\sigma}{\alpha}$, we have 
\[
P_{\mPQS{\sigma}{\alpha,E}}= P_{D_{\comp(\Des{}{\sourcetauE{E}}^{\rmc});\tread{\sinktauE{E}}}}.
\]
\end{remark}

\subsubsection{$\mbalpha \ne \mDIF{\alpha}, \mESF{\alpha}, \mQS{\alpha}$}
In the cases where $\mbalpha = \mRDIF{\alpha}$ or $\mRESF{\alpha}$, we have the isomorphisms
\[
\mRDIF{\alpha} \cong  \uptheta \circ \upchi[\mDIF{\alpha}]
\quad \text{and} \quad \mRESF{\alpha} \cong  \uptheta \circ \upchi[\mESF{\alpha}].
\]
Using this fact along with \cref{thm: auto-twists}, we obtain
\[
P_{\mRDIF{\alpha}} = \overline{P_{\mDIF{\alpha}}} \quad \text{and} \quad P_{\mRESF{\alpha}} = \overline{P_{\mESF{\alpha}}}.
\]
In the cases where $\mbalpha = \mYQS{\alpha}, \mYRQS{\alpha}$, or $\mRQS{\alpha}$, we have the $H_n(0)$-module isomorphisms
\begin{align*}
\mYQS{\alpha} 
\cong \upphi[\mQS{\alpha^\rmr}], \quad 
\mYRQS{\alpha} \cong \uptheta \circ \upchi[\mQS{\alpha}] \quad \text{and} \quad \mRQS{\alpha} \cong \upphi \circ \uptheta \circ \upchi [\mQS{\alpha^\rmr}].
\end{align*}
This implies that every indecomposable direct summand $V \in \setIndSummandY$ appears as $\phi[\mQS{\beta,E}]$ for some $\beta \in {\alpha, \alpha^\rmr}$, $E \in \classQS{\beta}$, and $\phi \in {\upphi, \uptheta \circ \upchi, \upphi \circ \uptheta \circ \upchi}$.
According to \cref{thm: auto-twists}, we can obtain the posets $P_{V}$ by applying the poset isomorphisms ${}^-$ and ${}^*$ to $P_{\mQS{\beta,E}}$. To be precise, we define
\begin{align*}
\mYQS{\alpha,E} := \upphi[\mQS{\alpha^\rmr,E}], \quad 
\mYRQS{\alpha,E} := \uptheta \circ \upchi[\mQS{\alpha,E}] \quad \text{and} \quad 
\mRQS{\alpha,E} := \upphi \circ \uptheta \circ \upchi [\mQS{\alpha^\rmr,E}].
\end{align*} 
Then, we have
\[
P_{\mYQS{\alpha,E}} = (\overline{P_{\mQS{\alpha^\rmr,E}}})^*, \quad P_{\mYRQS{\alpha,E}} = \overline{P_{\mQS{\alpha,E}}} \quad \text{and} \quad 
P_{\mRQS{\alpha,E}} = (P_{\mQS{\alpha^\rmr,E}})^*.
\]

\begin{remark}
In the context of the $K$-theory of Grassmannians, Pechenik--Yong \cite{17PY2} introduced genomic Schur functions $U_\lambda$ ($\lambda$ is a partition) as a genomic analog of Schur functions.
For every $1 \le m \le |\lambda|$, Kim--Yoo \cite{22KY} introduced the left $H_m(0)$-module $\bfG_{\lambda;m}$ such that $\ch([\bfG_{\lambda;m}])$ is equal to the $m$th homogeneous component of $U_\lambda$ and provided a direct sum decomposition
\[
\bfG_{\lambda;m} = \bigoplus_{E \in \calE_{\lambda;m}} \bfG_E.
\]
For the undefined notation $\calE_{\lambda;m}$ and $\bfG_E$, see \cite[Section 3.2]{22KY}.
Further, they proved that for each $E \in \calE_{\lambda;m}$, $\bfG_E \cong \sfB_L(w_0(\alpha), \rho)$ for some $\alpha \models n$ and $\rho \in \SG_n$ with $w_0(\alpha) \preceq_L \rho$.
Therefore, one can obtain a poset $P$ such that $M_P \cong \calF_m(\bfG_E)$ by applying \cref{Algo: Construction of D_alpha_rho} and \cref{Thm: Diagram Dalpharho}.
\end{remark}

\subsection{Tableau descriptions of the coefficients in the quasisymmetric power sum expansions of \texorpdfstring{$\DIF{\alpha}$ and $\ESF{\alpha}$}{Lg}}
\label{Tableau descriptions of the coefficients in the expansions}

In this subsection, we describe how to express the coefficients appearing in the quasisymmetric power sum expansions of two important bases, namely the dual immaculate function $\DIF{\alpha}$ and the extended Schur function $\ESF{\alpha}$, in terms of border strip tableaux.

To help with understanding, we first review Liu--Weselcouch's result on a quasisymmetric power sum expansion of skew Schur functions. 
For more details and definitions, please refer to \cite{21LW}. 
Given a skew partition $\lambda/ \mu$ of $n$, let $\tcd(\lambda / \mu) := \tcd(\lambda) \setminus \tcd(\mu)$, and let $T_{\lambda / \mu}$ be the filling of $\tcd(\lambda / \mu)$ with entries $1,2, \ldots, n$ from bottom to top starting with the leftmost column. 
We define $P_{\lambda / \mu} = ([n], \preceq_{\lambda / \mu})$ to be the poset such that $i \preceq_{\lambda / \mu} j$ if the box filled with $i$ is weakly below and left of the box filled with $j$ in $T_{\lambda / \mu}$.
For example, if $\lambda / \mu = (3,3,2) / (2)$, then 
\[
T_{\lambda / \mu} = \begin{ytableau}
\none & \none & 6 \\
2 & 4 & 5 \\
1 & 3
\end{ytableau} 
\quad \text{and} \quad 
P_{\lambda / \mu} = 
\begin{array}{l}
\begin{tikzpicture}
\def \pp {0.45}
\def \ccc {1mm}
\node[shape=circle,draw,minimum size=\ccc*3, inner sep=0pt] at (0*\pp, 0*\pp) (A6) {\tiny $6$};
\node[shape=circle,draw,minimum size=\ccc*3, inner sep=0pt] at (1*\pp, 1*\pp) (A5) {\tiny $5$};
\node[shape=circle,draw,minimum size=\ccc*3, inner sep=0pt] at (2*\pp, 0*\pp) (A4) {\tiny $4$};
\node[shape=circle,draw,minimum size=\ccc*3, inner sep=0pt] at (3*\pp, 1*\pp) (A3) {\tiny $3$};
\node[shape=circle,draw,minimum size=\ccc*3, inner sep=0pt] at (3*\pp, -1*\pp) (A2) {\tiny $2$};
\node[shape=circle,draw,minimum size=\ccc*3, inner sep=0pt] at (4*\pp, 0*\pp) (A1) {\tiny $1$};

\draw[line width=0.5mm] (A6) -- (A5);
\draw (A5) -- (A4) -- (A2);
\draw[line width=0.5mm] (A4) -- (A3);
\draw (A3) -- (A1);
\draw[line width=0.5mm] (A2) -- (A1);
\end{tikzpicture}
\end{array}.
\]
It is well known that 
$K_{P_{\lambda/ \mu}} = s_{\lambda/ \mu}$ (for instance, see \cite[p361]{99Stanley}).
Given $\alpha \models n$, let 
\[
\chi^{\lambda / \mu}(\alpha) := \sum_{T \in \mathrm{BST}_{s_{\lambda/\mu}}(\alpha)} (-1)^{\mathrm{ht}(T)},
\]
where $\mathrm{BST}_{s_{\lambda/\mu}}(\alpha)$ is the set of all border-strip tableaux of shape $\lambda / \mu$ and type $\alpha$ and $\mathrm{ht}(T)$ is the sum of the heights of the border strips that make up $T$.
By using \cite[Corollary 7.17.4]{99Stanley}, one can easily see that $\chi^{\lambda / \mu}(\alpha) = \chi^{\lambda / \mu}(\widetilde{\alpha})$.
In addition, \cite[Corollary 7.17.5]{99Stanley} says that \[
s_{\lambda/\mu} = \sum_{\nu \vdash n} \chi^{\lambda / \mu}(\nu) p_\nu / z_\nu,
\]
where $p_\nu$ is the power sum symmetric function associated with $\nu$.
Putting these together with \cref{lem: quasi power sum and power sum}, we have
\begin{align}\label{eq: power sum exp of schur}
K_{P_{\lambda/ \mu}} =  \sum_{\beta \models n} 
\left( \sum_{T \in \mathrm{BST}_{s_{\lambda/\mu}}(\beta)}(-1)^{\mathrm{ht}(T)} \right) \frac{\Psi_\beta}{z_\beta}.
\end{align}
On the other hand, due to \cref{Lem: Liu--Weselcouch result}, we have
\begin{align}\label{eq: KP lambda / mu and pt}
K_{P_{\lambda/ \mu}} = \sum_{\beta \models n} \left(
\sum_{f^* \in \mathsf{pt}_{P_{\lambda/ \mu}}(\beta)}
\sign{f^*}
 \right)\frac{\Psi_\beta}{z_\beta}.
\end{align}

Liu--Weselcouch showed in a direct way that the right-hand side of \cref{eq: power sum exp of schur} and that of \cref{eq: KP lambda / mu and pt} are equal.
Let us explain their proof briefly.
Given a starred $P_{\lambda/ \mu}$-partition $f^*$, let $T_{f^*}$ be the filling of $\tcd(\lambda / \mu)$ defined by $T_{f^*}((x,y)) = |f^*(i)|$, where  $(T_{\lambda / \mu})_{x,y} = i$.
For instance, 
\[
\text{if} \quad f^* =\hspace{-1ex}   
\begin{array}{l}
\begin{tikzpicture}
\def \pp {0.6}
\def \ccc {1.2mm}
\node[shape=circle,draw,minimum size=\ccc*3, inner sep=0pt] at (0*\pp, 0*\pp) (A6) {\small $6$};
\node[shape=circle,draw,minimum size=\ccc*3, inner sep=0pt] at (1*\pp, 1*\pp) (A5) {\small $5$};
\node[shape=circle,draw,minimum size=\ccc*3, inner sep=0pt] at (2*\pp, 0*\pp) (A4) {\small $4$};
\node[shape=circle,draw,minimum size=\ccc*3, inner sep=0pt] at (3*\pp, 1*\pp) (A3) {\small $3$};
\node[shape=circle,draw,minimum size=\ccc*3, inner sep=0pt] at (3*\pp, -1*\pp) (A2) {\small $2$};
\node[shape=circle,draw,minimum size=\ccc*3, inner sep=0pt] at (4*\pp, 0*\pp) (A1) {\small $1$};

\draw[line width=0.5mm] (A6) -- (A5);
\draw (A5) -- (A4) -- (A2);
\draw[line width=0.5mm] (A4) -- (A3);
\draw (A3) -- (A1);
\draw[line width=0.5mm] (A2) -- (A1);

\node[left] at (4*\pp - 0.1*\pp, 0*\pp - 0.1*\pp) {\tiny $1^*$};
\node[left] at (3*\pp - 0.1*\pp, -1*\pp - 0.1*\pp) {\tiny $-1$};
\node[left] at (3*\pp - 0.1*\pp, 1*\pp - 0.1*\pp) {\tiny $3^*$};
\node[left] at (2*\pp - 0.1*\pp, 0*\pp - 0.1*\pp) {\tiny $2^*$};
\node[left] at (1*\pp - 0.1*\pp, 1*\pp - 0.1*\pp) {\tiny $2$};
\node[left] at (0*\pp - 0.1*\pp, 0*\pp - 0.1*\pp) {\tiny $-2$};
\end{tikzpicture}
\end{array},
\quad \text{then} \quad
T_{f^*} = 
\begin{ytableau}
\none & \none & 2 \\
1 & 2 & 2 \\
1 & 3
\end{ytableau}\; .
\]
Here, the number written on the lower left of $\circled{i}$ means  $f^*(i)$.
One can easily see that $\mathrm{amb}(f^*) = (1,1,\ldots, 1)$ if and only if  $T_{f^*}$ is a border strip tableau and $\mathrm{wt}(f^*) = \mathrm{type}(T_{f^*})$.
This implies that the map $\Phi:  \mathsf{pt}_{P_{\lambda/ \mu}}(\beta) \ra \mathrm{BST}_{s_{\lambda/\mu}}(\beta)$, $f^* \mapsto T_{f^*}$ is a bijection.
In addition, the equality $\sign{f^*} = (-1)^{\mathrm{ht}(T_{f^*})}$ is obtained straightforwardly.
Thus, 
$$
\sum_{f^* \in \mathsf{pt}_{P_{\lambda/ \mu}}(\beta)}
\sign{f^*} = \sum_{T \in \mathrm{BST}_{s_{\lambda/\mu}}(\beta)}(-1)^{\mathrm{ht}(T)}.
$$

Motivated by the above result, we provide a description of the coefficients appearing in the quasisymmetric power sum expansions of $\DIF{\alpha}$ and $\ESF{\alpha}$ in terms of certain tableaux which we call \emph{border strip tableaux}.

From now on, $\balpha$ is limited to either $\DIF{\alpha}$ or $\ESF{\alpha}$.
By~\cref{Thm: Yalpha to quasipowersum}, we see that 
\begin{equation}\label{Eq: Yalpha to quasipowersums II}
\balpha = \sum_{\beta \models n} d_{\alpha \beta} \frac{\Psi_\beta}{z_\beta}
\quad
\text{ with } \quad  
d_{\alpha\beta} := \sum_{f^*}
\sign{f^*}.
\end{equation}
To begin with, let us introduce the notion of border strips corresponding to $f^*$'s appearing in~\cref{Eq: Yalpha to quasipowersums II}.
We have two cases.

(a) In the case where $\balpha= \DIF{\alpha}$,
a {\em border strip} $B$ is a connected skew diagram in $\tcd(\alpha)$ that is a horizontal strip
or the union of connected horizontal strips that intersect the first column.

(b) In the case where $\balpha=\ESF{\alpha}$,
a {\em border strip} $B$ is a skew diagram in $\tcd(\alpha)$ satisfying the following conditions: 
\begin{itemize}
\item In addition to the notion of the usual connectedness, assuming that two boxes of $B$ in the same column are also considered connected when there are no boxes in $\tcd(\alpha) \setminus B$ in the middle of them, $B$ is connected.
See the following example.
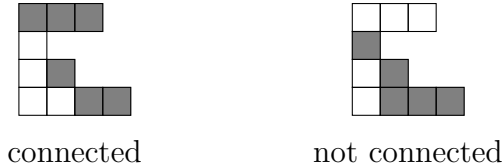
\[
\begin{tikzpicture}
\filldraw[black!50] (0,0) rectangle (\hhh*1,\vvv*1);
\filldraw[black!50] (\hhh*1,0) rectangle (\hhh*2,\vvv*1);
\filldraw[black!50] (\hhh*2,0) rectangle (\hhh*3,\vvv*1);
\filldraw[black!50] (\hhh*2,\vvv*-3) rectangle (\hhh*3,\vvv*-2);
\filldraw[black!50] (\hhh*3,\vvv*-3) rectangle (\hhh*4,\vvv*-2);
\filldraw[black!50] (\hhh*1,\vvv*-2) rectangle (\hhh*2,\vvv*-1);
\draw (0,0) rectangle (\hhh*1,\vvv*1); 
\draw (\hhh*1,0) rectangle (\hhh*2,\vvv*1);  
\draw (\hhh*2,0) rectangle (\hhh*3,\vvv*1);  
\draw (0,\vvv*-1) rectangle (\hhh*1,\vvv*0);  
\draw (0,\vvv*-2) rectangle (\hhh*1,\vvv*-1);  
\draw (\hhh*1,\vvv*-2) rectangle (\hhh*2,\vvv*-1);  
\draw (0,\vvv*-3) rectangle (\hhh*1,\vvv*-2);  
\draw (\hhh*1,\vvv*-3) rectangle (\hhh*2,\vvv*-2);  
\draw (\hhh*2,\vvv*-3) rectangle (\hhh*3,\vvv*-2);  
\draw (\hhh*3,\vvv*-3) rectangle (\hhh*4,\vvv*-2);  
\node[below] at (\hhh*2,\vvv*-3.5) {connected};
\end{tikzpicture}
\hspace*{20mm}
\begin{tikzpicture}
\filldraw[black!50] (\hhh*1,\vvv*-2) rectangle (\hhh*2,\vvv*-1);
\filldraw[black!50] (\hhh*1,\vvv*-3) rectangle (\hhh*2,\vvv*-2);
\filldraw[black!50] (\hhh*2,\vvv*-3) rectangle (\hhh*3,\vvv*-2);
\filldraw[black!50] (\hhh*3,\vvv*-3) rectangle (\hhh*4,\vvv*-2);
\filldraw[black!50] (\hhh*0,\vvv*-1) rectangle (\hhh*1,\vvv*-0);
\draw (0,0) rectangle (\hhh*1,\vvv*1); 
\draw (\hhh*1,0) rectangle (\hhh*2,\vvv*1);  
\draw (\hhh*2,0) rectangle (\hhh*3,\vvv*1);  
\draw (0,\vvv*-1) rectangle (\hhh*1,\vvv*0);  
\draw (0,\vvv*-2) rectangle (\hhh*1,\vvv*-1);  
\draw (\hhh*1,\vvv*-2) rectangle (\hhh*2,\vvv*-1);  
\draw (0,\vvv*-3) rectangle (\hhh*1,\vvv*-2);  
\draw (\hhh*1,\vvv*-3) rectangle (\hhh*2,\vvv*-2);  
\draw (\hhh*2,\vvv*-3) rectangle (\hhh*3,\vvv*-2);  
\draw (\hhh*3,\vvv*-3) rectangle (\hhh*4,\vvv*-2);  
\node[below] at (\hhh*2,\vvv*-3.5) {not connected};
\end{tikzpicture}
\]

\item If $(x,z), (y,z) \in B$ with $x<y$, then $(x,z-1) \notin B$.
\end{itemize}

In either case, 
the height ${\rm ht}(B)$ of a border strip $B$ is defined to be one less than the number of rows.
For example, the following skew diagrams are border strips of size $4$.
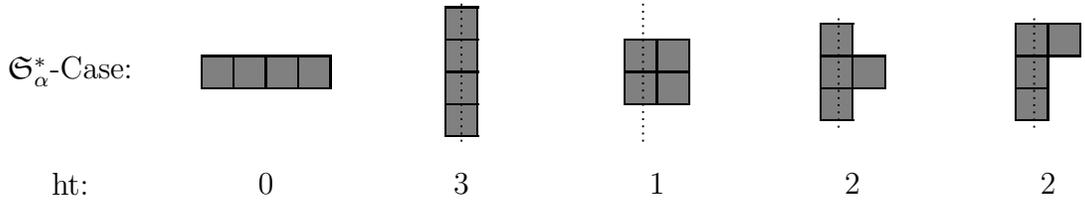
\[
\begin{tikzpicture}
\def \hhhh{26mm}
\node at (-\hhhh*0,0) (A1) {\begin{ytableau}
*(black!50) & *(black!50) & *(black!50) & *(black!50)
\end{ytableau}};    
\node at (\hhhh,0) (A2) {
\begin{ytableau}
*(black!50) \\
*(black!50) \\
*(black!50) \\
*(black!50)
\end{ytableau}};
\draw[-,dotted,line width=\lw*0.5] (\hhhh,-\vvv*2.5) -- (\hhhh,\vvv*2.5); 
\node at (\hhhh*2,0) (A3) {
\begin{ytableau}
*(black!50) & *(black!50) \\
*(black!50) & *(black!50) \\
\end{ytableau}};
\draw[-,dotted,line width=\lw*0.5] (\hhhh*2-\hhh*0.5,-\vvv*2.5) -- (\hhhh*2-\hhh*0.5,\vvv*2.5); 
\node at (\hhhh*3,0) (A4) {
\begin{ytableau}
*(black!50)  \\
*(black!50)  & *(black!50) \\
*(black!50) 
\end{ytableau}};
\draw[-,dotted,line width=\lw*0.5] (\hhhh*3-\hhh*0.5,-\vvv*2) -- (\hhhh*3-\hhh*0.5,\vvv*2); 
\node at (\hhhh*4,0) (A5) {
\begin{ytableau}
*(black!50) & *(black!50) \\
*(black!50)  \\
*(black!50)  
\end{ytableau}};
\draw[-,dotted,line width=\lw*0.5] (\hhhh*4-\hhh*0.5,-\vvv*2) -- (\hhhh*4-\hhh*0.5,\vvv*2); 

\node at (\hhhh*-1,0) {$\DIF{\alpha}$-Case:};
\node at (\hhhh*-1,\vvv*-4) {{\rm ht}:};
\node at (\hhhh*0,\vvv*-4) {$0$};
\node at (\hhhh*1,\vvv*-4) {$3$};
\node at (\hhhh*2,\vvv*-4) {$1$};
\node at (\hhhh*3,\vvv*-4) {$2$};
\node at (\hhhh*4,\vvv*-4) {$2$};
\end{tikzpicture}
\]
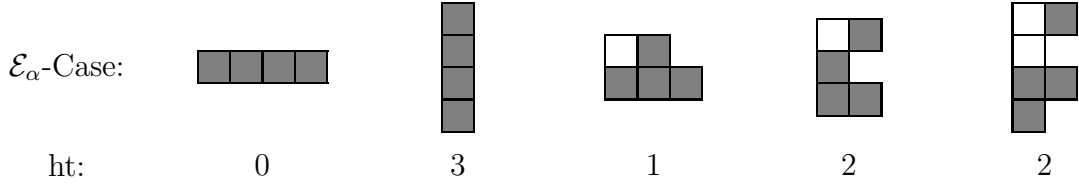
\[
\begin{tikzpicture}
\def \hhhh{26mm}
\node at (-\hhhh*0,0) (A1) {\begin{ytableau}
*(black!50) & *(black!50) & *(black!50) & *(black!50) 
\end{ytableau}};    
\node at (\hhhh*1,0) (A2) {\begin{ytableau}
*(black!50) \\
*(black!50) \\
*(black!50) \\
*(black!50) 
\end{ytableau}};    
\node at (\hhhh*2,0) (A3) {
\begin{ytableau}
\empty & *(black!50) \\
*(black!50) & *(black!50) & *(black!50) 
\end{ytableau}};    
\node at (\hhhh*3,0) (A4) {
\begin{ytableau}
\empty & *(black!50) \\
*(black!50) \\
*(black!50) & *(black!50)
\end{ytableau}};
\node at (\hhhh*4,0) (A5) {
\begin{ytableau}
\empty & *(black!50)  \\
\empty \\
*(black!50) & *(black!50) \\
*(black!50) 
\end{ytableau}};
\node at (\hhhh*-1,0) {$\ESF{\alpha}$-Case:};
\node at (\hhhh*-1,\vvv*-3.5) {{\rm ht}:};
\node at (\hhhh*0,\vvv*-3.5) {$0$};
\node at (\hhhh*1,\vvv*-3.5) {$3$};
\node at (\hhhh*2,\vvv*-3.5) {$1$};
\node at (\hhhh*3,\vvv*-3.5) {$2$};
\node at (\hhhh*4,\vvv*-3.5) {$2$};
\end{tikzpicture}
\]
Here the dotted lines denote the first column.

\begin{definition}\label{def of border strip tableaux}
Let $\alpha,\beta$ be compositions with $|\alpha|=|\beta|$.

(a) In the case where $\balpha = \DIF{\alpha}$, 
a \emph{border-strip tableau of shape $\alpha$ and type $\beta$} is defined to be an assignment of positive integers to the boxes in $\tcd(\alpha)$ such that
\begin{itemize}
\item the entries in every row weakly increase from left to right,

\item 
the entries in the first column weakly increase from top to bottom,

\item the integer $i$ appears $\beta_i$ times, and

\item the set of boxes occupied by $i$ forms a border strip associated with $\DIF{\alpha}$.   
\end{itemize}

(b) In the case where $\balpha = \ESF{\alpha}$, a \emph{border strip tableau of shape $\alpha$ and type $\beta$} is defined to be 
an assignment of positive integers to the boxes in $\tcd(\alpha)$ such that
\begin{itemize}
\item the entries in every row weakly increase  from left to right,

\item the entries in every column  weakly increase from top to bottom,

\item the integer $i$ appears $\beta_i$ times, and

\item the set of boxes occupied by $i$ forms a border strip associated with $\ESF{\alpha}$.   
\end{itemize}
\end{definition}
Let $\BSTx{\balpha}{\beta}$ be the set of border strip tableaux of shape $\alpha$ and type $\beta$ and associated with $\balpha$.
For each $T \in \BSTx{\balpha}{\beta}$,
we define $\htt{\balpha}{T}$ to be the sum of the heights of the border strips (associated with $\balpha$) that make up $T$.
Then we can derive the following tableau description.

\begin{corollary}\label{tableau description}
With the notation in~\cref{Eq: Yalpha to quasipowersums II}, we have
\[
d_{\alpha\beta}  =  \sum_{T \in \BSTx{\balpha}{\beta}} (-1)^{\htt{\balpha}{T}}.
\]
\end{corollary}

\begin{proof}
For the sake of simplicity, we only deal with the case where $Y_\alpha = \DIF{\alpha}$. The other case, $Y_\alpha = \ESF{\alpha}$, can be proven in the same way.

For $f^* \in \mathsf{pt}_{P_{\mDIF{\alpha}}}(\beta)$ and $1 \le i \le \ell(\beta)$, let
\[
P^{f^*}_i := (f^*)^{-1} \left( \{x \in \bfP^* \mid |x| = i\} \right),
\]
where the definition of  $\mathsf{pt}_{P_{\mDIF{\alpha}}}(\beta)$ can be found in \cref{eq: sfpt_P}.
Recall the diagram $D_{\alpha^\rmc; \rmread(\sinkSIT{\alpha})}$ in \cref{Thm: posets for V and X}(b).
For $1 \leq i \leq n$, we denote by $(x_i,y_i)$ the $i$th element when enumerating the elements in this diagram along the rows from left to right starting with the uppermost row.
By definition, we know that $k \preceq_{P_{\mDIF{\alpha}}} l$ if and only if $x_k \le x_l$ and $y_k \le y_l$ for all $k, l \in [n]$.
For $j \in P^{f^*}_i$, let $B_j = (r_j, c_j) \in \tcd(\alpha)$, where
\[
r_j = \ell(\alpha) - y_j + 1 
\quad \text{and} \quad
c_j = |\{ l \in [n] \mid y_l = y_j \text{ and } x_l \le x_j \}|.
\]
Given $f^* \in \mathsf{pt}_{P_{\mDIF{\alpha}}}(\beta)$, let $T_{f^*}$ be the filling of $\tcd(\alpha)$ defined by
\[
T_{f^*} (B_j) = i \quad \text{for all $j \in P^{f^*}_i$ and $1 \le i \le \ell(\beta)$}.
\]
One sees that the condition $\mathrm{amb}(f^*) = (1^{\ell(\beta)})$ implies that $T_{f^*}$ is a border strip tableau.
And the condition $\wt(f^*) = \beta$ implies that $T_{f^*}$ is of type $\beta$.
Therefore, we have a well-defined map $\Phi: \mathsf{pt}_{P_{\mDIF{\alpha}}}(\beta) \ra \BSTx{\DIF{\alpha}}{\beta}, \ f^* \mapsto T_{f^*}$.
This map is bijective and $\sign{f^*} = (-1)^{\mathrm{ht}_{\DIF{\alpha}}(\Phi(f^*))}$.
Thus,
\[
\sum_{f^* \in \mathsf{pt}_{P_{\mDIF{\alpha}}}(\beta)}
\sign{f^*}
=  \sum_{T \in \BSTx{\DIF{\alpha}}{\beta}} (-1)^{\htt{\DIF{\alpha}}{T}}. \qedhere
\]   
\end{proof}

\begin{example}
If $\alpha = (2,1,2)$ and $\beta = (4,1)$, then
\[
\def \pp {0.6}
\def \ccc {3.6mm}
\mathsf{pt}_{P_{\calV_{\alpha}}}(\beta) = 
\left\{
f^*_1 = 
\begin{array}{l}
\begin{tikzpicture}[baseline=5mm]
\node[shape=circle,draw,minimum size=\ccc, inner sep=0pt] at (\pp*0,\pp*2) (D1) {\small $1$};
\node[shape=circle,draw,minimum size=\ccc, inner sep=0pt] at (\pp*1,\pp*3) (D2) {\small $2$};
\node[shape=circle,draw,minimum size=\ccc, inner sep=0pt] at (\pp*1,\pp*1) (D3) {\small $3$};
\node[shape=circle,draw,minimum size=\ccc, inner sep=0pt] at (\pp*2,\pp*0) (D4) {\small $4$};
\node[shape=circle,draw,minimum size=\ccc, inner sep=0pt] at (\pp*4,\pp*2) (D5) {\small $5$};

\node[left] at (0*\pp - 0.1*\pp, 2*\pp - 0.1*\pp) {\tiny $1^*$};
\node[left] at (1*\pp - 0.1*\pp, 3*\pp - 0.1*\pp) {\tiny $2^*$};
\node[left] at (1*\pp - 0.1*\pp, 1*\pp - 0.1*\pp) {\tiny $-1$};
\node[left] at (2*\pp - 0.1*\pp, 0*\pp - 0.1*\pp) {\tiny $-1$};
\node[left] at (4*\pp - 0.1*\pp, 2*\pp - 0.1*\pp) {\tiny $1$};

\draw (D1) -- (D2);
\draw (D3);
\draw (D4) -- (D5);
\draw[line width=\lw] (D1) -- (D3) -- (D4);
\end{tikzpicture}
\end{array},
\ \ 
f^*_2 = 
\begin{array}{l}
\begin{tikzpicture}[baseline=5mm]
\node[shape=circle,draw,minimum size=\ccc, inner sep=0pt] at (\pp*0,\pp*2) (D1) {\small $1$};
\node[shape=circle,draw,minimum size=\ccc, inner sep=0pt] at (\pp*1,\pp*3) (D2) {\small $2$};
\node[shape=circle,draw,minimum size=\ccc, inner sep=0pt] at (\pp*1,\pp*1) (D3) {\small $3$};
\node[shape=circle,draw,minimum size=\ccc, inner sep=0pt] at (\pp*2,\pp*0) (D4) {\small $4$};
\node[shape=circle,draw,minimum size=\ccc, inner sep=0pt] at (\pp*4,\pp*2) (D5) {\small $5$};

\node[left] at (0*\pp - 0.1*\pp, 2*\pp - 0.1*\pp) {\tiny $1^*$};
\node[left] at (1*\pp - 0.1*\pp, 3*\pp - 0.1*\pp) {\tiny $1$};
\node[left] at (1*\pp - 0.1*\pp, 1*\pp - 0.1*\pp) {\tiny $-1$};
\node[left] at (2*\pp - 0.1*\pp, 0*\pp - 0.1*\pp) {\tiny $-1$};
\node[left] at (4*\pp - 0.1*\pp, 2*\pp - 0.1*\pp) {\tiny $2^*$};

\draw (D1) -- (D2);
\draw (D3);
\draw (D4) -- (D5);
\draw[line width=\lw] (D1) -- (D3) -- (D4);
\end{tikzpicture}
\end{array}
\right\}
\]
and
\[
\def \hhhh {-20mm}
\BSTx{\DIF{\alpha}}{\beta} =
\left\{ 
T_{f^*_1} = 
\begin{array}{l}
\begin{tikzpicture}
\draw (0+\hhhh,0) -- (0+\hhhh,\vvv*1) --  (\hhh*1+\hhhh,\vvv*1) -- (\hhh*1+\hhhh,\vvv*-1) -- (\hhh*2+\hhhh,\vvv*-1) -- (\hhh*2+\hhhh,\vvv*-2) -- (0+\hhhh,\vvv*-2) -- (0+\hhhh,0);
\draw (\hhh*1+\hhhh,\vvv*0) rectangle (\hhh*2+\hhhh,\vvv*1);
\node at (\hhh*0.5+\hhhh,\vvv*0.5) {$1$};
\node at (\hhh*1.5+\hhhh,\vvv*0.5) {$2$};
\node at (\hhh*0.5+\hhhh,\vvv*-0.5) {$1$};
\node at (\hhh*0.5+\hhhh,\vvv*-1.5) {$1$};
\node at (\hhh*1.5+\hhhh,\vvv*-1.5) {$1$};
\end{tikzpicture}
\end{array},\ \ 
T_{f^*_2} = 
\begin{array}{l}
\begin{tikzpicture}
\draw (0,0) -- (0,\vvv*1) --  (\hhh*2,\vvv*1) -- (\hhh*2,0) -- (\hhh*1,0) -- (\hhh*1,\vvv*-2) -- (\hhh*0,\vvv*-2) -- (0,0);
\draw (\hhh*1,\vvv*-2) rectangle (\hhh*2,\vvv*-1);
\node at (\hhh*0.5,\vvv*0.5) {$1$};
\node at (\hhh*1.5,\vvv*0.5) {$1$};
\node at (\hhh*0.5,\vvv*-0.5) {$1$};
\node at (\hhh*0.5,\vvv*-1.5) {$1$};
\node at (\hhh*1.5,\vvv*-1.5) {$2$};
\end{tikzpicture}
\end{array}
\right\}.
\]
One can see that 
\[
\Phi(P_1) = T_{f^*_1} \quad \text{and} \quad \Phi(P_2) = T_{f^*_2}.
\]
Each of \cref{Thm: Yalpha to quasipowersum} and \cref{tableau description} implies that $d_{\alpha\beta} = 2$.
Therefore,
\[
\left[\frac{\Psi_{(4, 1)}}{z_{(4,1)}} \right]\DIF{(2,1,2)} = 2.
\]
Here $\left[\frac{\Psi_{(4, 1)}}{z_{(4,1)}} \right]\DIF{(2,1,2)}$ is the coefficient of $\frac{\Psi_{(4, 1)}}{z_{(4,1)}}$ in $\DIF{(2,1,2)}$.
Similarly, one can see that $\BSTx{\ESF{\alpha}}{\beta} = \{T_2\}$, and thus $\left[\frac{\Psi_{(4, 1)}}{z_{(4,1)}} \right]\ESF{(2,1,2)} = 1$.
\end{example}

In the special case where all parts of $\beta$ are the same, 
\cref{tableau description} is presented in a very simple form.
It can also be considered as a dual immaculate analogue of~\cite[Corollary 30]{22LO}.

\begin{proposition}
Let $s \in \N$.
If $\alpha \models  m s$, then we have 
\[
d_{\alpha (s^m)} = \epsilon \cdot |\BSTx{\DIF{\alpha}}{(s^m)}|,
\]
where $\epsilon = (-1)^{\htt{\DIF{\alpha}}{T}}$ for any $T \in \BSTx{\DIF{\alpha}}{(s^m)}$.
\end{proposition}
\begin{proof}
If $\BSTx{\DIF{\alpha}}{(s^m)} = \emptyset$, then there is nothing to prove. 
From now on we assume that $\BSTx{\DIF{\alpha}}{(s^m)} \neq \emptyset$. 
Let $T \in \BSTx{\DIF{\alpha}}{(s^m)}$ and $1 \leq i \leq m$.
Let $T^{-1}(i)$ be the border strip in $T$ occupied by the boxes containing $i$.
By the definition of border strip tableaux, if $T^{-1}(i)$ does not intersect the first column, then $T^{-1}(i)$ is a connected horizontal strip.
Thus $\mathrm{ht}(T^{-1}(i)) = 0$.

Let 
\[
I(T) := \{1 \leq i \leq m \mid \text{$T^{-1}(i)$ intersects the first column}\}.
\]
For any $T \in \BSTx{\DIF{\alpha}}{(s^m)}$, the diagram composed of the border strips $T^{-1}(i)$ $(i \in I(T))$ is $\tcd(\ova)$, where $\ova = (\ova_1,\ova_2,\ldots,\ova_{\ell(\alpha)})$ is the composition with 
\[
\ova_i = \alpha_i - s \cdot \max\{p \mid \alpha_i - s \cdot p > 0\}.
\]
It is clear that there is only one way to decompose $\tcd(\ova)$ into border strips of size $s$. 
Therefore, $\htt{\DIF{\alpha}}{T} = \htt{\DIF{\alpha}}{T'}$ for any $T, T' \in \BSTx{\DIF{\alpha}}{(s^m)}$, which completes the proof.
\end{proof}

\begin{example}
We illustrate in the figure the border strip tableaux in $\BSTx{\DIF{(5,2,1,8)}}{(4^4)}$.
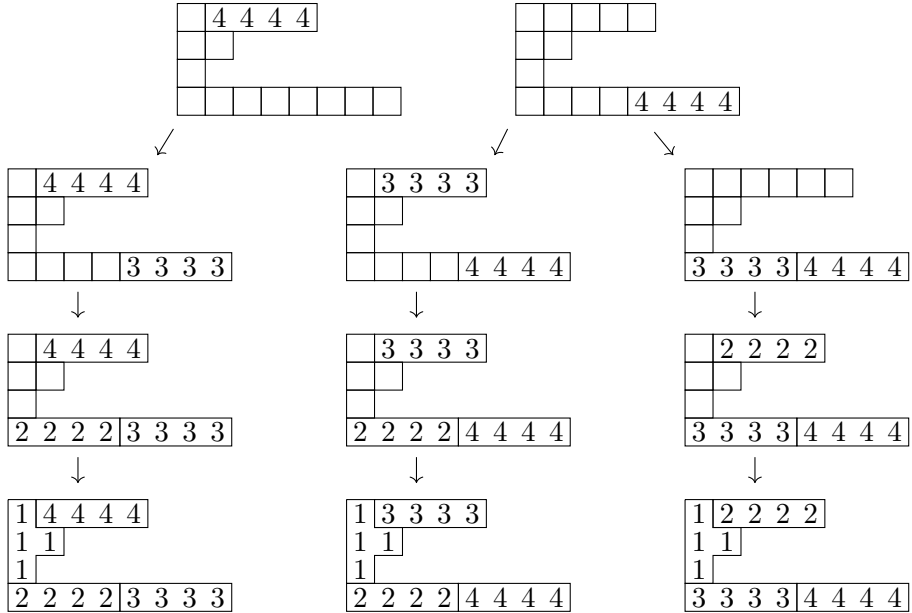
\[
\begin{tikzpicture}
\def \hp {45mm}
\def \vp {22mm}
\node at (\hhh*0.1,\vvv*-3.1) (A1) {};
\foreach \c in {0,...,3}{
    \draw (\hhh*0,\vvv*0-\vvv*\c) rectangle (\hhh*1,\vvv*1-\vvv*\c);
}
\foreach \c in {1,...,4}{
    \node at (\hhh*0.5+\hhh*\c,\vvv*0.5) {\small $4$};
}
\draw (\hhh*1,\vvv*0) rectangle (\hhh*5,\vvv*1);
\draw (\hhh*1,\vvv*-1) rectangle (\hhh*2,\vvv*0);
\foreach \c in {0,...,6}{
    \draw (\hhh*1+\hhh*\c,\vvv*-3) rectangle (\hhh*2+\hhh*\c,\vvv*-2);
}
\node at (\hp+\hhh*-0.1,\vvv*-3.1) (A2) {};
\foreach \c in {0,...,3}{
    \draw (\hhh*0+\hp,\vvv*0-\vvv*\c) rectangle (\hhh*1+\hp,\vvv*1-\vvv*\c);
}
\foreach \c in {0,...,3}{
    \draw (\hhh*1+\hhh*\c+\hp,\vvv*0) rectangle (\hhh*2+\hhh*\c+\hp,\vvv*1);
}
\draw (\hhh*1+\hp,\vvv*-1) rectangle (\hhh*2+\hp,\vvv*0);
\foreach \c in {0,...,2}{
    \draw (\hhh*1+\hhh*\c+\hp,\vvv*-3) rectangle (\hhh*2+\hhh*\c+\hp,\vvv*-2);
}
\draw (\hhh*4+\hp,\vvv*-3) rectangle (\hhh*8+\hp,\vvv*-2);
\foreach \c in {1,...,4}{
    \node at (\hhh*3.5+\hhh*\c+\hp,\vvv*-2.5) {\small $4$};
}

\node at (\hhh*5.1+\hp*-0.5,-\vp+\vvv*1.1) (B1) {};
\foreach \c in {0,...,3}{
    \draw (\hhh*0+\hp*-0.5,\vvv*0-\vvv*\c-\vp) rectangle (\hhh*1+\hp*-0.5,\vvv*1-\vvv*\c-\vp);
}
\foreach \c in {1,...,4}{
    \node at (\hhh*0.5+\hhh*\c+\hp*-0.5,\vvv*0.5-\vp) {\small $4$};
}
\draw (\hhh*1+\hp*-0.5,\vvv*0-\vp) rectangle (\hhh*5+\hp*-0.5,\vvv*1-\vp);
\draw (\hhh*1+\hp*-0.5,\vvv*-1-\vp) rectangle (\hhh*2+\hp*-0.5,\vvv*0-\vp);
\foreach \c in {0,...,2}{
    \draw (\hhh*1+\hhh*\c+\hp*-0.5,\vvv*-3-\vp) rectangle (\hhh*2+\hhh*\c+\hp*-0.5,\vvv*-2-\vp);
}
\draw (\hhh*4+\hp*-0.5,\vvv*-3-\vp) rectangle (\hhh*8+\hp*-0.5,\vvv*-2-\vp);
\foreach \c in {1,...,4}{
    \node at (\hhh*3.5+\hhh*\c+\hp*-0.5,\vvv*-2.5-\vp) {\small $3$};
}
\node at (\hhh*5.1+\hp*0.5,-\vp+\vvv*1.1) (B2) {};
\foreach \c in {0,...,3}{
    \draw (\hhh*0+\hp*0.5,\vvv*0-\vvv*\c-\vp) rectangle (\hhh*1+\hp*0.5,\vvv*1-\vvv*\c-\vp);
}
\foreach \c in {1,...,4}{
    \node at (\hhh*0.5+\hhh*\c+\hp*0.5,\vvv*0.5-\vp) {\small $3$};
}
\draw (\hhh*1+\hp*0.5,\vvv*0-\vp) rectangle (\hhh*5+\hp*0.5,\vvv*1-\vp);
\draw (\hhh*1+\hp*0.5,\vvv*-1-\vp) rectangle (\hhh*2+\hp*0.5,\vvv*0-\vp);
\foreach \c in {0,...,2}{
    \draw (\hhh*1+\hhh*\c+\hp*0.5,\vvv*-3-\vp) rectangle (\hhh*2+\hhh*\c+\hp*0.5,\vvv*-2-\vp);
}
\draw (\hhh*4+\hp*0.5,\vvv*-3-\vp) rectangle (\hhh*8+\hp*0.5,\vvv*-2-\vp);
\foreach \c in {1,...,4}{
    \node at (\hhh*3.5+\hhh*\c+\hp*0.5,\vvv*-2.5-\vp) {\small $4$};
}
\node at (\hhh*-0.1+\hp*1.5,-\vp+\vvv*1.1) (B3) {};
\foreach \c in {0,...,5}{
    \draw (\hhh*\c+\hp*1.5,\vvv*0-\vp) rectangle (\hhh*\c+\hhh+\hp*1.5,\vvv*1-\vp);
}
\foreach \c in {0,1}{
    \draw (\hhh*\c+\hp*1.5,\vvv*-1-\vp) rectangle (\hhh*\c+\hhh+\hp*1.5,\vvv*0-\vp);
}
\draw (\hp*1.5,\vvv*-2-\vp) rectangle (\hhh*1+\hp*1.5,\vvv*-1-\vp);
\draw (\hhh*0+\hp*1.5,\vvv*-3-\vp) rectangle (\hhh*4+\hp*1.5,\vvv*-2-\vp);
\draw (\hhh*4+\hp*1.5,\vvv*-3-\vp) rectangle (\hhh*8+\hp*1.5,\vvv*-2-\vp);
\foreach \c in {1,...,4}{
    \node at (\hhh*3.5+\hhh*\c+\hp*1.5,\vvv*-2.5-\vp) {\small $4$};
}
\foreach \c in {1,...,4}{
    \node at (\hhh*-0.5+\hhh*\c+\hp*1.5,\vvv*-2.5-\vp) {\small $3$};
}

\node at (\hhh*2.5+\hp*-0.5,\vp*-2+\vvv*1.2) (C1) {};
\foreach \c in {0}{
    \draw (\hhh*\c+\hp*-0.5,\vvv*0+\vp*-2) rectangle (\hhh*\c+\hhh+\hp*-0.5,\vvv*1+\vp*-2);
}
\foreach \c in {0,1}{
    \draw (\hhh*\c+\hp*-0.5,\vvv*-1+\vp*-2) rectangle (\hhh*\c+\hhh+\hp*-0.5,\vvv*0+\vp*-2);
}
\draw (\hhh*1+\hp*-0.5,\vvv*0+\vp*-2) rectangle (\hhh*4+\hhh+\hp*-0.5,\vvv*1+\vp*-2);
\draw (\hp*-0.5,\vvv*-2+\vp*-2) rectangle (\hhh*1+\hp*-0.5,\vvv*-1+\vp*-2);
\draw (\hhh*0+\hp*-0.5,\vvv*-3+\vp*-2) rectangle (\hhh*4+\hp*-0.5,\vvv*-2+\vp*-2);
\draw (\hhh*4+\hp*-0.5,\vvv*-3+\vp*-2) rectangle (\hhh*8+\hp*-0.5,\vvv*-2+\vp*-2);
\foreach \c in {1,...,4}{
    \node at (\hhh*0.5+\hhh*\c+\hp*-0.5,\vvv*0.5+\vp*-2) {\small $4$};
}
\foreach \c in {1,...,4}{
    \node at (\hhh*3.5+\hhh*\c+\hp*-0.5,\vvv*-2.5+\vp*-2) {\small $3$};
}
\foreach \c in {1,...,4}{
    \node at (\hhh*-0.5+\hhh*\c+\hp*-0.5,\vvv*-2.5+\vp*-2) {\small $2$};
}
\node at (\hhh*2.5+\hp*0.5,\vp*-2+\vvv*1.2)  (C2) {};
\foreach \c in {0}{
    \draw (\hhh*\c+\hp*0.5,\vvv*0+\vp*-2) rectangle (\hhh*\c+\hhh+\hp*0.5,\vvv*1+\vp*-2);
}
\foreach \c in {0,1}{
    \draw (\hhh*\c+\hp*0.5,\vvv*-1+\vp*-2) rectangle (\hhh*\c+\hhh+\hp*0.5,\vvv*0+\vp*-2);
}
\draw (\hhh*1+\hp*0.5,\vvv*0+\vp*-2) rectangle (\hhh*4+\hhh+\hp*0.5,\vvv*1+\vp*-2);
\draw (\hp*0.5,\vvv*-2+\vp*-2) rectangle (\hhh*1+\hp*0.5,\vvv*-1+\vp*-2);
\draw (\hhh*0+\hp*0.5,\vvv*-3+\vp*-2) rectangle (\hhh*4+\hp*0.5,\vvv*-2+\vp*-2);
\draw (\hhh*4+\hp*0.5,\vvv*-3+\vp*-2) rectangle (\hhh*8+\hp*0.5,\vvv*-2+\vp*-2);
\foreach \c in {1,...,4}{
    \node at (\hhh*0.5+\hhh*\c+\hp*0.5,\vvv*0.5+\vp*-2) {\small $3$};
}
\foreach \c in {1,...,4}{
    \node at (\hhh*3.5+\hhh*\c+\hp*0.5,\vvv*-2.5+\vp*-2) {\small $4$};
}
\foreach \c in {1,...,4}{
    \node at (\hhh*-0.5+\hhh*\c+\hp*0.5,\vvv*-2.5+\vp*-2) {\small $2$};
}
\node at (\hhh*2.5+\hp*1.5,\vp*-2+\vvv*1.2)  (C3) {};
\foreach \c in {0}{
    \draw (\hhh*\c+\hp*1.5,\vvv*0+\vp*-2) rectangle (\hhh*\c+\hhh+\hp*1.5,\vvv*1+\vp*-2);
}
\foreach \c in {0,1}{
    \draw (\hhh*\c+\hp*1.5,\vvv*-1+\vp*-2) rectangle (\hhh*\c+\hhh+\hp*1.5,\vvv*0+\vp*-2);
}
\draw (\hhh*1+\hp*1.5,\vvv*0+\vp*-2) rectangle (\hhh*4+\hhh+\hp*1.5,\vvv*1+\vp*-2);
\draw (\hp*1.5,\vvv*-2+\vp*-2) rectangle (\hhh*1+\hp*1.5,\vvv*-1+\vp*-2);
\draw (\hhh*0+\hp*1.5,\vvv*-3+\vp*-2) rectangle (\hhh*4+\hp*1.5,\vvv*-2+\vp*-2);
\draw (\hhh*4+\hp*1.5,\vvv*-3+\vp*-2) rectangle (\hhh*8+\hp*1.5,\vvv*-2+\vp*-2);
\foreach \c in {1,...,4}{
    \node at (\hhh*0.5+\hhh*\c+\hp*1.5,\vvv*0.5+\vp*-2) {\small $2$};
}
\foreach \c in {1,...,4}{
    \node at (\hhh*3.5+\hhh*\c+\hp*1.5,\vvv*-2.5+\vp*-2) {\small $4$};
}
\foreach \c in {1,...,4}{
    \node at (\hhh*-0.5+\hhh*\c+\hp*1.5,\vvv*-2.5+\vp*-2) {\small $3$};
}
\node at (\hhh*2.5+\hp*-0.5,\vp*-3+\vvv*1.2) (D1) {};
\draw (\hhh*0+\hp*-0.5,\vvv*1+\vp*-3) -- (\hhh*1+\hp*-0.5,\vvv*1+\vp*-3) -- (\hhh*1+\hp*-0.5,\vvv*0+\vp*-3) -- (\hhh*2+\hp*-0.5,\vvv*0+\vp*-3) -- (\hhh*2+\hp*-0.5,\vvv*-1+\vp*-3) -- (\hhh*1+\hp*-0.5,\vvv*-1+\vp*-3) -- (\hhh*1+\hp*-0.5,\vvv*-2+\vp*-3) -- (\hhh*0+\hp*-0.5,\vvv*-2+\vp*-3) -- (\hhh*0+\hp*-0.5,\vvv*1+\vp*-3);
\draw (\hhh+\hp*-0.5,\vvv*0+\vp*-3) rectangle (\hhh*5+\hp*-0.5,\vvv*1+\vp*-3);
\draw (\hhh*0+\hp*-0.5,\vvv*-3+\vp*-3) rectangle (\hhh*4+\hp*-0.5,\vvv*-2+\vp*-3);
\draw (\hhh*4+\hp*-0.5,\vvv*-3+\vp*-3) rectangle (\hhh*8+\hp*-0.5,\vvv*-2+\vp*-3);
\node at (\hhh*0.5+\hp*-0.5,\vvv*0.5+\vp*-3) {\small $1$};
\node at (\hhh*0.5+\hp*-0.5,\vvv*-0.5+\vp*-3) {\small $1$};
\node at (\hhh*1.5+\hp*-0.5,\vvv*-0.5+\vp*-3) {\small $1$};
\node at (\hhh*0.5+\hp*-0.5,\vvv*-1.5+\vp*-3) {\small $1$};
\foreach \c in {1,...,4}{
    \node at (\hhh*0.5+\hhh*\c+\hp*-0.5,\vvv*0.5+\vp*-3) {\small $4$};
}
\foreach \c in {1,...,4}{
    \node at (\hhh*3.5+\hhh*\c+\hp*-0.5,\vvv*-2.5+\vp*-3) {\small $3$};
}
\foreach \c in {1,...,4}{
    \node at (\hhh*-0.5+\hhh*\c+\hp*-0.5,\vvv*-2.5+\vp*-3) {\small $2$};
}
\node at (\hhh*2.5+\hp*0.5,\vp*-3+\vvv*1.2) (D2) {};
\draw (\hhh*0+\hp*0.5,\vvv*1+\vp*-3) -- (\hhh*1+\hp*0.5,\vvv*1+\vp*-3) -- (\hhh*1+\hp*0.5,\vvv*0+\vp*-3) -- (\hhh*2+\hp*0.5,\vvv*0+\vp*-3) -- (\hhh*2+\hp*0.5,\vvv*-1+\vp*-3) -- (\hhh*1+\hp*0.5,\vvv*-1+\vp*-3) -- (\hhh*1+\hp*0.5,\vvv*-2+\vp*-3) -- (\hhh*0+\hp*0.5,\vvv*-2+\vp*-3) -- (\hhh*0+\hp*0.5,\vvv*1+\vp*-3);
\node at (\hhh*0.5+\hp*0.5,\vvv*0.5+\vp*-3) {\small $1$};
\node at (\hhh*0.5+\hp*0.5,\vvv*-0.5+\vp*-3) {\small $1$};
\node at (\hhh*1.5+\hp*0.5,\vvv*-0.5+\vp*-3) {\small $1$};
\node at (\hhh*0.5+\hp*0.5,\vvv*-1.5+\vp*-3) {\small $1$};
\draw (\hhh+\hp*0.5,\vvv*0+\vp*-3) rectangle (\hhh*5+\hp*0.5,\vvv*1+\vp*-3);
\draw (\hhh*0+\hp*0.5,\vvv*-3+\vp*-3) rectangle (\hhh*4+\hp*0.5,\vvv*-2+\vp*-3);
\draw (\hhh*4+\hp*0.5,\vvv*-3+\vp*-3) rectangle (\hhh*8+\hp*0.5,\vvv*-2+\vp*-3);
\foreach \c in {1,...,4}{
    \node at (\hhh*0.5+\hhh*\c+\hp*0.5,\vvv*0.5+\vp*-3) {\small $3$};
}
\foreach \c in {1,...,4}{
    \node at (\hhh*3.5+\hhh*\c+\hp*0.5,\vvv*-2.5+\vp*-3) {\small $4$};
}
\foreach \c in {1,...,4}{
    \node at (\hhh*-0.5+\hhh*\c+\hp*0.5,\vvv*-2.5+\vp*-3) {\small $2$};
}
\node at (\hhh*2.5+\hp*1.5,\vp*-3+\vvv*1.2) (D3) {};
\draw (\hhh*0+\hp*1.5,\vvv*1+\vp*-3) -- (\hhh*1+\hp*1.5,\vvv*1+\vp*-3) -- (\hhh*1+\hp*1.5,\vvv*0+\vp*-3) -- (\hhh*2+\hp*1.5,\vvv*0+\vp*-3) -- (\hhh*2+\hp*1.5,\vvv*-1+\vp*-3) -- (\hhh*1+\hp*1.5,\vvv*-1+\vp*-3) -- (\hhh*1+\hp*1.5,\vvv*-2+\vp*-3) -- (\hhh*0+\hp*1.5,\vvv*-2+\vp*-3) -- (\hhh*0+\hp*1.5,\vvv*1+\vp*-3);
\node at (\hhh*0.5+\hp*1.5,\vvv*0.5+\vp*-3) {\small $1$};
\node at (\hhh*0.5+\hp*1.5,\vvv*-0.5+\vp*-3) {\small $1$};
\node at (\hhh*1.5+\hp*1.5,\vvv*-0.5+\vp*-3) {\small $1$};
\node at (\hhh*0.5+\hp*1.5,\vvv*-1.5+\vp*-3) {\small $1$};
\draw (\hhh+\hp*1.5,\vvv*0+\vp*-3) rectangle (\hhh*5+\hp*1.5,\vvv*1+\vp*-3);
\draw (\hhh*0+\hp*1.5,\vvv*-3+\vp*-3) rectangle (\hhh*4+\hp*1.5,\vvv*-2+\vp*-3);
\draw (\hhh*4+\hp*1.5,\vvv*-3+\vp*-3) rectangle (\hhh*8+\hp*1.5,\vvv*-2+\vp*-3);
\foreach \c in {1,...,4}{
    \node at (\hhh*0.5+\hhh*\c+\hp*1.5,\vvv*0.5+\vp*-3) {\small $2$};
}
\foreach \c in {1,...,4}{
    \node at (\hhh*3.5+\hhh*\c+\hp*1.5,\vvv*-2.5+\vp*-3) {\small $4$};
}
\foreach \c in {1,...,4}{
    \node at (\hhh*-0.5+\hhh*\c+\hp*1.5,\vvv*-2.5+\vp*-3) {\small $3$};
}
\draw[->] (A1) -- (B1);
\draw[->] (A2) -- (B2);
\draw[->]  (\hhh*-1.1+\hp*1.5,\vvv*-3.6) -- (B3);
\draw[->] (\hhh*2.5+\hp*-0.5,\vp*-1-\vvv*3.4) -- (C1);
\draw[->] (\hhh*2.5+\hp*0.5,\vp*-1-\vvv*3.4) -- (C2);
\draw[->] (\hhh*2.5+\hp*1.5,\vp*-1-\vvv*3.4) -- (C3);
\draw[->] (\hhh*2.5+\hp*-0.5,\vp*-2-\vvv*3.4) -- (D1);
\draw[->] (\hhh*2.5+\hp*0.5,\vp*-2-\vvv*3.4) -- (D2);
\draw[->] (\hhh*2.5+\hp*1.5,\vp*-2-\vvv*3.4) -- (D3);
\end{tikzpicture}
\]
One sees that 
$\epsilon = 1$ for all $T \in \BSTx{\DIF{(5,2,1,8)}}{(4^4)}$.
\end{example}

\hspace*{5mm}

\noindent {\bf Acknowledgments.}
The authors are grateful to the anonymous referees for their careful readings of the manuscript and valuable advice.

\end{document}